\documentclass[onefignum,onetabnum]{siamart171218}

\usepackage{amsfonts}
\usepackage{graphicx}
\usepackage{caption}
\usepackage{subcaption}
\usepackage{mathtools}
\usepackage{amssymb}
\usepackage{multirow}
\usepackage{amsmath}
\usepackage{algorithm}
\usepackage{algpseudocode}
\usepackage{hyperref}
\usepackage{array}
\usepackage{mathabx}
\usepackage{bm}
\usepackage[noadjust]{cite} 
\usepackage{comment}
\usepackage{amsopn}
\usepackage{enumitem}
\usepackage{placeins}
\usepackage{float}
\usepackage{booktabs}
\usepackage{makecell}
\usepackage{array}

\ifpdf
\DeclareGraphicsExtensions{.eps,.pdf,.eps,.jpg}
\else
\DeclareGraphicsExtensions{.eps}
\fi

\newtheorem{Remark}[theorem]{Remark}

\newcommand{\R}{\mathbb{R}}

\newcommand{\normF}[1]{\left\|#1\right\|_F}

\newcommand{\DA}{\Delta\!A}
\newcommand{\DB}{\Delta\!B}
\newcommand{\rtilde}{\tilde{r}}

\headers{Efficient Approximations of Matrix Multiplication}{S. Kar, Hariprasad M., Sai Gowri J. N., and M. Venkatapathi}

\title{Efficient approximations of matrix multiplication using truncated decompositions}

\author{
Suvendu Kar
\and
Hariprasad M.\thanks{School of Mathematics and Natural Sciences, Chanakya University, Devanahalli, Bangalore }\email{(hariprasad.h@chanakyauniversity.edu.in)}
\and
Sai Gowri J. N.
\and
Murugesan Venkatapathi\thanks{Department of Computational \& Data Sciences, Indian Institute of Science, Bangalore }\email{(murugesh@iisc.ac.in)}
}

\makeatletter
\newcommand*{\addFileDependency}[1]{
	\typeout{(#1)}
	\@addtofilelist{#1}
	\IfFileExists{#1}{}{\typeout{No file #1.}}
}
\makeatother

\DeclarePairedDelimiter\norm{\lVert}{\rVert}

\begin{document}

\maketitle
	
\begin{abstract}
We exploit the truncated singular value decomposition and the recently proposed circulant decomposition for an efficient first-order approximation of the multiplication of large dense matrices. A decomposition of each matrix into a sum of a sparse matrix with relatively few dominant entries and a dense residue can also use the above approach, and we present methods for multiplication using a Fourier decomposition and a cycle decomposition-based sparsifications. The proposed methods scale as $\mathcal{O}(n^2 \log n)$ in arithmetic operations for $n \times n$ matrices for usable tolerances in relative error $\sim$ 1\%. We also present demonstrations of large gains in the efficiency and speed of end-to-end operations of Large Language Models (LLMs) as a motivation. Note that different decompositions for the two matrices $A$ and $B$ in the product $AB$ are also possible in this approach, using efficient a priori evaluations for suitability, to improve further on the error tolerances demonstrated here.
\end{abstract}
	
	\begin{keywords}
		matrix approximations, singular values, circulant matrix, sparsification, Fast Fourier Transform.
	\end{keywords}
	
	\begin{AMS}
		65F55, 65F35, 65F25.
	\end{AMS}

\section{Introduction}
Reducing the scaling of arithmetic operations required in multiplying two (dense) matrices can widely increase the efficiency of computing. Even though the asymptotic scaling of the Strassen-Coppersmith-Winograd type algorithms \cite{3,4,5,6,11, alman2024asymmetryyieldsfastermatrix} as $\mathcal{O}(n^{2.37})$ is attractive compared to the naive approach using $\mathcal{O}(n^3)$ arithmetic operations for an exact multiplication, they have been non-trivial in practical implementation \cite{7,8,9}. As an alternative, randomized algorithms for approximation have been proposed to incur lower computing effort at the cost of lower precision in the evaluated matrix product \cite{matrixmultiply2006Kannan,12}. Nevertheless, these methods are not sufficiently precise or efficient to have a wide impact on practical computing. We propose approximations based on relatively few components of a decomposition that each multiply a general matrix efficiently, thus allowing a first-order approximation of the product of two truncated decompositions.

One may also note that in the more recent applications such as the domain of deep learning, multiplication of large matrices is unavoidable. For example, the forward pass of a convolutional neural network processes data in batches and it may perform evaluations of the form $Z = WX + B$, with $W$ the weights of the connections. Back-propagation computes the gradient of the loss functions using the chain rule which may again involve the multiplication of large matrices \cite{ciresan2011flexible,chellapilla2006high}. Computation of self-attention in transformers uses the product of query, key, and value matrices \cite{vaswani2017attention}, and such multiplication of large matrices is not restricted to the above cases.

We begin with the used notations in \cref{sec:notations}, and the relevant decompositions of matrices in \cref{sec:Decompositions}. The subsequent \cref{sec:methods,sec:analysis,sec:results} detail the methods and algorithms, a theoretical analysis of errors, and the numerical experiments, respectively.

\subsection{Notations used}\label{sec:notations}
All vectors represent a column, where unit vectors have hats with the 2-norm $\norm{\hat{v}}_2=1$, and only the other vectors $\bm v$ are in bold. Matrices $V$ are in capitals. Its columns are indicated by superscripts with $V^{(i)}$ for the $i^{th}$ column and rows by subscripts with $V_{(i)}$ for the $i^{th}$ row. Plain subscripts in $V_i$ represent a particular matrix component. Row vectors $\bm v^T$ or $\bm v^*$ have a transpose or a conjugate transpose of the columns indicated respectively.

\subsection{Relevant decompositions}\label{sec:Decompositions}

\subsubsection{Singular Value Decomposition}
Any $m \times n$ matrix $A$ is given by a sum of weighted rank-one matrices, where its rank $r \leq min(m,n)$:
\begin{align*}
    A &= \sigma_1 \hat{u}_1 \hat{v}_1^* + \sigma_2 \hat{u}_2 \hat{v}_2^* \dots + \sigma_{r} \hat{u}_{r} \hat{v}_{r}^* = \sum_{k = 1}^{r} \sigma_k \hat{u}_k \hat{v}_k^* \\
    &= U \Sigma V^*  \quad \text{ where } \hat{u}_k \in \mathbb{C}^m, \hat{v}_k \in \mathbb{C}^n \text{ , and } \Sigma_{kk} = \sigma_k \in \mathbb{R}_{\geq 0},
\end{align*}
where $\sigma_1 > \sigma_2 > \dots > \sigma_r > 0$. Given $\hat{u}, \hat{v}$ are each sets of mutually orthogonal vectors, we have:
\begin{equation*}
  A\hat{v}_k = \sigma_k \hat{u}_k, \text{ and } A^*\hat{u}_k = \sigma_k \hat{v}_k.  
\end{equation*}   

Though this decomposition requires $\mathcal{O}(n^3)$ for $n \times n$ matrices, we can use randomized algorithms to evaluate only the $k$ largest components in $\mathcal{O}(kn^2)$ arithmetic operations as shown in $\cref{algo:rsvd}$.

\subsubsection{Decomposition into cycles}\label{cycle-decomp}
A matrix can also be decomposed into its $n$ cycles each containing $n$ entries, which in turn are combinations of the $2n-1$ diagonals of a matrix. Let $C$ be the permutation matrix corresponding to a full cycle.
\begin{align}
	C = 
	\begin{bmatrix}
	0 & 1 \\
	I_{n-1} & 0 
	\end{bmatrix}_{ n \times n}.
        \label{exp: of C}
        \end{align}

\vspace{0.2cm}
The decomposition of any matrix into cycles is given by $A = \sum_{j=0}^{n-1} \Lambda_jC^j$, where $\Lambda_j$ are diagonal matrices containing $n$ entries of a cycle of the matrix, and $C^j$ is the above permutation matrix $C$ raised to the power $j$. While it is equally possible to have $A = \sum_{j=0}^{n-1} C^j \Lambda_j$, the ordering of the diagonal elements of $\Lambda_j$ differ in both cases. This decomposition requires $\mathcal{O}(n^2)$ arithmetic operations for its evaluation (see \cref{appendix:example_finding_cycles} for an example). Left multiplication of a cycle to a matrix involves a cyclic permutation of the columns of the matrix followed by a scaling of the rows using the diagonal entries in the cycle, costing $\sim 2n^2$ operations. The converse applies to the right multiplication of a cycle. This decomposition may be useful in approximating matrices with dominant cycles numbering $k \ll n$.

\subsubsection{Circulant Decomposition}

\begin{theorem}\label{th:circulant_decomposition}
Any $n \times n$ matrix A is given by a sum of $n$ circulant matrices $R_k$ with phase factors given by roots of unity.
\begin{align*}
    A &= R_{0} + R_{1}D^1 \dots + R_{n-1}D^{n-1} \\
    &= \sum_{k = 0}^{n-1} R_{k}D^k \quad \text{where } D(q,q) = \omega^q = e^{i\frac{2 \pi}{n} q} \text{, and }
    D(p,q) =0 \text{ for }p \neq q.   
\end{align*}

\end{theorem}

\begin{proof}
    Using the cycle decomposition described in section \ref{cycle-decomp},
    \begin{align}
        A = \sum \limits_{j=0}^{n-1} C^j \Lambda_j. \label{rightdecomp}
    \end{align}
    Then using a Fourier series expansion of a cycle $j$,
   \begin{align}
        C^j \Lambda_j = C^j\sum_{k}  r_{j,k} D^k.
   \end{align}
where $r_{j,k}=\frac{1}{n} \sum_l \omega^{-kl}\Lambda_j(l,l)$. Thus we get: 
   \begin{align*}
    A &= \sum_{j,k} C^j r_{j,k} D^k, \\
    & = \sum_{k} \left(\sum_{j} r_{j,k} C^j \right) D^k = \sum_{k} R_k D^k
   \end{align*}
where the circulant matrix $R_k$ for a given $k$, has the constants $r_{j,k}$ in all the $n$ entries of a cycle $j$. This decomposition requires only $\mathcal{O}(n^2 \log n)$ arithmetic operations for its evaluation using a Fast Fourier Transform. 
\end{proof}

\begin{Remark}
We can show that $R_kD^k$ are mutually orthogonal w.r.t. a Frobenius inner product $\odot$, where $A \odot A = \sum_{i,j} \overline{A(i,j)}A(i,j) = \norm{A}_F^2$. We know the property of invariance of the Frobenius inner product with unitary transformations\cite{Hari2024circulant}. Thus, $(WAW^*)\odot(WBW^*) = A \odot B$, where the unitary matrix $W(p,q) = \frac{1}{\sqrt{n}}e^{-i \frac{2\pi pq}{n}}$, with $p,q=0,1,2\cdots n-1$. Further, 
\begin{align*}
    R_iD^i \odot R_jD^j &= (WR_iD^iW^*) \odot (WR_jD^jW^*) = (WR_iW^*WD^iW^*) \odot (WR_jW^*WD^jW^*)\\
    &=(\Lambda_i'C^i) \odot (\Lambda_j'C^j) = 0 \text{ for } i \neq j\\
    &= \normF{R_i}^2 \text{ for } i=j.
\end{align*}
where we have used known relations $WD^jW^*=C^j$, and the diagonalization of a circulant matrix in the Fourier domain as $WR_jW^*=\Lambda_j'$. This property of orthogonality along with the decomposition in \cref{th:circulant_decomposition} gives us $A \odot R_kD^k = \norm{R_k}_F^2$. 
\end{Remark}

\subsubsection{Fourier series based decompositions}
We can also expand every row of the matrix $A$ using Fourier series, giving us the following decomposition.

\begin{theorem}
Every matrix $A$ $\in \mathcal{C}^{m \times n}$ is a weighted sum of rank-one Fourier components.
\begin{align*}
    A = \sum_{k=0}^{n-1} a_k \hat{f}_k \hat{\omega}_k^T \text{, and } A \odot A = \norm{A}_F^2 = \sum_{k=0}^{n-1} a_k^2.
\end{align*}
where $\norm{\hat{f}_k}_2 = \norm{\hat{\omega}_k}_2=1$, and $a_k$ $\in$ $\mathbb{R}_{\geq 0}$.
\end{theorem}

\begin{proof}
Consider the Fourier series representation of a row of $A$ written as :
    \begin{align*}
        A_{(i)} &= \sum_{k} f_{i,k} \mathbf{1}^T D^k.  
    \end{align*}
    
    Then, using the canonical basis vectors $\hat{e}_i$ we have, 
    \begin{align*}
        A &= \sum_{i} \hat{e}_i A_{(i)} = \sum_{i} \hat{e}_i \sum_{k} f_{i,k} \mathbf{1}^T D^k\\
        &= \sum_{k} \sum_{i} \hat{e}_i f_{i,k} \mathbf{1}^T D^k
        =  \sum_{k=0}^{n-1} \bm f_k \mathbf{1}^T D^k.
    \end{align*}
    Let $a_k = \sqrt{n}\norm{\bm f_k}_2$, $\hat{f}_k=\frac{\bm f_k}{\norm{\bm f_k}_2}$, and $\hat{\omega}_k^T=\frac{1}{\sqrt{n}} \mathbf{1}^T D^k$. Thus, $A = \sum_{k=0}^{n-1} a_k \hat{f}_k \hat{\omega}_k^T.$

Given $D^i \odot D^j = \sum_q e^{(j-i)\frac{2q \pi}{n}}$ with $q=0, 1, 2 \dots n-1$, we have $D^i \odot D^j =0$ for $i \neq j$, and $D^i \odot D^j = n\delta_{ij}$. Hence, $A \odot A = \norm{A}_F^2 = \sum_{k=0}^{n-1} a_k^2$.

\end{proof}

Note that a column-wise Fourier series representation of $A$ such that we have a rank one decomposition into $A = \sum_{k=0}^{m-1} b_k \hat{\omega}_k \hat{f}_k^T$, with $\norm{A}_F^2 = \sum_{k=0}^{m-1} b_k^2$ is also possible. This decomposition requires $\mathcal{O}(n^2 \log n)$ arithmetic operations for its evaluation using a Fast Fourier Transform (FFT).

\section{Methods} \label{sec:methods}

\subsection{Truncated decompositions and a first-order approximation}
In the proposed approach, we first partially decompose a given matrix $A$ into a sum of $k$ dominant components, and a residue representing the remaining $n-k$ components i.e. $A=\sum_{i=0}^{k-1} A_i + \Delta A$.  Note that the multiplication of a particular component $A_i$ to any given matrix may scale as $\mathcal{O}(n^2 \log n)$ arithmetic operations at most depending on the decomposition used. We may not only multiply the dominant components of the partial decompositions $A_i$ and $B_k$ with each other (as a zeroth order approximation), but also the residue matrices with those multiplication-friendly components i.e. $A_i \cdot \Delta B$ and $\Delta A \cdot B_k$ resulting in a first order approximation as given below.
\begin{align}
    A \cdot B  &= \left\{\sum_i A_i + \Delta A\right\} \cdot \left\{\sum_k B_k + \Delta B\right\} \\
    &= \sum_i A_i \cdot B+ \sum_k \Delta A \cdot B_k+\textcolor{red}{(\Delta A \cdot \Delta B)} \\
    &\approx M = \sum_i A_i \cdot B+ \sum_k \Delta A \cdot B_k.
\label{eq:first_order_expansion}
\end{align}

Also, the first-order approximation may increase the rank of resulting matrices in many cases, compared to the zeroth order multiplication, especially when using SVD components where the ranks of $A_i$ and $B_i$ are one. But, the performance of this approach in reducing errors in the approximate multiplication is more significant. The relative Root-Mean-Square (R.M.S) error in entries of the product is given by $\frac{\norm{\Delta A \cdot \Delta B}_F}{\norm{AB}_F}$. In \cref{sec:analysis} it is shown that this relative error in the approximation can be smaller than the product of the relative errors in the two matrices i.e. $\frac{\norm{\Delta A}_F}{\norm{A}_F}$ and $\frac{\norm{\Delta B}_F}{\norm{B}_F}$. The relative error can also be estimated in $\mathcal{O}(n^2)$ arithmetic operations as summarized later in \cref{sec:summary_estimation}. Hence a low truncation error in the decomposition of one of the two matrices would suffice, even though it is not necessary, for a low error in the approximated product. Using the presented theoretical results, it is also possible to evaluate (and plot) the error bounds for the number of components included in the approximation, given the spectra of singular values or magnitudes of circulant components.

\subsection{Based on Singular Value Decomposition}
Given $\tilde{A} = U \Sigma_A V^T$ and $\tilde{B} = W \Sigma_B Y^T$ as the truncated decompositions (using $\cref{algo:rsvd}$ for real $A$, $B$ respectively), we can approximate their product using the $k$ largest singular value components as :
\begin{align*}
    M \approx \sum\limits_{i=1}^{k} \sigma_i^A u_i v_i^T B +\sum_{i=1}^{k} \sigma_{i}^B \Delta Aw_{i}y_{i}^{T} = \tilde{A}B + \Delta A \tilde{B}.
\end{align*}

Note that multiplication of \emph{each} SVD component with any other matrix is $\sim 2n^2$ arithmetic operations, and the two truncated matrix multiplications in the above approximation require $\sim 4kn^2$ operations using $k$ SVD components of $A$ and $B$. The relative Frobenius error in approximating the matrix using a partial SVD of $A$ is $\frac{\norm{\Delta A}_F}{\norm{A}_F}=\frac{(\norm{A}_F^2-\sum_{i=1}^{k} \sigma_i^2)^{\frac{1}{2}}}{\norm{A}_F}$, which can be evaluated in $\sim n^2$ arithmetic operations.
 
\begin{algorithm}
\caption{Randomized partial SVD \cite{halko2010findingstructurerandomnessprobabilistic}}\label{algo:rsvd}
\begin{algorithmic}[1]
\State \textbf{Input:} Matrix \( A \) of size \( m \times n \), an integer \( s \)
\State \( m, n \gets \text{size}(A) \)
\State \( k \gets \min(s \times \lfloor \log_2(n) \rfloor + 1, n) \)
\State \( \phi \gets \text{random matrix of size } n \times k \)
\State \( Y \gets A \phi \)
\State \( Q, R \gets \text{QR decomposition of } Y \)
\State \( B \gets Q^TA \)
\State \( U, \Sigma, V^T \gets \text{SVD of } B \)
\State \( U \gets Q U \)
\State \(V \gets (V^{T})^{T} \)
\State \textbf{Output:} Matrices \( U ( \in \mathbb{R}^{m\times k}), \Sigma (\in \mathbb{R}^{k\times k}), V(\in \mathbb{R}^{n\times k}) \)
\end{algorithmic}
\end{algorithm}

\subsection{Based on Circulant Decomposition}

Let $A = \sum_{k = 0}^{n-1}R_kD^k$. Any $n \times n$ matrix can be approximated as a partial sum of the above. The relative weight of circulant component $R_k$ in $A$ is $w_k = \frac{\|R_k\|_F^2}{\norm{A}_F^2}$ where $\norm{A}_F^2=\sum_k \|R_k\|_F^2$. The distribution of magnitudes of the circulant components depends on any periodic structure of the entries in the matrix. This distribution for various types of matrices are shown in \cref{sec:plots_CD}. Using some values of $k$ where the above weights are the largest, indicated by $k'$, we can decompose A into $\sum_{k'} R_{k'} D^{k'} + \Delta A$. Here, the relative error in the approximation of the matrix is given by $\frac{\norm{\Delta A}_F}{\norm{A}_F}$. Note that in the case of rectangular matrices, one may have to evaluate these approximations using square blocks of the matrix.

\subsubsection{Evaluating circulant components $R_k$}

If a single circulant component for a given $k$ is to be evaluated, it costs $\mathcal{O}(n^2)$ arithmetic operations.

\begin{theorem}\label{th:fourier_circulant_entries}
When the unique entries of $R_k$ i.e. $R_{k}(j,0)$, where $j=0,1 \dots n-1$, are evaluated using a discrete Fourier transform of the corresponding cycles $j$ of $A$, it is equivalent to a minimization of $\norm{A-R_kD^k}_F$.
\end{theorem}

\begin{proof}
Since the circulant decomposition is orthogonal w.r.t. a Frobenius inner product, let us minimize $\norm{A-R_kD^{k}}_F$ to explicitly evaluate the entries of $R_k$.
\begin{equation*}\label{eq:norm_minimize}
   R_k = \arg \min_{M \in R} f(M) = \arg \min_{M \in R} \norm{A-MD^k}_F.    
\end{equation*}
Decomposing into cycles, let $A = \sum_{j=0}^{n-1} C^j \Lambda_j$, and we also know circulant matrices have constant entries in each cycle with $M = \sum_{j=0}^{n-1} a_jC^j$ and up to $n$ unique constants $a_j$. Decomposing the above minimization into a minimization over each cycle to get the corresponding unique entries of $R_k$ in its first column:
\begin{equation*}
\min_{M \in R} f(M) \implies R_k(j,0) = \arg \min_{a_j} \norm{\Lambda_j-a_jD^k}_F \,\forall\, j.
\end{equation*}
Thus the constants $a_j$ in the $j^{th}$ cycle of $R_k$ are given by the first moment i.e. the average value of entries in $\Lambda_jD^{-k}$ as:
\begin{equation*}
   R_{k}(j,0) = \frac{1}{n} \left( \mathbf{1}^T (A \circ C^{j})D^{-k} \mathbf{1} \right) \text{for } j = 0,1, \cdots n-1.
\end{equation*}
where $\mathbf{1}$ is a $n$-vector with entries all ones, thus representing a $k^{\text{th}}$ Fourier component in the discrete Fourier transform of the $j^{\text{th}}$ cycle of matrix $A$. 
\end{proof}

\vspace{2mm}
Let the unitary matrix $W(p,q) = \frac{1}{\sqrt{n}}e^{-i \frac{2\pi pq}{n}}$, with $p,q=0,1,2\cdots n-1$. For a vector $\bm x$, let $\bm y = W \bm x$ be the FFT of $\bm x$, and $\bm x = W^* \bm y$ be the corresponding IFFT transform. Recall from \cref{th:circulant_decomposition} and \cref{th:fourier_circulant_entries} that the Discrete Fourier Transform of the $j^{th}$ cycle of a matrix can provide us its circulant components $r_{j,k}$. Thus arranging the cycles of matrix $A$ into columns of $\underline{A}$ and applying Fast Fourier Transform (FFT) operations on its columns i.e. $\underline{B}=W\underline{A}$ where $\underline{B}/\sqrt{n}$ provides us all the circulant components of the matrix in only $\mathcal{O}(n^2 \log n)$ arithmetic operations. Note that :
\begin{align}
  \frac{1}{\sqrt{n}}\underline{B}(k,j)=\frac{1}{\sqrt{n}}\sum_l W(k,l)\underline{A}(l,j) = \frac{1}{n} \sum_l \omega^{-kl}\Lambda_j(l,l)=r_{j,k}.
\end{align}
where $r_{j,k}$ is the constant in all the $n$ entries of the $j^{th}$ cycle of the circulant matrix $R_k$, and the above is described in \cref{algo:circulant_decomposition_fft}. One can also use the inverse transformation of a cycle of the 2D-FFT of $A$ given by $WAW^*$, to evaluate a corresponding circulant component i.e. $R_kD^k=W^*(WAW^* \circ C^k)W$. But this evaluation costs twice as the number of arithmetic operations as the above. The similarity transformation between the cycles of the 2D-FFT of $A$ and the circulant components of $A$, was shown elsewhere \cite{Hari2024circulant}.

\begin{algorithm}
\caption{Circulant Decomposition of a Matrix using FFT \cite{Hari2024circulant}} \label{algo:circulant_decomposition_fft}
\begin{algorithmic}
\State \textbf{Input:} Matrix \( A \in \mathbb{R}^{n \times n} \)
\State Initialize empty list R
\State Step 1: Construct $\underline{A}$ as $\underline{A}(:,j) = diag(\Lambda_j)$ with $\Lambda_j$ as in equation \eqref{rightdecomp}
\State Step 2: Compute $\underline{B} = W \underline{A}$
\State 
\For{each \( k \) from 0 to \( n-1 \)} \\
     $R(k) =\frac{1}{\sqrt{n}}transpose(\underline{B}(k,:))$  
\EndFor
\Comment{The first column of corresponding circulant component $R_k$ of $A$.} 
\end{algorithmic}
\end{algorithm}

\subsubsection{Multiplication using circulant components}\label{optimal_procedure_to_compute_step_6}
 
\begin{algorithm}
\caption{First-order Approximation of Matrix Multiplication via Circulant Decomposition} \label{algo:cd_matrix_multiplication}
\begin{algorithmic}
\State \textbf{Input:} Two \( n \times n \) matrices \( A \) and \( B \), integer k.
\State Step 1: Find circulant components of \( A \) and \( B \):
\State \hspace{10pt} \( ( R_A) \gets \text{using \ref{algo:circulant_decomposition_fft} for A} \)
\State \hspace{10pt} \( ( R_B) \gets \text{using \ref{algo:circulant_decomposition_fft} for B} \)
\State Step 2: Compute magnitudes of circulant columns:
\State \hspace{10pt} \( T_A[i] = \| R_A[i] \|_2 \quad \text{for} \quad i = 0, \dots, n-1 \)
\State \hspace{10pt} \( T_B[i] = \| R_B[i] \|_2 \quad \text{for} \quad i = 0, \dots, n-1 \)
\State Step 3: Select \( k \) indices based on top $k$ values of $T$:
\State \hspace{10pt} \( In_A \gets \text{k indices of top values in } T_A \)
\State \hspace{10pt} \( In_B \gets \text{k indices of top values in } T_B \)
\State Step 4: Create empty lists $L_A, L_B$ of length k each. For those \(k\) indices in $In_A, In_B$:
\State \hspace{10pt} \( L_A[i]=\sqrt{n}WR_A[In_{A}[i]]  \quad \text{for } \quad i=0,1,\cdots,k-1 \)
\State \hspace{10pt} \( L_B[i]=\sqrt{n}WR_B[In_{B}[i]]  \quad \text{for } \quad i=0,1,\cdots,k-1 \)
\State Step 5: Compute $\sum_{k} A_{k}B$
\State \hspace{10pt} \(M \gets \text{IFFT}\left (\sum_{t =0}^{k-1}diag(L_A[t]) C^{In_A[t]} \text{FFT}(B) \right) \)
\State Step 6 : Compute $\sum_{j} \Delta AB_{j}$ and update M
\State \hspace{10pt} \(\Delta A \gets A - \sum_{t \in In_A} (R_{A_t}D^{t}) \) \Comment{$R_{A_t}$ is the circulant matrix having first column as $R_A[t]$}
\State \hspace{10pt} \(M \gets M + \left(\text{IFFT}\left( \sum_{t =0}^{k-1} (C^{{{In}_B}[t]})^*diag(\overline{L_B[t]})\text{FFT}(\Delta A^*)\right)\right)^*\)
\State \textbf{Output} : \( M \) \Comment{Return the final approximation.}
\end{algorithmic}
\end{algorithm}

After an evaluation of circulant components $A$ and $B$ using \cref{algo:circulant_decomposition_fft}, the \cref{algo:cd_matrix_multiplication} completes the first-order multiplication evaluations i.e. $AB \approx M = \sum_{k} A_{k}B+ \sum_{j} \Delta AB_{j}$. Here, only the selected indices $j$ and $k$ based on the magnitude of circulant components of $A$ and $B$ are used. There are two salient features of this evaluation in the above algorithm explained here. While multiplying circulant components ($R_kD^k$) and a matrix on the right using FFT operations, a naive approach results in a notable increase of the computing effort required for this method. In the optimal evaluation of circulant matrix-vector products, recall that one evaluates a Hadamard product of the FFT of the first column of a circulant matrix and the FFT of given column vector. An inverse FFT of this resulting vector provides us the circulant matrix-vector product in $\sim 2n \log n$ arithmetic operations. Also, in an optimal evaluation of $R_kD^kB$, evaluating the FFT of the columns of $D^kB$ for each value of $k$ can be rewritten as a cyclic permutation of the FFT of the columns of $B$ i.e. $WD^kB = C^kWB$, avoiding multiple FFT operations for different values of $k$ in this evaluation. This evaluation using the first column $\bm l_k$ of $R_k$ for any given column $\bm b$ in $B$ is then given by: 
\begin{equation*}
 W\bm l_k \circ W(D^k\bm b) = W \bm l_k \circ WD^kW^{*}W\bm b =  W\bm l_k \circ C^kW\bm b  
\end{equation*}
where we have used the relations $W^*W=WW^*=I$ and $W^*D^kW=C^k$, and `$\circ$' is a Hadamard product.

Secondly, the inverse transform can be applied only once on the sum of all of the above products in the transformed domain i.e. $W^*(\sum_k WR_kD^kB)$. This obviates the need for inverse transformation of each term separately before the summation. The above procedure of efficiently multiplying a matrix on the right of a circulant component can also be used to evaluate the products $B_j^*\Delta A^*$ as well, and apply a conjugate transpose to provide us $\Delta A B_j$. Both the properties described above allow us to reduce the scaling of the overall computing effort from $\mathcal{O}(kn^2 \log n)$ to $\mathcal{O}(kn^2)$ when $k > \log n$ is the number of circulant components used in the multiplication of matrices, as given by \cref{algo:cd_matrix_multiplication}. For a given $k < \log n$ the scaling of the algorithm remains $\mathcal{O}(n^2 \log n)$ as the one-time application of FFT on $A$ and $B$ becomes the leading term in the scaling of arithmetic operations.

\subsection{Based on Fourier-domain sparsification} \label{sec:Fourier_sparsification}
Recalling the first order approximation proposed using sparse matrix components $A_s$ and $B_s$:
\begin{align}
    A \cdot B  &= \{A_s + \Delta A\} \cdot \{B_s + \Delta B\} = A_s \cdot B + \Delta A \cdot B_s + \textcolor{red}{(\Delta A \cdot \Delta B)} \\
    &\approx M = A_s \cdot B + \Delta A \cdot B_s.
    \label{eq:sparse_multiply}
\end{align}
This formulation applies a Fourier transform to both matrices $A$ and $B$ such that:
\begin{equation*}
     A \cdot B = F^*(A) \cdot F(B),   
\end{equation*}
where $F^*(A) := AW^*$, $F(B) := WB$ respectively. Here again, the unitary matrix is $W(p,q) = \frac{1}{\sqrt{n}}e^{-i \frac{2\pi pq}{n}}$, with $p,q=0,1,2\cdots n-1$ and $n$ represents the common inner dimension when two rectangular matrices are being multiplied. One can sparsify the matrix $F^*(A)$ by selecting the top $k$ dominant frequencies for each row, with $k$ potentially varying for each row. Let this sparse matrix be represented as $\Tilde{F}^*(A)$. Similarly, let $\Tilde{F}(B)$ denote the sparsified version of $F(B)$ using only $k$ dominant frequencies in the columns. Using the above in the matrix product \eqref{eq:sparse_multiply}, a first order approximation $M$ is evaluated as follows:
\begin{equation*}
 M = \Tilde{F}^*(A) \cdot F(B) + [F^*(A) - \Tilde{F}^*(A)] \cdot \Tilde{F}(B)   
\end{equation*}
The relative errors defined earlier apply here as well, and the below \cref{alg:mul_topk} describing its implementation may require exploitation of sparse data structures.

\begin{algorithm}
\caption{Top-$k$ Matrix Approximation by Row} \label{alg:topk}
\begin{algorithmic}[1]
\Require $A \in \mathbb{C}^{m \times n}$ (input matrix), $k$ (sparsity parameter), $n$ (number of columns)
\Ensure $A_{\text{approx}} \in \mathbb{C}^{m \times n}$ (matrix with top-$k$ entries per row)
\State $k=min(k,n)$ \Comment{Ensure $k \leq n$}
\State Initialize $A_{\text{approx}} \gets 0^{m \times n}$ \Comment{Zero matrix}
\For{$i \gets 1$ to $m$}
    \State $\mathbf{a} \gets |A(i,:)|$ \Comment{Absolute values of $i$-th row}
    \State $ I \gets \text{argsort}(\mathbf{a})$ \Comment{Indices of increasing values of \( a \)}
    \State $\text{topk\_indices} \gets I(n-k+1:n)$ reversed \Comment{Indices of $k$ largest entries}
    \State $A_{\text{approx}}(i, \text{topk\_indices}) \gets A(i, \text{topk\_indices})$ \Comment{Copy original values}
\EndFor
\end{algorithmic}
\end{algorithm}

\begin{algorithm}[H]
\caption{First-order approximation of matrix multiplication based on sparsification} \label{alg:mul_topk}
\begin{algorithmic}[1]
\Require $A \in \mathbb{C}^{m \times n}$, $B \in \mathbb{C}^{n \times p}$ (input matrices), 
\Ensure $M \in \mathbb{C}^{m \times p}$ (Approximation of matrix multiplication AB)

\State $\tilde{A} \gets AW^*$ \Comment{$\mathcal{F}^*(A)$}
\State $\tilde{B} \gets WB$ \Comment{$\mathcal{F}(B)$}
\State $\tilde{A}_{\text{approx}} \gets \text{Algorithm} \ref{alg:topk}(\tilde{A},k,n)$ 
\State $\tilde{B}_{\text{approx}}\gets \text{Algorithm} \ref{alg:topk}(\tilde{B},k,p)$
\State $M \gets \tilde{A}_{approx}\tilde{B}$
\State $M \gets M + (\tilde{A}-\tilde{A}_{\text{approx}})\tilde{B}_{\text{approx}}$
\State Return : $M$

\end{algorithmic}
\end{algorithm}

\subsection{Other theoretical aspects in practical implementation:} \label{sec:summary_estimation}
The following points of consideration are highlighted for practical implementation and optimal computation, given a required tolerance in relative error.
\begin{itemize}
    \item As mentioned earlier, we can infer the suitability of a decomposition using values of $\frac{\norm{\Delta A}_F} {\norm{A}_F}$ and $\frac{\norm{\Delta B}_F} {\norm{B}_F}$, for a  given number of components of the particular decomposition. The circulant decomposition and the Fourier decomposition used for sparsification can be evaluated in full in $\mathcal{O}(n^2 \log n)$, and they are negligible in computational cost relative to the naive exact multiplication of two matrices. But, only the top-$k$ singular value components are evaluated using the randomized partial SVD for some $k \ll n$ in $\mathcal{O}(kn^2)$ arithmetic operations. Note that two different decompositions for $A$ and $B$ are admissible in the proposed approach.
    \item Even before the approximation of the product using truncated decompositions of the matrices at hand, we can estimate the relative error $\frac{\norm{\Delta A \cdot \Delta B}_F} {\norm{AB}_F}$ using $\mathbb{E}[\norm{AB}_F] \approx c_1 \norm{A}_F \norm{B}_F$, and $\mathbb{E}[\norm{\Delta A.\Delta B}_F] \approx c_2 \norm{\Delta A}_F \norm{\Delta B}_F$ with appropriate constants $c$ using statements 3.1 - 3.6 of the next \cref{sec:analysis}, and in the \cref{tab:front_constants} for additional distributions of entries of a matrix which may be known or sampled as such. Avoiding such an assumption on matrices $A$ and $B$ using bounds such as statements 3.8 and 3.14 are possible as well. We may include additional terms of the decomposition to reduce $\norm{\Delta A}_F$ and $\norm{\Delta B}_F$ in approximating the product if these resulting a priori estimates of relative error are not satisfactory. This step is optional in practice.
    \item We may also estimate the relative error after computing the approximate product with a given number of components. If the resulting relative error is not small enough, we may then approximate the residue $\Delta A. \Delta B$ to a required tolerance using the same approach and sum it to the current approximation $M$. This posterior estimation of relative error can be performed using statement 3.7 of the next \cref{sec:analysis}.
    \item One may improve the posterior estimation using results in statements 3.8 and 3.14 of the next section at a marginal increment of the cost, if assumptions on residue matrices $\Delta A$ and $\Delta B$ such as signed independent entries or a flat SVD/CD spectra, are to be strictly avoided.
    \item One may use $\norm{A(BG)}_F/\sqrt{l}$ with a thin $n \times l$ random matrix $G$ with entries from $\mathcal{N}(0,1)$ or a Rademacher distribution \cite{Hutchinson_1989} where $l \ll n$, as an estimate of $\norm{AB}_F$ in the denominator, and similarly for $\norm{\Delta A.\Delta B}_F$ in the numerator of the relative error. This numerical/statistical estimation of the Frobenius norms may be applicable for larger dimensions and moderate tolerance in relative errors $\epsilon_{norm}$ of the estimated norms. Note that the dimension $l$ of the matrix $G$ has to vary as $\epsilon_{norm}^{-2}$. One may be able to allow larger errors $\epsilon_{norm}$ when the relation $\norm{\Delta A.\Delta B}_F \ll \norm{AB}_F$ is already well satisfied.
\end{itemize}

\section{Analysis : Relative errors and its estimation}\label{sec:analysis}

Our goal is to derive relations useful for estimating the relative error in the approximate matrix products i.e. $\frac{\norm{\Delta A \cdot \Delta B}_F} {\norm{AB}_F}$ for any given problem, using $\norm{\Delta A}_F$, $\norm{\Delta B}_F$ and $\norm{A}_F$, $\norm{B}_F$ respectively. In the first part of this section, we present results on the expected Frobenius norm of a product of matrices, when entries of the two matrices are independent. In this initial part of the analysis using statements numbered 3.1 - 3.6, we present results that allow us to upper bound the expected relative errors, with derivations that are agnostic to the decomposition used. These results are relatively robust to varying distributions of the entries in the matrix (known a priori, or sampled as such) using the numerically evaluated front constants $c$ in $\mathbb{E}[\norm{\Delta A \Delta B}_F] \approx c \norm{\Delta A}_F \norm{\Delta B}_F$, which are presented in \cref{tab:front_constants} in the appendix. We also present \cref{th:corollary_error_general} which allows us to upper bound or estimate the relative error for the product using the norm $\norm{M}_F$ of the approximate product, under the assumption of independent entries of residue matrices, post the multiplication.

Later we presents results in statements 3.8 to 3.14 where we show that the above derived relations also apply to matrices even with dependent entries, when they have a relatively flat (SVD or CD) spectra i.e. when the dominant components have been already included in the truncated decomposition leaving a residue spectra of nearly constant magnitudes. One can also use the evaluated spectra of a partial SVD or the full CD to derive more stringent error estimators with no such assumptions on the spectra, but with a marginal increase in the cost of estimation, as shown.

To proceed, we consider matrices with a Haar distribution of singular vectors for any given distribution of singular values, which also encompass matrices with random entries from a standard normal distribution.

\begin{lemma}\label{th:expected_norm_square}
    Let $A = Q_1 D_1 Q_2$ and $B = Q_3 D_2 Q_4$ be singular value decompositions of $A$ and $B$, and $Q_1,Q_2,Q_3,Q_4$ be Haar distributed unitary matrices. Then we have $\mathbb{E}(\norm{AB}^2_F) = \frac{1}{n}\norm{D_1}^2 \norm{D_2}^2$. 
\end{lemma}
\begin{proof}
    The Frobenius norm is invariant under unitary matrix multiplication:
    \begin{align*}
        \norm{AB}^2_F &=  \norm{D_1 Q_2Q_3 D_2}^2_F.
    \end{align*}
    Then we have Haar distribution which is invariant under unitary matrix multiplication, so $Q = Q_2Q_3$ is also Haar distributed. Therefore, 
     \begin{align*}
        \norm{AB}^2_F &=  \norm{D_1 Q D_2}^2_F, \\
        &= \sum_{i,j} D^2_1(i) Q^2_{ij} D^2_2(j).
    \end{align*}
    The expected values become:
    \begin{align*}
        \mathbb{E}[\norm{AB}^2_F] &=  \norm{D_1 Q D_2}^2_F, \\
        &= \sum_{i,j} D^2_1(i)\mathbb{E}[ Q^2_{ij}] D^2_2(j), \\
        &= \sum_{i,j} \frac{1}{n}D^2_1(i) D^2_2(j) = \frac{1}{n} \norm{D_1}_F^2 \norm{D_2}_F^2.
    \end{align*}
\end{proof}

Next to apply concentration inequality we need the Lipschitz constant of the function.
$f(Q) = \norm{D_1 Q D_2}^2_F = Tr(D_1^2Q D_2^2 Q^T)$. 
Now we need to consider: 
\begin{align*}
   | f(Q) - f(P) | &= |Tr(D_1^2(Q D_2^2 Q^T - P D_2^2 P^T ))| \\
   & \leq \norm{D_1^2}_F \norm{(Q D_2^2 Q^T - P D_2^2 P^T )}_F \quad \text{Cauchy-Schwarz for Frobenius inner product.}\\
   & \leq \norm{D_1^2}_F \norm{Q D_2^2 Q^T - Q D_2^2 P + Q D_2^2 P- P D_2^2 P^T}_F \\
   &\leq \norm{D_1^2}_F \left(\norm{Q D_2^2 Q^T - Q D_2^2 P^T}_F + \norm{Q D_2^2 P^T- P D_2^2 P^T}_F \right) \quad \text{Triangle inequality.}\\
   & \leq 2 \norm{D_1^2}_F \norm{D_2^2}_2 \norm{Q-P}_F \quad \text{using the unitary invariance and the maximum diagonal entry}. 
\end{align*}
So the function has the Lipschitz constant $2 \norm{D_1^2}_F \norm{D_2^2}_2$.

From the symmetry of the problem we have $L = \min \{ 2 \norm{D_1^2}_F \norm{D_2^2}_2, 2 \norm{D_1^2}_2 \norm{D_2^2}_F \}$. Next we recall a theorem from \cite{meckes2019random} (Chapter 5 Theorem 5.17), 

\begin{theorem}\textbf{(Concentration on compact groups:)}\label{elizabeths}
Given $n_1, \cdots , n_k \in  \mathbb{N}$, let $X = \mathbb{G}_{n_1} \times \cdots \times \mathbb{G}_{n_k}$, where for each $n_i$, $\mathbb{G}_{n_i}$ is one of $\mathbb{SO}(n_i)$, $\mathbb{SO}^{-}(n_i)$, $\mathbb{SU}(n_i)$, $\mathbb{U}(n_i)$, or $\mathbb{Sp}(2n_i)$. Let $X$ be equipped with the $\ell_2$-sum of Hilbert--Schmidt metrics on the $\mathbb{G}_{n_i}$.

Suppose that $F : X \to \mathbb{R}$ is $L$-Lipschitz, and that $\{ U_j \in \mathbb{G}_{n_j} : 1 \leq j \leq k \}$ are independent, Haar-distributed random matrices. Then for each $t > 0$,
\[
\mathbb{P} \left[ F(U_1, \dots, U_k) \geq \mathbb{E}[F(U_1, \dots, U_k)] + t \right] 
\leq \exp\left( -\frac{(n - 2)t^2}{24L^2} \right),
\]
where $n = \min\{ n_1, \dots, n_k \}$.
\end{theorem}

 Let $J = \text{diag}(-1,1,\cdots ,1)$. Then we decompose the space of orthogonal matrices as $\mathbb{O}(n) = \mathbb{SO}(n) \bigsqcup (\mathbb{SO}(n)J)$. Then note that $-F(Q) = \norm{D_1 Q D_2}_F^2 = \norm{D_1 Q D_2 J}_F^2 = \norm{D_1 QJ D_2}_F^2 = -F(QJ)$. Let us define the measure $\mu_{\mathbb{O}(n)} = \frac{1}{2} \left( \mu_{\mathbb{SO}(n)}   +  \mu_{\mathbb{SO}(n) J} \right)$. This gives us $\mathbb{E}_{\mathbb{O}(n)}[F] = \frac{1}{2} \left( \mathbb{E}_{\mathbb{SO}(n)}[F]   +  \mathbb{E}_{\mathbb{SO}(n) J}[F] \right) = \mathbb{E}_\mathbb{SO}(n)[F]$.  

Then we have the corollary in our setting:
\begin{corollary}\label{th:concentrationprob}
    $ \mathbb{P} \left[   \norm{D_1 Q D_2}^2_F \leq \mathbb{E}[\norm{D_1 Q D_2}^2_F] - t \right] 
\leq \exp\left( -\frac{(n - 2)t^2}{96 \norm{D_1}^4_F \norm{D_2}_2^4} \right)$.
\end{corollary}

\begin{proof}
Let us apply the Theorem \ref{elizabeths}, with $X = \mathbb{SO}(n)$ the norm $\norm{\cdot}_F$, function $F = - f(Q) = - \norm{D_1 Q D_2}^2_F$, and the Lipschitz constant $L = 2 \norm{D_1}^2_F \norm{D_2}^2_2$. This gives us:
\begin{align}
    \mathbb{P} \left[  - \norm{D_1 Q D_2}^2_F \geq \mathbb{E}[- \norm{D_1 Q D_2}^2_F] + t \right] 
\leq \exp\left( -\frac{(n - 2)t^2}{24L^2} \right), \\
 \mathbb{P} \left[\norm{D_1 Q D_2}^2_F \leq \mathbb{E}[\norm{D_1 Q D_2}^2_F] - t \right] 
\leq \exp\left( -\frac{(n - 2)t^2}{96 \norm{D_1}^4_F \norm{D_2}_2^4} \right) \label{concentration}. 
\end{align}
\end{proof}

Since $\mathbb{O}(n)$ is disjoint union of two cosets of $\mathbb{SO}(n)$, the Haar measure of $\mathbb{O}(n)$ is average of the Haar measures on these two cosets, and $F$ takes identical values on the $Q \in \mathbb{SO}(n)$ and $QJ \in \mathbb{SO}(n)J$, we have the same concentration ineqality extended to $\mathbb{O}(n)$. Thus from \eqref{concentration}, we have ${\norm{D_1 Q D_2}^2_F} \approx \mathbb{E}[\norm{D_1 Q D_2}^2_F]$ with very high probability when $L^2/n \to 0$. 

Thus $\norm{AB}_F \rightarrow \frac{1}{\sqrt{n}}\norm{A}_F \norm{B}_F$ for a Haar distribution of singular vectors and any given singular value distribution in the limit of large dimensions. This included entries of the matrix from a standard normal distribution. The above results allows us to estimate $\norm{\Delta A \cdot \Delta B}_F$, and $\norm{AB}_F$ as well, given $\norm{\Delta A}_F$,$\norm{\Delta B}_F$, and $\norm{A}_F$,$\norm{B}_F$ when entries of these matrices are expected to equally likely to positive or negative values. On the other hand, to estimate an expected value of $\norm{AB}_F$ given $\norm{A}_F$ and $\norm{B}_F$, where entries of $A$, $B$ are unsigned values, we proceed to the following.

\begin{theorem}\label{th:uniform_distribution_norms}
    For two real random matrices $A^{m \times n}$ and $B^{n \times p}$ with entries $\sim$ $U(0,a)$, $\mathbb{E}[\norm{AB}_F^2] \approx \frac{9}{16}\norm{A}_F^2 \norm{B}_F^2$ for $n \gg 1$.
\end{theorem}

\begin{proof}
    Given $A \in \mathbb{R}^{m\times n}$, and $B \in \mathbb{R}^{n\times p}$, we have: \begin{align}   \|AB\|^2_F&=\sum_{i=1}^m\sum_{j=1}^p(AB)_{ij}^2\notag\\
    &=\sum_{i=1}^m \sum_{j=1}^{p}\left\{\sum_{k=1}^n A_{ik}B_{kj}\right\}^2\notag\\
    &=\sum_{i=1}^m\sum_{j=1}^p\sum_{k=1}^n\left\{A_{ik}^2 B_{kj}^2 + \sum_{l\ne t,t=1}^nA_{il}B_{lj}A_{it}B_{tj}\right\}\label{eq:uniform_distribution_norms1}
    \end{align}
Now, for an expectation on the both sides of $\eqref{eq:uniform_distribution_norms1}$, and recalling that $\mathbb{E}[X]=\frac{a}{2}$, $\mathbb{E}[X^2]=\frac{a^2}{3}$, when $X \sim \mathcal{U}(0,a)$, we have: 
\begin{align}
    \mathbb{E}[\|AB\|_F^2] &=\sum_{i=1}^m\sum_{j=1}^p \sum_{k=1}^n\left\{\mathbb{E}[A_{ik}^2 B_{kj}^2] + \sum_{l\ne t, t=1}^n \mathbb{E}[A_{il}B_{lj}A_{it}B_{tj}]\right\}\notag\\
    &=\sum_{i=1}^m\sum_{j=1}^p\sum_{k=1}^n \left\{\left(\frac{a^2}{3}\right)^2+\sum_{l\ne t,t=1}^n\left(\frac{a}{2}\right)^4\right\} \notag\\ 
    &=mpn\frac{a^4}{9}+mpn(n-1)\frac{a^4}{16}\label{eq:uniform_distribution_norms2}
\end{align}

Similarly,
\begin{equation}
    \mathbb{E}[\|A\|_F^2] = mn\frac{a^2}{3} \text{ and }
    \mathbb{E}[\|B\|_F^2] = np\frac{a^2}{3}\label{eq:uniform_distribution_norms3}
\end{equation}

Using \eqref{eq:uniform_distribution_norms2} and \eqref{eq:uniform_distribution_norms3}, we have:
\begin{align}
    \frac{\mathbb{E}[\|AB\|_F^2]}{\mathbb{E}[\|A\|_F^2].\mathbb{E}[\|B\|_F^2]}&=\frac{mpn\frac{a^4}{9}+mpn(n-1)\frac{a^4}{16}}{mn^2p\frac{a^4}{9}} 
    = \frac{\frac{a^4}{9n}+\frac{(n-1)a^4}{16n}}{\frac{a^4}{9}} \\
    & \approx \frac{9}{16}\text{for } n \gg 1
\end{align}

\end{proof}

\begin{lemma}
    Let $Z = \norm{AB}_F^2$ with $n \times n$ real matrices $A$ and $B$ having entries from $\mathcal{U}(0,a)$; then $\text{Var}(Z) = \mathcal{O}(n^7)$. We also have  $ \mathbb{P}(|Z -  \mathbb{E}[Z]| > \epsilon \mathbb{E}[Z]) = \frac{1}{\epsilon^2} \mathcal{O}\left(\frac{1}{n}\right).$\label{th:tailbound_uniform_case}  
\end{lemma}
\begin{proof}
    Let $Y_{ij} = (AB)^2_{i,j}$. Then we have:
    \begin{align}
        Z &= \sum_{i,j} Y_{i,j} \text{ and } \mathbb{E}[Z] = \mathcal{O}(n^4).\\
        \text{Var}(Z) &= \sum_{i,j} \text{Var}(Y_{i,j}) + 2\sum_{(i,j) \neq (i',j')} \text{Cov}(Y_{i,j},Y_{i',j'}). \label{eq:cov_uniform_D} 
    \end{align}

We have $\text{Var}(Y_{i,j}) = \text{Var}((AB)_{i,j}^2) \leq \mathbb{E}[(AB)^4_{i,j}] \leq \mathbb{E}[\norm{A_{(i)}}^4] \mathbb{E}[\norm{B^{(j)}}^4]$, expanding the expectation of fourth power of norms and assuming finite fourth moment on the entries of the matrix, we have $\text{Var}(Y_{i,j}) = \mathcal{O}(n^4)$.  
We also have $\text{Cov}(Y_{i,j}Y_{i',j'}) = 0$ whenever $i \neq i'$ and $j \neq j'$. 
Further, we apply Cauchy-Schwarz inequality on each of the other terms in the sum in \eqref{eq:cov_uniform_D} to get
$\text{Cov}(Y_{i,j}Y_{i',j'}) \leq \sqrt{\text{Var}(Y_{i,j} )\text{Var}(Y_{i',j'})} = \mathcal{O}(n^4)$.
We know that there are $\mathcal{O}(n^3)$ such terms given by the combination of indices for which either $i = i'$ or $j = j'$. Thus, we have $\text{Var}(Z) = \mathcal{O}(n^7)$. 

Now, we apply the Chebyshev inequality:
\begin{align*}
    \mathbb{P}(|Z -  \mathbb{E}[Z]| > \epsilon \mathbb{E}[Z]) \leq \frac{\text{Var}(Z)}{\epsilon^2 \mathbb{E}[Z]^2} = \frac{1}{\epsilon^2}\mathcal{O}\left(\frac{1}{n} \right)
\end{align*}
\end{proof}

\begin{theorem}
The expected relative error in the first-order multiplication of $A$ and $B$ can be upper bound by -

Case-1: Mean-zero entries
\begin{equation*}
    \mathbb{E}\left[\frac{\norm{\Delta A \cdot \Delta B}_F}{\norm{AB}_F}\right] \leq \frac{1}{\sqrt{1-\epsilon}} \frac{\norm{\Delta A}_F \norm{\Delta B}_F}{ \norm{A}_F \norm{B}_F } + C \sqrt{\mathbb{P}(G^c)}.
\end{equation*}

Case-2: Unsigned entries
\begin{equation*}
    \mathbb{E}\left[\frac{\norm{\Delta A \cdot \Delta B}_F}{\norm{AB}_F}\right] \leq \frac{1}{\sqrt{n(1-\epsilon)}} \frac{4\norm{\Delta A}_F \norm{\Delta B}_F}{3 \norm{A}_F \norm{B}_F } + C \sqrt{\mathbb{P}(G^c)}.
\end{equation*}   
\end{theorem}

\begin{proof}
    Consider $X = \norm{\Delta A.\Delta B}_F]$ and $Z =  \norm{AB}^2_F >0$.
    Then we have, 
    \begin{align*}
        \mathbb{E}\left[ \frac{X}{\sqrt{Z}}\right] & = \mathbb{E}[X Z^{-1/2}] \\
        &\leq \mathbb{E}[X^2]^{1/2} \mathbb{E}[Z^{-1}]^{1/2}  \text{ Cauchy-Schwarz inequality.}
    \end{align*}

Let $\mu_Z = \mathbb{E}[Z]$. Then the good event be $G = \{ Z : Z \geq (1-\epsilon) \mu_Z \}$ for $0 < \epsilon < 1$.  
Then we split $\mathbb{E}[Z^{-1}] = \mathbb{E}[Z^{-1} \mathbf{1}_G] + \mathbb{E}[Z^{-1} \mathbf{1}_{G^c}]$. 
This further gives us:
\begin{align*}
    Z^{-1} &\leq \frac{1}{(1-\epsilon) \mu_Z}  \text{       on $G$ }\\
    \mathbb{E}[Z^{-1}] &\leq \frac{1}{(1-\epsilon) \mu_Z}  \text{       on $G$ }
\end{align*}

Now we use the fact that $\sqrt{a+b} \leq \sqrt{a} + \sqrt{b}$, to write:
\begin{align}
    \mathbb{E}\left[ \frac{X}{\sqrt{Z}}\right]  &\leq \mathbb{E}[X^2]^{1/2} \left(  \frac{1}{\sqrt{(1-\epsilon) \mu_Z}} + \sqrt{\mathbb{E}[Z^{-1} \mathbf{1}_{G^c}]} \right). \label{expectineq} 
\end{align}

\begin{align}   \mathbb{E}\left[\frac{X}{\sqrt{Z}}\right]  &\leq  \frac{\mathbb{E}[X^2]^{1/2}}{\sqrt{(1-\epsilon) \mu_Z}}  + R_\epsilon \label{eq:expectation_bound}
\end{align}

With $R_{\epsilon} = \mathbb{E}[X^2]^{1/2}\sqrt{\mathbb{E}[Z^{-1} \mathbf{1}_{G^c}]}$. We know that $\norm{AB}_F \geq \sigma_{min}(A) \norm{B}_F$, so we have:
\begin{align*}
    R_\epsilon &\leq \frac{\mathbb{E}(X^2)^{1/2}}{\sigma_{min}(A) \norm{B}_F} \sqrt{\mathbb{P}(G^c)}  \\
    &\leq C \sqrt{\mathbb{P}(G^c)}.
\end{align*}

Now for the case I : The singular vectors of $A, B, \Delta A, \Delta B$ are from the Haar measure. So we use Lemma \ref{th:expected_norm_square} in both numerator and denominator of \eqref{eq:expectation_bound} to get:
\begin{align}
    \mathbb{E}\left[ \frac{X}{\sqrt{Z}}\right]  &\leq \frac{1}{\sqrt{1-\epsilon}} \frac{\norm{\Delta A}_F \norm{\Delta B}_F}{ \norm{A}_F \norm{B}_F } + C \sqrt{\mathbb{P}(G^c)}.
\end{align}

From the Corollary \ref{th:concentrationprob}, we have $\mathbb{P}(G^c) \leq e^{-nc\epsilon^2}$ so for some small $\delta >0$ and large enough $mn$ we have $R_\epsilon < \delta$. 

For the case II: When $\Delta A, \Delta B$ have Haar singular vectors, but $A,B$ have uniform entries from a positive distribution, we use \cref{th:uniform_distribution_norms} for the expected value $\mu_z$ in \eqref{eq:expectation_bound} to get:

\begin{align*}
    \mathbb{E}\left[ \frac{X}{\sqrt{Z}}\right]  &\leq \frac{1}{\sqrt{n(1-\epsilon)}} \frac{4\norm{\Delta A}_F \norm{\Delta B}_F}{3 \norm{A}_F \norm{B}_F } + C \sqrt{\mathbb{P}(G^c)}.
\end{align*}

From \cref{th:tailbound_uniform_case}, we have the $\sqrt{\mathbb{P}(G^c)} = \mathcal{O}(\frac{1}{\sqrt{n}})$ and therefore this term approaches zero for large $n$. 
\end{proof}

\begin{corollary}
    The expected relative error in the first order approximation of the multiplication of two real matrices $A$, $B$ in general, can be bound as:
    \begin{align*}
    \mathbb{E}\left[ \frac{\norm{\Delta A.\Delta B}_F}{\norm{AB}_F}\right]  &\leq \frac{1}{\sqrt{n}} \frac{\norm{\Delta A}_F \norm{\Delta B}_F}{\norm{M}_F }
\end{align*}
when using the truncated SVD, and estimated as:
    \begin{align*}
    \mathbb{E}\left[ \frac{\norm{\Delta A.\Delta B}_F}{\norm{AB}_F}\right]  &\approx \frac{1}{\sqrt{n}} \frac{\norm{\Delta A}_F \norm{\Delta B}_F}{\norm{M}_F }
\end{align*}
when using the truncated circulant and Fourier decompositions, where the approximation $M=AB-\Delta A\Delta B$. \label{th:corollary_error_general}
\end{corollary}
\begin{proof}
    Let $A=Q_1 \Sigma^AV^T$ and $B=U\Sigma^BQ_2$ be the singular value decompositions for real matrices $A$ and $B$, with corresponding right singular vectors $\hat{v}_i$ and left singular vectors $\hat{u}_i$. Using the truncated decomposition, we know:
    \begin{align*}
    M &= (A-\Delta A)B + \Delta A (B-\Delta B)\\
    \implies\norm{M}_F^2 &= \sum_{i=1}^k \sum_{j=1}^n (\hat{v}_i^T\hat{u}_j\sigma_i^A \sigma_j^B)^2 + \sum_{i=k+1}^n \sum_{j=1}^k (\hat{v}_i^T\hat{u}_j\sigma_i^A \sigma_j^B)^2\\ &\leq \sum_{i=1}^n \sum_{j=1}^n (\hat{v}_i^T\hat{u}_j\sigma_i^A \sigma_j^B)^2 = \norm{AB}_F^2
    \end{align*}
Using the above known lower bound for the denominator, and with $\epsilon,\mathbb{P}(G^c)=0$, and the expected value $\norm{\Delta A. \Delta B}$ for the Haar measure in \ref{th:expected_norm_square} for the numerator of \eqref{eq:expectation_bound}, we get an upper bound:
    \begin{align*}
    \mathbb{E}\left[ \frac{\norm{\Delta A.\Delta B}_F}{\norm{AB}_F}\right]  &\leq \frac{1}{\sqrt{n}} \frac{\norm{\Delta A}_F \norm{\Delta B}_F}{\norm{M}_F}.
\end{align*}
In the case of other truncated decompositions, we know $|\norm{M}_F-\norm{\Delta A. \Delta B}_F| \leq \norm{AB}_F \leq \norm{M}_F + \norm{\Delta A. \Delta B}_F$. Thus the R.H.S of the above becomes an estimate. 
\end{proof}

In the next part of the analysis, we derive bounds for the norm of the residual in the approximate multiplication without assuming the independence of entries of the residue matrices. In turn we include the variables available in the partial SVD and full circulant decomposition in our analysis, and this may result in a marginal increase of the cost of  estimation. When these residue spectra are flat or nearly constant, we retrieve the earlier results on the error in the multiplication, even for dependent entries in the matrices.

\begin{theorem}[Residual bound for SVD]
Let $A, B \in \R^{n \times n}$, let $A_k, B_k$ be their best rank-$k$ approximations, and let $\DA = A - A_k$, $\DB = B - B_k$. Similarly, let $\rtilde(\DA)$, $\rtilde(\DB)$ be the stable ranks of $\DA$ and $\DB$ respectively. Then: $$\normF{\DA \DB}^2 \leq \frac{1}{\sqrt{\rtilde(\DA)\rtilde(\DB)}}\normF{\DA}^2 \normF{\DB}^2.$$
where $\rtilde(\DA):= \frac{\normF{\DA}^2}{(\sigma_{k+1}^A)^2} = \frac{\sum_{i=k+1}^n (\sigma_i^A)^2}{(\sigma_{k+1}^A)^2}$, and correspondingly for $\rtilde(\DB)$.
\end{theorem}

\begin{proof}
Given the right singular vectors $\hat{v}$ of $A$ and left singular vectors $\hat{u}$ of $B$, we know:
\begin{align*}
   \normF{\DA \DB}^2 &= \sum_{i=k+1}^n \sum_{j=k+1}^n \Bigl(\hat{v}_i^T\hat{u}_j\sigma_i^A \sigma_j^B\Bigr)^2\\
   & \leq \sum_{i=k+1}^n \bigl(\sigma_i^A\bigr)^2  \sum_{j=k+1}^n (\hat{v}_i^T\hat{u}_j)^2 \bigl(\sigma_j^B\bigr)^2 \text{ where } \sum_j (\hat{v}_i^T\hat{u}_j)^2\leq 1\\
   & \leq \sum_{i=k+1}^n \bigl(\sigma_i^A\bigr)^2 \bigl(\sigma_{k+1}^B\bigr)^2 = \normF{\DA}^2 \norm{\DB}_2^2
\end{align*}
giving us,
\begin{align}
\normF{\DA \DB}^2 & \leq \normF{\DA}^2 \norm{\DB}_2^2 \text{ and similarly}\\
 \normF{\DA \DB}^2 & \leq \norm{\DA}_2^2 \normF{\DB}^2.
\end{align}
Combining both the above inequalities:
\begin{equation*}
   \normF{\DA \DB}^2 \leq \norm{\DA}_2 \norm{\DB}_2\normF{\DA} \normF{\DB} 
\end{equation*}
Substituting $\rtilde(\DA)= \frac{\normF{\DA}^2}{\norm{\DA}_2^2}$ and $\rtilde(\DB)= \frac{\normF{\DB}^2}{\norm{\DB}_2^2}$ in the above, we complete the proof.
 
\end{proof}
Note that $\sigma_{k+1}^A$ and $\sigma_{k+1}^B$ are available in \cref{algo:rsvd} used for a truncated SVD with $k+1$ components. This results in $\sim 2n^2$ additional operations each for the two matrices in evaluating $k+1$ components instead of $k$ components, and thus the L.H.S of the above relation can be bound at approximately thrice the cost of the corresponding term in \cref{th:corollary_error_general}.

\begin{corollary}
    When all the singular values of $\DA$ are identical, $\rtilde(\DA)=n-k$, and similarly $\rtilde(\DB)=n-k$ for identical singular values in $\Delta B$. In such cases, we have:
    $$\normF{\DA \DB} \leq \frac{1}{\sqrt{n-k}}\normF{\DA} \normF{\DB}$$
    which for $k \ll n$ is well approximated by the corresponding relation of corollary 3.7.
\end{corollary}

\begin{lemma}\label{lem:circulant-conjugation}

Let $R \in \mathbb{C}^{n \times n}$ be circulant, i.e. $R(i,k) = r_{k-i \ (\mathrm{mod}\ n)}.$ Let $D = \mathrm{diag}(\omega^0, \omega^1, \dots, \omega^{n-1})$ with $\omega = e^{-2\pi i / n}.$ Then for any integer $j$, the matrix $\widetilde{R} := D^j R D^{-j}$ is also circulant. Consequently, $D^j R = \widetilde{R} D^j.$
\end{lemma}

\begin{proof}
We expand the $(i,k)$ entry of the conjugated circulant as:
\[
(D^j R D^{-j})(i,k)
= D^j(i,i)\, R(i,k)\, D^{-j}(k,k).
\]
Using $D^j(i,i) = \omega^{ij}$ and $D^{-j}(k,k) = \omega^{-kj}$, we obtain
\[
(D^j R D^{-j})(i,k)
=
\omega^{ij} \, R(i,k) \, \omega^{-kj}
=
R(i,k)\, \omega^{j(i-k)}.
\]
Since $R(i,k) = r_{k-i}$, this becomes
\[
(D^j R D^{-j})(i,k)
=
r_{k-i}\, \omega^{-j(k-i)},
\]
which depends only on $(k-i) \bmod n$. Hence $\widetilde{R}$ is circulant.
\end{proof}

\begin{definition}[Spectral concentration parameter]
Let $\tilde{A}$ be obtained by retaining the index set $S^A$ of $k$ largest circulant components, and $\DA = A - \tilde{A}$. For $\Delta A$, define for each Fourier mode $i$:
\[
X_i := \sum_{j \notin S^A} |\lambda_i(R_j^A)|^2.
\]
Let us define,
\[
\mu_A := \frac{\max_i X_i}{\frac{1}{n}\|\Delta A\|_F^2}.
\]
Similarly, define $\mu_B$ for $\Delta B$.
\end{definition}

\begin{Remark}\label{rem:Xi}
Let $\underline{\Delta A}(i,j) = \Lambda_j(i,i)$, where $\Delta A=\sum_0^{n-1}C^j\Lambda_j$ i.e. the cycles of $\Delta A$ as the entries of corresponding columns of $\underline{\Delta A}$. Then the $j^{\text{th}}$ row of $\frac{1}{\sqrt{n}} W \underline{\Delta A}$ is the first column of circulant matrix 
component $R_j$, where the unitary matrix $W(p,q) = \frac{1}{\sqrt{n}}e^{-i \frac{2\pi pq}{n}}$, with $p,q=0,1,2\cdots n-1$. The $j^{\text{th}}$ column of the matrix $W\underline{\Delta A}^TW^T$ has all the eigenvalues of $R_j$. Similarly, the 2-norm of the $i^{\text{th}}$ row of the matrix $W \underline{\Delta A}^T W^T$ is $\sqrt{X_i}$ representing a Fourier mode $i$, therefore:
\[\mu_A := \frac{\max_i X_i}{\frac{1}{n}\|\Delta A\|_F^2 } \leq \frac{\|\underline{\Delta A}\|^2_2}{\frac{1}{n}\|\Delta A\|_F^2}.\] Similarly, using the $j^{\text{th}}$ column of the matrix $W \underline{\Delta A}^T W^T$ i.e.  $\norm{(W \underline{\Delta A}^T W^T)^{(j)}}_2^2=\norm{R_j^A}_F^2$, we also have:
     \[ \frac{\max_{j \notin S^A} \normF{R_j^A}^2}{\frac{1}{n}\|\Delta A\|_F^2 } \leq \frac{\|\underline{\Delta A}\|^2_2}{\frac{1}{n}\|\Delta A\|_F^2}.\] 
\end{Remark}

\begin{theorem}[Residual bound for Circulant Decomposition]
\label{th:cd-deterministic}
Let $A, B \in \mathbb{R}^{n \times n}$ and let $\Delta A, \Delta B$ be their residuals from a truncated circulant decomposition. Then,
\[
\|\Delta A \Delta B\|_F
\;\le\;
\frac{(\mu_A \mu_B)^{1/4}}{\sqrt{n}}\;
\|\Delta A\|_F \, \|\Delta B\|_F.
\]
\end{theorem}

\begin{proof}
Using the circulant decomposition,
\[
\Delta A \Delta B
=
\sum_{j \notin S^A} \sum_{l \notin S^B}
R_j^A D^j R_l^B D^l.
\]
In the subsequent arguments we use $\sum_{j+l = s}^*$ to denote the indices $j \notin S^A$ and $l \notin S^B$ with their sum as $s$. 

Using the Lemma \ref{lem:circulant-conjugation}, we have $D^j R_l^B = \tilde{R}_l^B D^j$, and we obtain:
\[
\Delta A \Delta B
=
\sum_s R_s D^s,
\quad
R_s := \sum_{j+l=s}^* R_j^A \tilde{R}_l^B.
\]
By orthogonality, $\|\Delta A \Delta B\|_F^2 = \sum_s \|R_s\|_F^2.$ Since circulant matrices are diagonalized by the Fourier transform,
\[
\|R_s\|_F^2 = \sum_{i=1}^n |\lambda_i(R_s)|^2,
\quad
\lambda_i(R_s) = \sum_{j+l=s}^* \lambda_i(R_j^A)\lambda_i(R_l^B).
\]
Let $a_j^{(i)} := \lambda_i(R_j^A),
\quad b_l^{(i)} := \lambda_i(R_l^B).$

Then, $\lambda_i(R_s) = \sum_{j+l=s}^* a_j^{(i)} b_l^{(i)}$, and we have the circular convolution $\lambda_i(R_s) = (a^{(i)} \circledast b^{(i)})(s)$. 

Summing over $s$,
\[
\sum_s |\lambda_i(R_s)|^2
=
\|a^{(i)} \circledast b^{(i)}\|_2^2
\le
\|a^{(i)}\|_2^2 \|b^{(i)}\|_2^2.
\]
Thus,
\[
\sum_s |\lambda_i(R_s)|^2
\le
\left(\sum_{j \notin S^A} |\lambda_i(R_j^A)|^2\right)
\left(\sum_{l \notin S^B} |\lambda_i(R_l^B)|^2\right)
= X_i Y_i.
\]
Summing over, $\|\Delta A \Delta B\|_F^2
\le \sum_i X_i Y_i.$ Applying Cauchy--Schwarz inequality,
\[
\sum_i X_i Y_i
\le
\left(\sum_i X_i^2\right)^{1/2}
\left(\sum_i Y_i^2\right)^{1/2}.
\]

Given $\sum_i X_i = \|\Delta A\|_F^2$, we have
\[
\sum_i X_i^2 \le (\max_i X_i)\sum_i X_i
=
\frac{\mu_A}{n}\|\Delta A\|_F^4,
\]
and similarly for $Y_i$. Therefore,
\[
\|\Delta A \Delta B\|_F^2
\le
\frac{\sqrt{\mu_A \mu_B}}{n}
\|\Delta A\|_F^2 \|\Delta B\|_F^2.
\]

Applying square roots on both sides of the above completes the proof.
\end{proof}

\begin{corollary}\label{th:CD_final_bounds}
Note that $\mu_A, \mu_b \to 1$ indicating a constant energy in Fourier modes of the entries of circulants $R_k$ over all $k$, reduces the above in \cref{th:cd-deterministic} to the corresponding bounds for $\normF{\DA \DB}$ used in \cref{th:corollary_error_general}. Further, using Remark \ref{rem:Xi} for bounds on $\mu_A$ and $\mu_B$ along with the above results, we have: 
\begin{align*}
\normF{\Delta A \Delta B}^2 & \leq \norm{\underline{\Delta A}}_2 \norm{\underline{\Delta B}}_2 \normF{\Delta A} \normF{\Delta B} \\
\implies \normF{\Delta A \Delta B} & \leq (\norm{\underline{\Delta A}}_\infty \norm{\underline{\Delta A}}_1 \norm{\underline{\Delta B}}_\infty \norm{\underline{\Delta B}}_1)^{\frac{1}{4}} (\normF{\Delta A} \normF{\Delta B})^{\frac{1}{2}},          
\end{align*}
where we have used the inequality $\norm{\underline{\Delta A}}_2 \leq \sqrt{\norm{\underline{\Delta A}}_\infty \norm{\underline{\Delta A}}_1}$, and correspondingly for $\norm{\underline{\Delta B}}_2$. 
\end{corollary}
Considering the cost of evaluating the norms of matrices, note that the above expression in \cref{th:CD_final_bounds} can be evaluated at approximately thrice the cost of estimating the numerator in Corollary \ref{th:corollary_error_general}. One can directly estimate $\norm{\underline{\Delta A}}_2^2$ as well, using successive matrix-vector products in evaluating $\frac{\norm{(\underline{\Delta A}^T\underline{\Delta A})^kx}_2}{\norm{x}_2}$ with reasonably large enough values $k$ for some vector $x$. We can also evaluate $\mu_A$ and $\mu_B$ directly as part of \cref{algo:circulant_decomposition_fft} to use the corresponding bounds in \cref{th:cd-deterministic}. Note that $\underline{\DA}$ and $\underline{\DB}$ are explicit in the evaluation of the circulant decomposition.

Also, using Remark \ref{rem:Xi}, the condition $\frac{\|\underline{\Delta A}\|^2_2}{\frac{1}{n}\|\Delta A\|_F^2} \rightarrow 1$ is also sufficient to indicate a flat circulant spectra of $\DA$ where $\norm{R_j^A}_F^2=\normF{\Delta A}^2/n$. When this condition is satisfied by $\DA$ and $\DB$, we see that the results of \cref{th:CD_final_bounds} reduces to the corresponding bounds for $\normF{\DA \DB}$ used in \cref{th:corollary_error_general}.

\section{Numerical results} \label{sec:results}

In this section, relative errors $\frac{\norm{AB-M}_F}{\norm{AB}_F}$ and the required arithmetic operations are presented for various types of matrices multiplied using the approximate methods discussed earlier. The results also show the complementarity of the different decompositions for approximately multiplying matrices of various types. The distribution of the singular values and the magnitudes of circulant components for these example matrices are in \cref{sec:plots_SVD} and \cref{sec:plots_CD} for reference. It is significant that the relative errors can be estimated prior to the multiplication using $\norm{\Delta A}_F$, $\norm{\Delta B}_F$, $\norm{A}_F$, $\norm{B}_F$, and also post the approximation of the product using $\norm{\Delta A}_F$, $\norm{\Delta B}_F$ in $\mathcal{O}(n^2)$ arithmetic operations (see \cref{sec:summary_estimation} and \cref{sec:analysis}). Note that one can evaluate a decomposition to the required approximation for a given matrix in the pair, or ascertain the suitability of a decomposition.

\subsection{Truncated Singular Value Decomposition (SVD)} \label{sec:results_truncated_SVD}
Averaging results over an identical set of matrices for multiple methods relevant to SVD, we present a comparison between 1) a randomized low rank multiplication summing only $k = s\log_2 n < n$ outer products of columns of A and rows of B \cite{12,matrixmultiply2006Kannan} described in \cref{sec:low_rank_algorithm}; 2) the zeroth order multiplication using the largest $\lceil s \log_2 n \rceil$ singular value components of $A$ and $B$ using a randomized partial SVD; and 3) augmenting the above with the inclusion of first order terms in the multiplication; see \cref{fig:SVD_results1,fig:SVD_results2}.

We may observe that the randomized low rank multiplication summing $c= \lceil s \log_2 n \rceil$ outer products, implicitly represents up to $c$ largest singular value components in the range of $A$, where these column indices $k$ are chosen based on their 2-norm. It also includes up to $c$ singular value components of $B$, but not necessarily the largest of them, as its rows given by the identical set of indices $k$ have been chosen based on the 2-norm of the corresponding columns of $A$ (see \cref{sec:low_rank_algorithm} for more details). The converse is true if the set of indices $k$ are selected based on the 2-norm of the rows of $B$. Thus the zeroth order multiplication using the \textit{\textbf{top}} $s\log_2 n$ SVD components of \textit{\textbf{both}} $A$ and $B$ results in lower relative errors than this randomized outer products. But, as expected, the proposed first-order multiplication presents significantly lower errors without an increase in the order of scaling of the computing effort. Note that this includes the residue terms $A_i \cdot \Delta B$ and $\Delta A \cdot B_k$ as well, where $A_i$ and $B_k$ are the individual singular value components included in the truncated decomposition. We have also included a challenging example in matrix multiplication where the entries are oscillating in signs with exponentially varying magnitudes (see \cref{tab:svd1}). This sign-oscillatory ill-conditioned matrix `Kappa' (with condition number $\kappa \to \infty$) was synthesized using the relation $A_{ij}=A_{ji}=e^{(-0.5|i-j|)}\times sin(i+1)$ for the entries. In this case, the proposed first order approach achieves acceptable results using SVD, and better results in the CD based multiplication due to its periodicity in entries.

\begin{table}[H]
    \centering
    \begin{tabular}{|c|c|c|c|c|c|c|}
        \hline
        \multirow{2}{*}{Matrix} & \multicolumn{3}{c|}{$s$ for 5\% relative error} & \multicolumn{3}{c|}{$s$ for 1\% relative error}\\
        \cline{2-7}
        & Low-rank & SVD-Zeroth& SVD-First& Low-rank & SVD-Zeroth& SVD-First\\
         \hline
         Symmetric \& Toeplitz & 24 & 1& 1& -& 1& 1\\
         \hline
          Symmetric \& Symmetric& 16& 1& 1& -& 1& 1   \\
         \hline
         Symmetric \& Hankel &25 & 1& 1& - & 1& 1 \\
         \hline
         General \& Symmetric &25  & 1& 1 & - & 1 & 1 \\
         \hline
         DFT Model \& DFT Model &8  &1 &1 & -&1 & 1 \\
         \hline
         Images\footnotemark[1]& 9& 1&1 & - & 5&2  \\
         \hline
        Bus494 \& Bus494\footnotemark[2]& 50 & 4&2 &62  &10 &3  \\
         \hline
         Type-1 \& Type-1\footnotemark[4]& - &5 &3 &-  &7 & 4 \\
         \hline
         Type-1 \& Toeplitz& - & 73& 4&-  & -&6  \\
         \hline
        Toeplitz \& Toeplitz &34 & 1& 1& - & 31 & 9\\
         \hline
        Block Toeplitz \& Block Toeplitz & 34& 1& 1 & - & 30 & 9\\
         \hline
        Topelitz \& Hankel&  33& 1 & 1 & - & 31& 9  \\
         \hline
         General \& Toeplitz& 34& 1& 1 & - & 30 & 9 \\
         \hline
        Hankel \& Hankel& 34 &1 & 1& -& 31& 9\\
         \hline
        General \& Hankel& 33 & 1 & 1& -& 30& 9 \\
        \hline
        LLM-1($Q_1I_1,K_1I_1,V_1I_1$)\footnotemark[4]&-&171&8&-&183&9\\
         \hline
         LLM-2($Q_2I_2, K_2I_2, V_2I_2$)\footnotemark[4]&-&171&8&-&184&9\\
         \hline
          General \& General &33  & 1& 1& - & 29& 9\\
         \hline
         Bus662 \& Bus662\footnotemark[3]& 74 &18 &7 &74  &39 &15  \\
         \hline
         Type-2 \& Type-2\footnotemark[4]& - &27 &2 & - & 67& 16 \\
         \hline
         Kappa \& General& - & 67 &38 &- &76 & 59 \\
         \hline
        Kappa \& Kappa &-  &68 & 40& - & 76& 60 \\
         \hline
        Kappa \& Toeplitz &- & 69& 39 & - &- & 61 \\
         \hline
        
         Type-3 \& Toeplitz&-  &74 & 56&-  &- & 69 \\
         \hline
         Type-3 \& Type-3\footnotemark[4]& - &76 & 62& - & -&71  \\
         \hline
         
    \end{tabular}
    \caption{Results for different types of matrices of size $700 \times 700$ (see footnotes for exceptions). The average numbers $s$ reported indicate the corresponding front constants in the $\mathcal{O}(n^2 \log_2 n)$ scaling of arithmetic operations, using $\lceil s \log_2 n \rceil$ components of the Singular Value Decomposition (SVD). `-' indicates that it did not achieve the desired tolerance in relative error even after using $2n^3$ arithmetic operations required in the conventional multiplication. DFT - Density Functional Theory; LLM - Large Language Model. All structured matrices with some periodicity in entries are of low effective ranks.}
    \label{tab:svd1}
\end{table}
\footnotetext[1]{RGB images have been normalized and resized into 2D with size 700$\times$700}
\footnotetext[2]{494 BUS matrix was taken from MatrixMarket\cite{matrixmarket}}
\footnotetext[3]{662 BUS matrix was taken from Matrixmarket\cite{matrixmarket}}
\footnotetext[4]{\cref{sec:matrix_generation} contains details generating matrices with low-rank (Type-1), dominant low-rank components (Type-2), and linear decay of singular values (Type-3), respectively. It also contains the details for generating the Query, Key, Value, and Input matrices $Q_i, K_i, V_i, I_i$, respectively, each of size $2048 \times 2048$.}

The large improvements of the proposed first-order approach over the conventional low-rank multiplication and the zeroth order approximation is also emphasized using the examples of a Large Language Model (LLM) in \cref{tab:svd1}. Even if one of the matrices in the product has high effective rank, the first-order approach presents small relative errors as in evaluation of $QI$, $KI$ and $VI$, where only $I$ is of low rank. The matrix product $QK^T$ used in LLM-1 and LLM-2 required $125\log_2 n$ and $158\log_2 n$ components and provide only marginal gains in the proposed first-order approach to achieve relative errors of $5\%$ and $1\%$, respectively. This is due to the dominance of all singular values of such matrices, unlike the input matrix $I$.

\begin{figure}
    \centering
    \begin{minipage}[b]{0.48\textwidth}
        \centering
        \includegraphics[width=\textwidth]{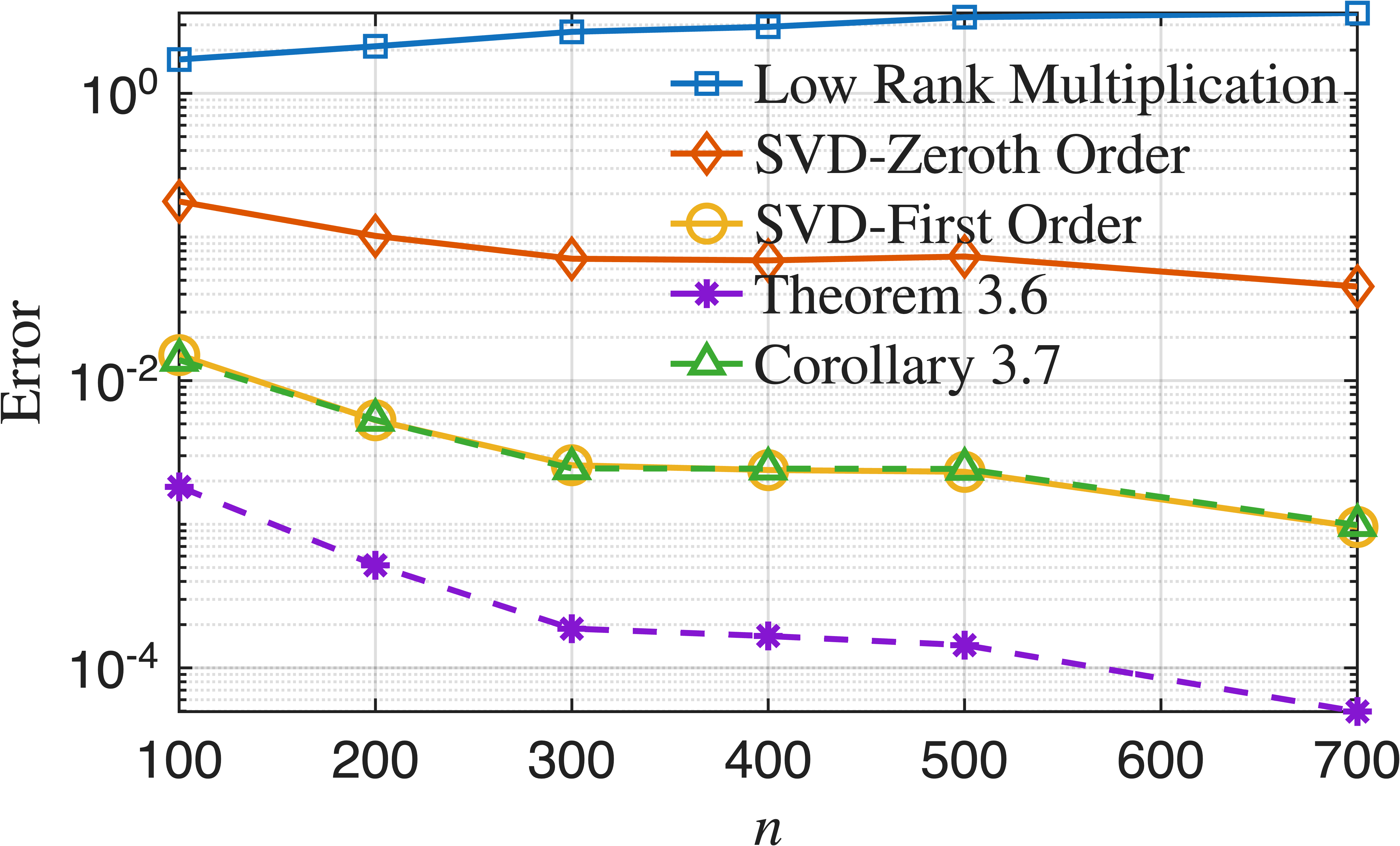}
        \caption{Size of matrix vs. mean relative error in multiplying two Type-1 matrices. $5\lceil \log_2 n \rceil$ components of the SVD were used. A priori estimates using Theorem 3.6 ignoring its $\mathcal{O}(\frac{1}{\sqrt{n}})$ second term, and posterior estimates using Corollary 3.7 are also plotted above. While the former is smaller than the actual errors, the latter provides estimates that overlap with the observed relative error.}\label{fig:SVD_results1}
    \end{minipage}
    \hfill
        \begin{minipage}[b]{0.48\textwidth}
        \centering
        \includegraphics[width=\textwidth]{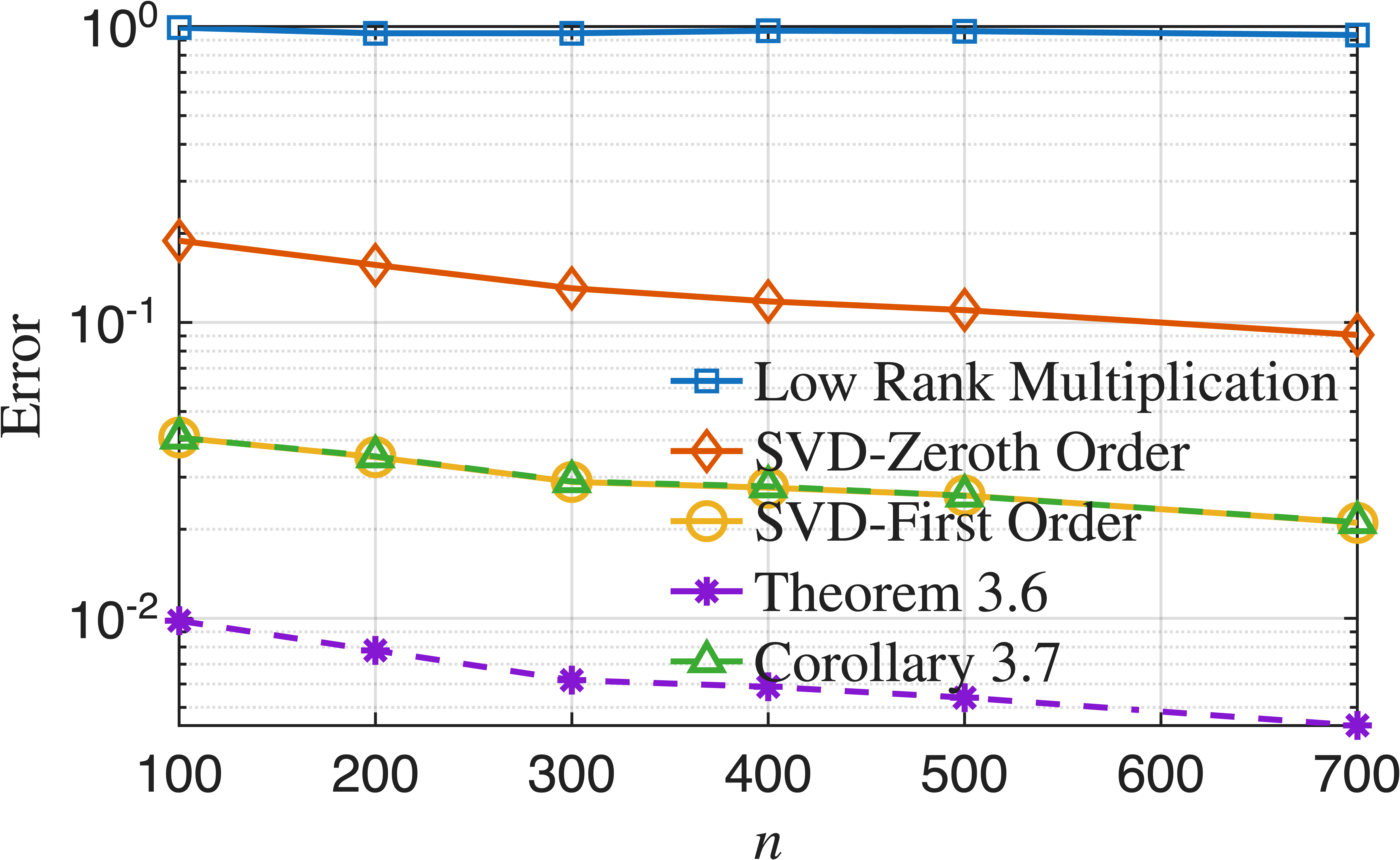}
        \caption{Size vs. mean relative error in multiplying two Type-2 matrices. $5\lceil \log_2 n \rceil$ components of the SVD were used. A priori estimates using Theorem 3.6 ignoring its $\mathcal{O}(\frac{1}{\sqrt{n}})$ second term, and posterior estimates using Corollary 3.7 are also plotted above. While the former is smaller than the actual errors, the latter provides estimates that overlap with the observed relative error.}\label{fig:SVD_results2}
    \end{minipage}
\end{figure}

\subsection{Truncated Circulant Decomposition}

The circulant decomposition is well-suited to provide a compressed approximation of matrices with some periodicity in their entries (see \cref{sec:plots_CD} for examples). Thus, they can present small relative errors even in an approximate multiplication using a first order approach used here; see \cref{fig:CD_results1,fig:CD_results2}. The circulant components are not an ideal basis to compactly approximate sparse matrices, and unstructured dense matrices $A$ with random entries and a mean value $\sim 0$, both of which are not of interest here. But, note that these matrices presented reasonable results in the SVD based multiplication, thus establishing the complementarity of these two decompositions. For example, the sparse matrices indicated by `Bus-664' and `Bus-494', and the multiplication of two general matrices with random entries require significantly more circulant components to converge to the required tolerances in the relative error; see \cref{tab:cd1}. Similarly, while the SVD based multiplication may produce reasonable results with periodic matrices of low rank as in \cref{tab:svd1}, they fail for full rank matrices with the same periodicity in entries. Whereas the CD based multiplication produces useful results agnostic of the rank when the matrices have periodicity in entries.

\begin{figure}
    \centering
    \begin{minipage}[b]{0.48\textwidth}
        \centering
        \includegraphics[width=\textwidth]{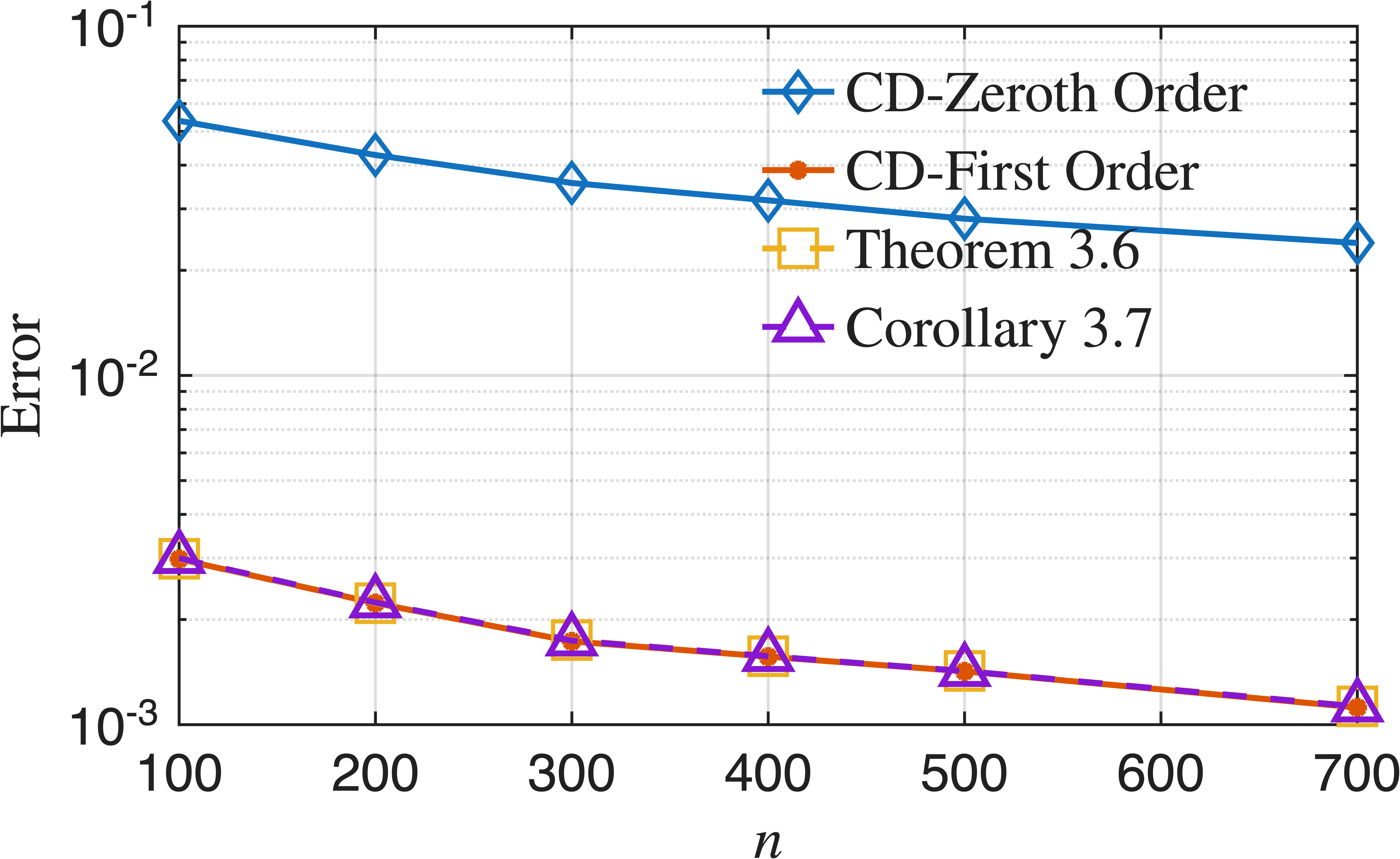}  
        \caption{Size of matrix vs. mean relative error in multiplying a general and a Toeplitz matrix with randomly generated entries from $\mathcal{U}$(0,1). $5\lceil \log_2 n \rceil$ components of the CD were used. A priori estimates using Theorem 3.6 ignoring its $\mathcal{O}(\frac{1}{\sqrt{n}})$ second term, and posterior estimates using Corollary 3.7 are also plotted above. Both provide estimates that overlap with the observed relative error.}\label{fig:CD_results1}
    \end{minipage}
    \hfill
    \begin{minipage}[b]{0.48\textwidth}
        \centering
        \includegraphics[width=\textwidth]{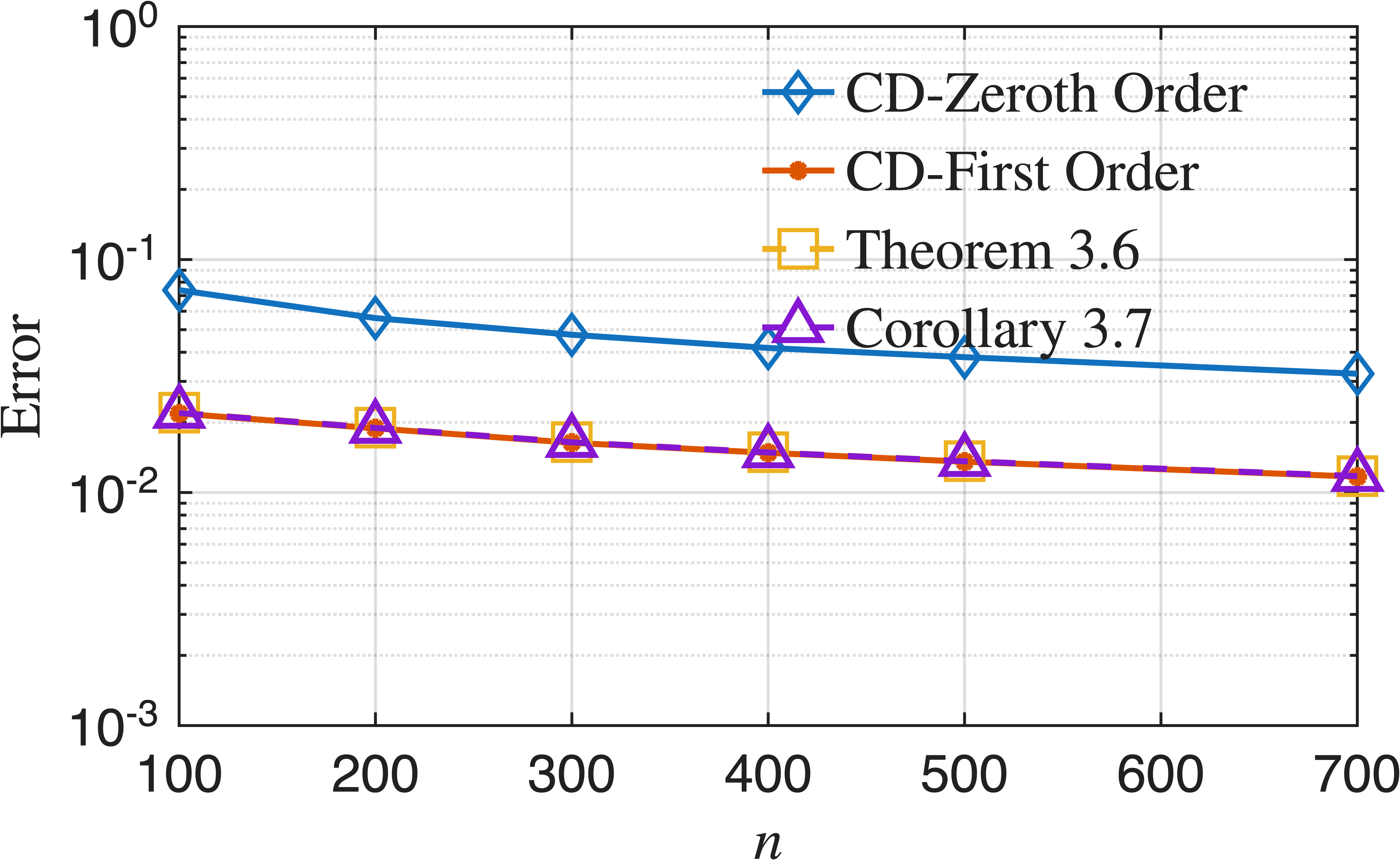}  
        \caption{Size of matrix vs. mean relative error in multiplying two random matrices with entries from a $\mathcal{U}$(0,1). $5\lceil \log_2 n \rceil$ components of the CD were used. A priori estimates using Theorem 3.6 ignoring its $\mathcal{O}(\frac{1}{\sqrt{n}})$ second term, and posterior estimates using Corollary 3.7 are also plotted above. Both provide estimates that overlap with the observed relative error.}\label{fig:CD_results2}
    \end{minipage}
\end{figure}

\begin{table}[H]
    \centering
    \begin{tabular}{|c|c|c|c|c|}
        \hline
        \multirow{2}{*}{Matrix} & \multicolumn{2}{c|}{$s$ for 5\% relative error} & \multicolumn{2}{c|}{$s$ for 1\% relative error}\\
        \cline{2-5}
        & CD-Zeroth & CD-First & CD-Zeroth &CD-First\\
         \hline
        Toeplitz \& Toeplitz & 1  & 1& 1& 1 \\
         \hline
         Symmetric \& Toeplitz & 1 & 1& 53& 1\\
         \hline
         Topelitz \& Hankel& 1 & 1 &37 & 1  \\
         \hline
        General \& Toeplitz & 1 & 1 & 66& 1 \\
         \hline
         Symmetric \& Symmetric& 1 & 1& 63& 1  \\
         \hline
         Symmetric \& Hankel&  1& 1& 59 & 1\\
         \hline
         General \& Symmetric& 1 & 1& 68 & 1\\
         \hline
          DFT Model \& DFT Model & 1 &1 & 47& 1  \\
         \hline
        Block Toeplitz \& Block Toeplitz &1 &1 & 8& 2 \\
         \hline
        Images\footnotemark[1] & 3 & 1&20 &2  \\
         \hline
         Kappa \& General& 76 & 4& -& 30 \\
         \hline
        Hankel \& Hankel& 1 & 1& 54 & 5 \\
         \hline
         General \& Hankel& 1 & 1 & 67& 8   \\
         \hline
        Kappa \& Toeplitz& 22 & 1& 54& 10  \\
         \hline
         General \& General& 1 & 1& 71& 16  \\
         \hline
         Kappa \& Kappa& 35 &11 & 66& 34  \\
         \hline
          Type-3 \& Toeplitz&-  &5 &- &  41 \\
         \hline
         Type-1 \& Toeplitz& - &5 &- & 42 \\   \hline
         Bus494 \& Bus494\footnotemark[2] &58  & 40& 62& 54  \\
         \hline
         Bus662 \& Bus662\footnotemark[3] &70  &49 & 73& 64  \\
         \hline
         Type-2 \& Type-2\footnotemark[4]& 76& 59 & -& 74 \\
         \hline
        Type-1 \& Type-1\footnotemark[4]& - &74 & -& 77 \\         
         \hline
         Type-3 \& Type-3\footnotemark[4]&-  & 74&- & 77   \\
        \hline
        LLM-1($Q_1I_1,K_1I_1,V_1I_1,Q_1K_1^T$)\footnotemark[4]&186&173&-&184\\
         \hline
         LLM-2($Q_2I_2, K_2I_2, V_2I_2,Q_2K_2^T$)\footnotemark[4]&186&173&-&184\\
         \hline
    \end{tabular}
    \caption{Results for different types of matrices of size $700 \times 700$ (see footnotes for exceptions). The average numbers $s$ reported indicate the corresponding front constants in the $\mathcal{O}(n^2 \log_2 n)$ scaling of arithmetic operations, using $\lceil s \log_2 n \rceil$ components of the Circulant Decomposition (CD). `-' indicates that it did not achieve the desired tolerance in relative error even after using $2n^3$ arithmetic operations required in the conventional multiplication. DFT - Density Functional Theory. Note that the circulant decomposition based method is agnostic to the rank of the structured matrices with periodicity in entries.}
    \label{tab:cd1}
\end{table}

\subsection{Sparsified multiplications}
In this section, we present results for multiplication of matrices that can be decomposed into a sparse matrix with fewer entries having large magnitudes, and a dense residue matrix with the smaller entries. The Fourier series based sparsification in the frequency domain can be performed in $\mathcal{O}(n^2 \log n)$ arithmetic operations, before a multiplication of the two sparse matrices and their residues. Reducing the given matrix to a sparse matrix with a few cycles is more appropriate for matrices that have dominant entries in some cycles of a matrix and not others. This approach for multiplying such matrices using the cycle decomposition are not demonstrated here for ensuring brevity. 

\begin{figure}
    \centering
    \begin{minipage}[b]{0.48\textwidth}
        \centering
        \includegraphics[width=\textwidth]{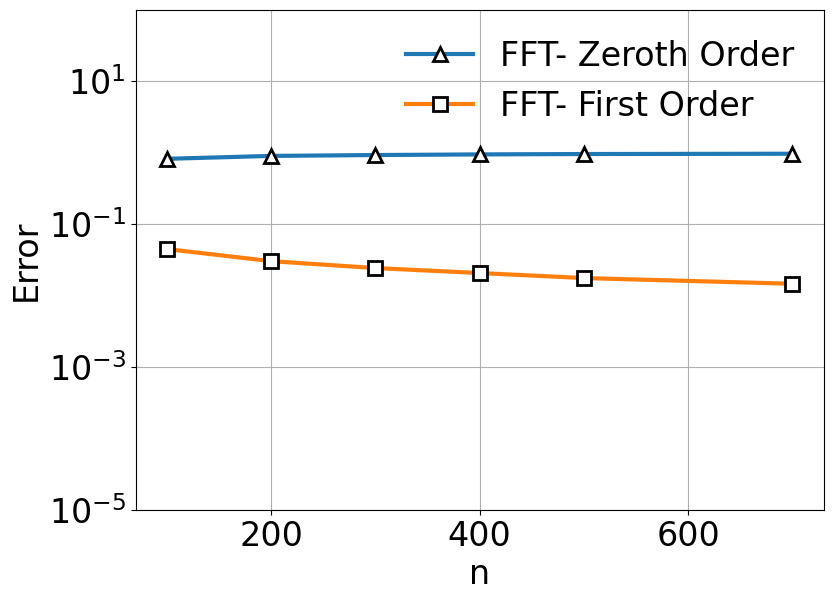}  
        \caption{Size of matrix vs. mean relative error in multiplying a general and a Toeplitz matrix with randomly generated entries from a $\mathcal{U}$(0,1) distribution. $5\lceil \log_2 n \rceil$ entries for every row/column were used in the FFT-based sparsification.} \label{fig:FFT_results1}
    \end{minipage}
    \hfill
    \begin{minipage}[b]{0.48\textwidth}
        \centering
        \includegraphics[width=\textwidth]{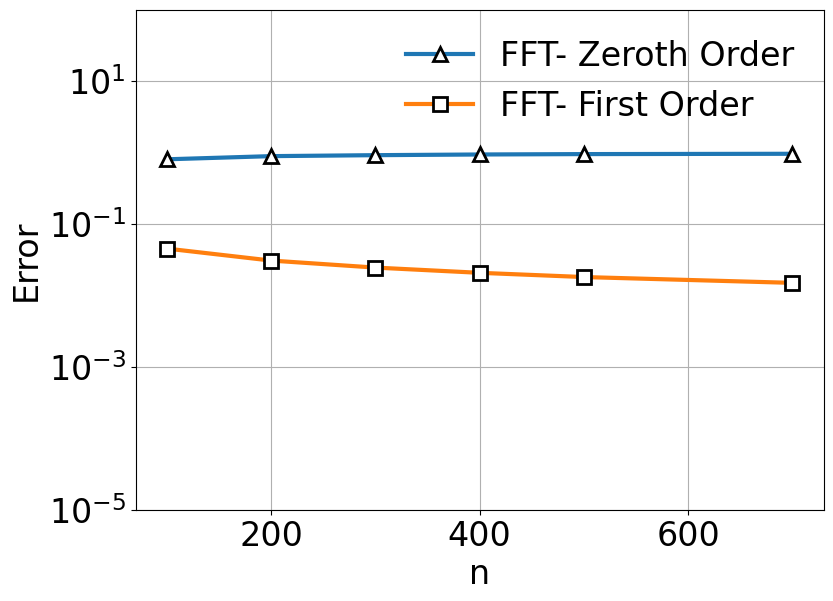}  
        \caption{Size of matrix vs. mean relative error in multiplying two matrices with randomly generated entries from a $\mathcal{U}$(0,1) distribution. $5\lceil \log_2 n \rceil$ entries for every row/column were used in the FFT-based sparsification.} \label{fig:FFT_results2}
    \end{minipage}
\end{figure}

The more generally applicable sparsification for dense matrices based on FFT on the rows of $A$ and the columns of $B$, and a multiplication of the largest entries in the transformed matrices as a first order approach, was described earlier in \cref{sec:Fourier_sparsification}. Here we note that its performance is comparable to, but not as efficient as the SVD and CD based multiplications; see \cref{tab:FFT1}, \cref{fig:FFT_results1,fig:FFT_results2}. Also, while the SVD based approach requires only efficient scaling of matrix-vector products, and the CD based approach requires optimal scaling of FFT operations, this sparsified FFT method requires further leveraging of sparse data structures in addition. Nevertheless, its proposed first-order approach may outperform the other two methods in special cases where a Fourier series based compression of the rows and columns is more effective.

\begin{table}[]
    \centering
    \begin{tabular}{|c|c|c|c|c|}
        \hline
        \multirow{2}{*}{Matrix} & \multicolumn{2}{c|}{$s$ for 5\% relative error} & \multicolumn{2}{c|}{$s$ for 1\% relative error}\\
        \cline{2-5}
        &S-FFT-Zeroth& S-FFT-First &S-FFT-Zeroth& S-FFT-First\\
         \hline
	Image \& Image & - & 1 & - &1\\
    \hline
         Symmetric \& Toeplitz & 74 & 1& -& 1\\
         \hline
         Symmetric \& Symmetric& 74 & 1& -& 1  \\
         \hline
         Symmetric \& Hankel&  74& 1& - & 1\\
         \hline
         General \& Symmetric& 74 & 1& 75 & 1\\
         \hline
         Images\footnotemark[1] & - & 1&- &1  \\
         \hline
         Bus494 \& Bus494\footnotemark[2] &42  & 2& 48& 3  \\
         \hline
         Toeplitz \& Toeplitz & 74  & 1& -& 5 \\
         \hline
         Block Toeplitz \& Block Toeplitz &74 &1 & -& 5 \\
         \hline
         Topelitz \& Hankel& 74 & 1 &- & 5  \\
         \hline
         General \& General& 74 & 1& 75 & 5  \\
         \hline
        General \& Toeplitz & 74 & 1 & 75 &6  \\
         \hline
        Hankel \& Hankel& 74 & 1& - & 6 \\
         \hline
         General \& Hankel& 74 & 1 & 75& 6   \\
         \hline
         Bus662 \& Bus662\footnotemark[3] &52  &6 & 67& 20  \\
         \hline
         Kappa \& Kappa& 73 &35 & -& 50  \\
         \hline
          Type-2 \& Type-2\footnotemark[4]& 74& 29 & -& 52 \\
         \hline
         Type-1 \& Type-1\footnotemark[4]& 71 &52 & 74& 63 \\        
         \hline
         Type-3 \& Type-3\footnotemark[4]&71  & 52&74 & 64   \\
         \hline
         Kappa \& General& 74 & 42& 75& 66 \\
         \hline
          Kappa \& Toeplitz& 74 & 43& -& 67  \\
         \hline
          DFT Model \& DFT Model & - &47 & -& 67  \\
         \hline
          Type-3 \& Toeplitz&74  &62 &- &  71 \\
         \hline
         Type-1 \& Toeplitz& 74 &63 &- & 71 \\  
	\hline
	Type-1 \& Toeplitz & 74 & 63 & - & 71 \\
	\hline
        LLM-1($Q_1I_1,K_1I_1,V_1I_1,Q_1K_1^T$)\footnotemark[4]&182&136&186&166\\
         \hline
         LLM-2($Q_2I_2, K_2I_2, V_2I_2,Q_2K_2^T$)\footnotemark[4]&182&136&186&166\\
         \hline
    \end{tabular}
    \caption{Results of a first-order multiplication based on the sparsified Fast-Fourier-Transformed matrices (S-FFT) of size $700 \times 700$ (see footnotes for exceptions). The average number of Fast Fourier Transform (FFT) components included in each row/column of matrices $A$ and $B$ is $\lceil s \log_2 n \rceil$, with $s$ indicating the corresponding front constant in the $\mathcal{O}(n^2 \log_2 n)$ scaling of arithmetic operations. `-' indicates that it did not achieve the desired tolerance in relative error, even after using $2n^3$ arithmetic operations required in the conventional multiplication. DFT - Density Functional Theory.}
    \label{tab:FFT1}
\end{table}

\subsection{Run-time experiments}
It should be noted that while we have so far detailed the arithmetic complexity involved in this approach of approximating the multiplication of matrices, other system-dependent considerations may also play a role in selecting a particular decomposition for the approximation. 
Systems can be optimized for a particular approach to effectively translate its gains in arithmetic efficiency into speed, and these may involve vectorized data and instructions in communications, and multi-core processing among others. The optimality of matrix-vector products may suffice for the partial-SVD based method, and only a fraction of it involving a reduced QR factorization is not friendly to a naive parallel computation. The circulant decomposition based method may also require optimal scaling of FFT operations in the system used. Further, the Fourier decomposition and sparsification based approach may also need the use of sparse data structures. First we present run-time results demonstrating the gains in speed over the randomized outer product algorithm known for approximating matrix multiplication, in a commonly used system. This is aimed at highlighting the potential practical gains in applications where tolerances in relative error $\sim$ 1\% are admissible.

Here we are concerned with the product of $n \times n$ matrices as examples: two Type-1 matrices, and the product of a random symmetric matrix and a Toeplitz matrix deficient in rank, as used in the examples of SVD based multiplication (\cref{tab:svd1}) of the previous subsection of the manuscript. The corresponding run-times of the randomized outer-product are compared below with the first-order SVD method in \cref{fig:randomized_vs_svd}, for a relative error tolerance of 1\%. The relative errors in both cases were evaluated using the results of an exact multiplication. But the approximation code based on the SVD decomposition also included the computing effort of an error estimation as proposed in the work, along with the effort of the partial decomposition and approximate multiplication. The proposed method can be orders of magnitude faster than the randomized outer products as the dimensions of the matrix increases.

\begin{figure}[h!]
    \centering    
    \begin{subfigure}{0.48\textwidth}
        \centering
        \includegraphics[width=\linewidth]{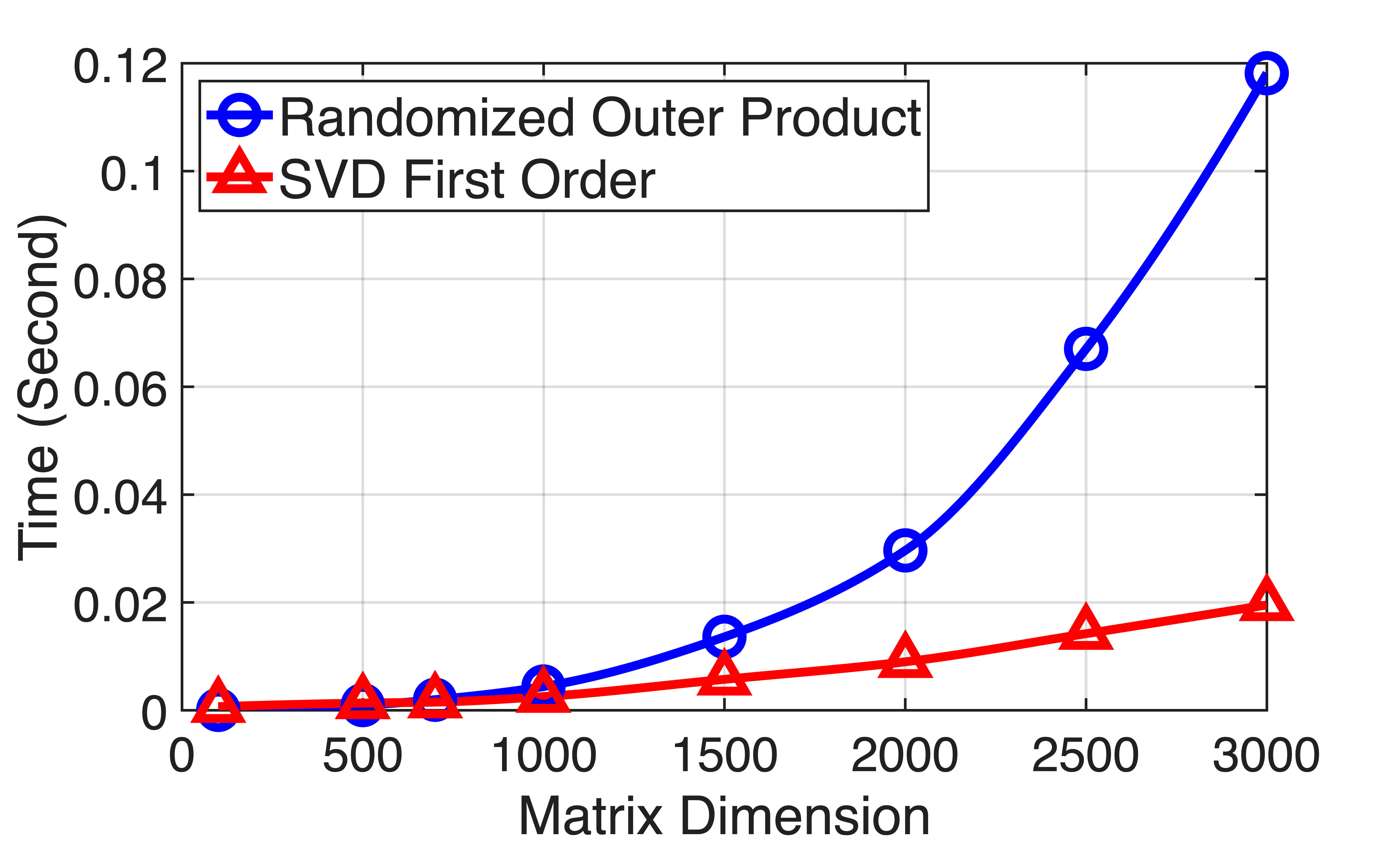}
        \caption{ Symmetric $\times$ Toeplitz   }
    \end{subfigure}
    \hfill
    \begin{subfigure}{0.48\textwidth}
        \centering
        \includegraphics[width=\linewidth]{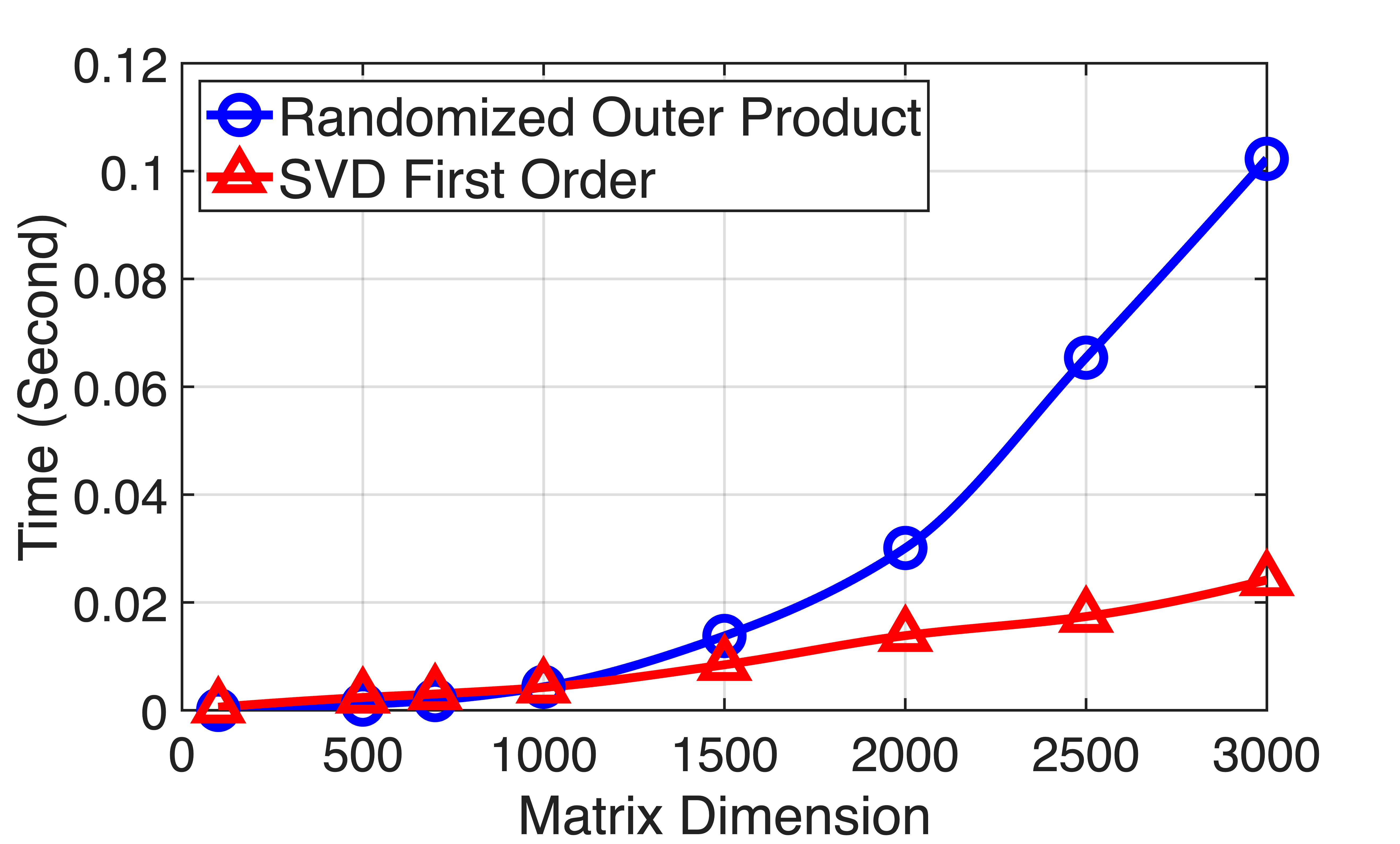}
        \caption{ Type-1 $\times$ Type-1 }
    \end{subfigure}
    \caption{Run-time comparison for Randomized Outer Product vs SVD First-Order with 1\% tolerance. Each data point represents an average over 25 trials.}
\label{fig:randomized_vs_svd}
\end{figure}

\begin{figure}[h!]
    \centering   
    \begin{subfigure}{0.48\textwidth}
        \centering
        \includegraphics[width=\linewidth]{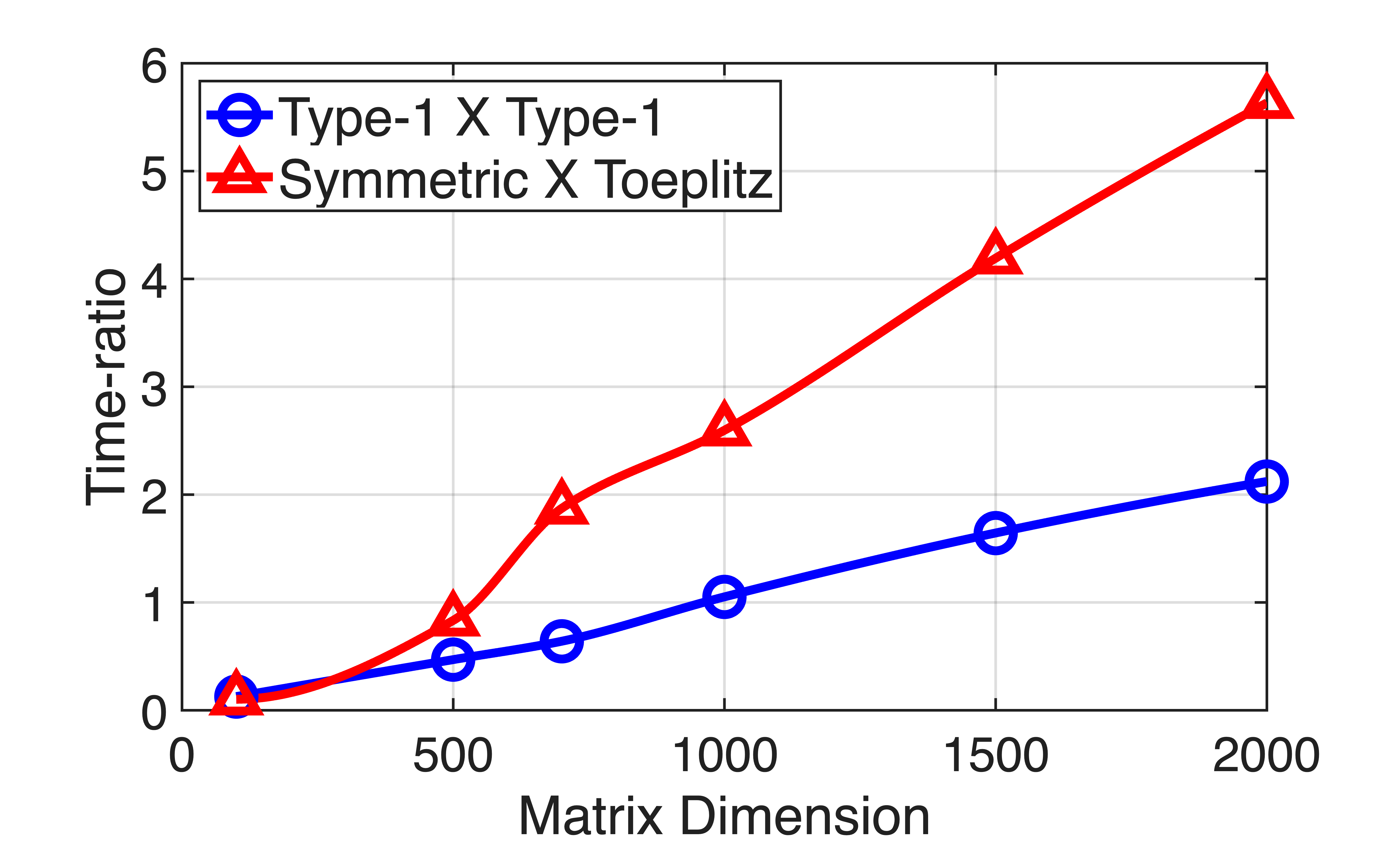}
        \caption{Randomized Outer Product over SVD First-Order}
    \end{subfigure}
    \hfill
    \begin{subfigure}{0.48\textwidth}
        \centering
        \includegraphics[width=\linewidth]{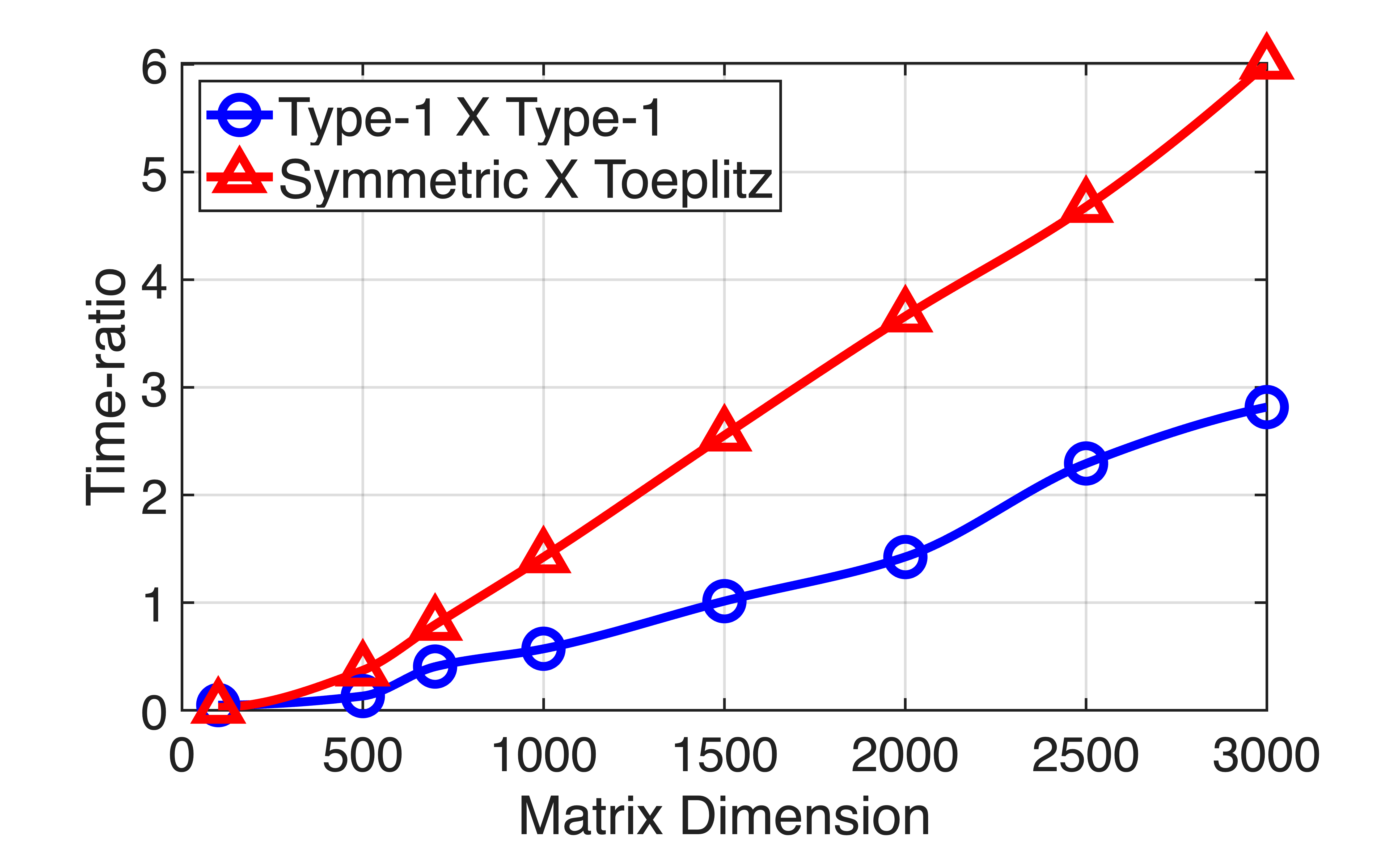}
        \caption{Naive Exact Product over SVD First-Order}
    \end{subfigure}
    \caption{Runtime-ratios of known methods over the proposed approximation with 1\% tolerance in relative error. A ratio $> 1$ implies gains due to the proposed approximation. Each data point represents an average over 25 trials.}
    \label{fig:time_ratio}
\end{figure}

Further, we are interested in the (cross-over) dimensions of matrices beyond which the proposed approximations can outperform the naive exact multiplication in arithmetic efficiency for a given tolerance in error. This can be estimated using the number of operations described earlier as $\sim n^2 s\log_2 n$ with the appropriate values of $s$ and the small front constants for the leading terms (8 for SVD and 4 for CD) and the secondary, tertiary terms in the scaling with the dimension of the matrices of the presented algorithms 2.1-2.3. The corresponding average values of $s$ are 1 and 5 respectively for the two types of matrices. It implies cross-over dimensions $\lesssim$ 500 for the above examples beyond which the proposed approximations require lower number of arithmetic operations for the given tolerance in relative error. Even using the programs in high-level languages provided by the authors \cite{Codes}, and even in settings of multi-core processors, the proposed approximations are expected to catch up and outperform the naive exact multiplication in practice as the dimensions of matrices increase. This is demonstrated\footnotemark[5] \footnotetext[5]{All experiments were performed on a MacBook Pro equipped with an Apple M4 Pro SoC featuring a 14-core CPU and 20-core GPU, 48 GB unified memory, 512 GB SSD, running macOS Tahoe 26.1 and MATLAB R2025b. All reported run-time results are averaged over 25 independent trials for each matrix dimension. Run-times were measured using MATLAB's tic/toc functions for wall-times with microsecond resolution.} for the random symmetric $\times$ Toeplitz, and the Type-1 $\times$ Type-1 multiplications in \cref{fig:time_ratio}b, using the run-time ratios of the naive exact multiplication and the SVD based approximation codes. Ratios greater than 1 imply run-time gains due to the proposed approximations, and the corresponding cross-over dimensions observed are $\sim$ 700 and $\sim$ 1500.

\subsection{Example end-to-end application: Large Language Model (LLM) prompt prefilling}
\label{subsec:llm_application}

To demonstrate the practical viability and wall-clock acceleration due to the proposed first-order approximation using truncated decompositions in a real-world matrix-heavy deployment, we apply it to the inference phase of Large Language Models (LLMs). The longer the text and its complexity (given by no. of tokens used), and/or the size of model used, the larger are the sizes of matrices and the possible gains due to the proposed approximations. Even in our demonstrations where the above sizes are limited by the computational power used by us, large gains in efficiency and speed manifest, along with the scaling up of gains with the sizes.

\subsubsection{Task description and formulation}
Modern LLMs generate text by first processing the entirety of a user's input prompt to build a contextual representation, a phase known as ``prefilling.'' When an LLM processes a large prompt of long text, the input sequence length $m$ grows significantly. The primary computational bottleneck during this phase relies on dense matrix-matrix multiplications between the dynamic input activation matrix $A \in \R^{m \times d_{in}}$ and the static layer weights $B \in \R^{d_{in} \times d_{out}}$ (specifically within the Query, Key, Value projections and the Feed-Forward Networks). For exact multiplication, this scales as $\mathcal{O}(m \cdot d_{in} \cdot d_{out})$, representing an $\mathcal{O}(n^3)$ bottleneck as dimension $n$ scales. Given $A$ and $B$, the proposed first-order approximation $M \approx \sum_{i} A_i B + \sum_{k} \Delta A B_k$ is executed as:
$$M \approx (U_A S_A) (V_A^T B) + (\Delta A U_B) (S_B V_B^T)$$
In this deployment setting, the truncated Singular Value Decomposition (SVD) of the static weights ($U_B, S_B, V_B^T$) is precomputed offline using the randomized partial SVD algorithm, as it need not be re-computed for each inference\cite{Codes}. 

Benchmarking linear algebraic approximations on modern hardware presents a unique challenge in differentiating the gains in efficiency and the gains in speed: modern processors (including GPUs and Apple Silicon SoCs) feature highly specialized silicon (e.g., Tensor Cores or AMX coprocessors) and highly parallelized Level-3 BLAS libraries that are explicitly hardwired to accelerate the naive exact $\mathcal{O}(n^3)$ dense matrix multiplications. All LLM inference benchmarks were executed using Python 3.11.15 to leverage the specialized adaptive interpreter, ensuring that Python-level execution overhead did not artificially mask the mathematical reduction in FLOPs. The environment utilized PyTorch 2.10 and Transformers 5.3.0, compiled natively for the ARM64 architecture to interface optimally with the hardware's native linear algebra backend. To guarantee exact reproducibility and hardware-agnostic algorithmic fairness, all multi-threading capabilities were disabled (\texttt{torch.set\_num\_threads(1)}), forcing a strictly sequential execution path for both the exact baseline and the proposed approximation. By stripping away multi-core hardware optimizations, the measured wall-clock execution time serves as a direct proxy for the total arithmetic complexity (FLOPs) and the efficiency of the computation. In later experiments, multithreading with 14 CPU cores were enabled to demonstrate the gains in speed that are relatively unaffected by the parallel computing architecture. Experiments were executed on an Apple M4 Pro SoC with 48GB unified memory.

\subsubsection{Results and Observations}
The proposed approximations were evaluated across two model scales: OPT-125M (inner dimensions $d \in \{768, 3072\}$) and OPT-1.3B (inner dimensions $d \in \{2048, 8192\}$). The models were tasked with a prefill workload resulting in an activation matrix row dimension of $m \in \{28815, 44183\}$. The exact multiplication in the linear projection layers, was replaced with the first-order approximation using truncated decompositions, parametrized by the front factor $s$ in the $s\log_2 n$ components used. We measured the end-to-end forward pass latency (ms) for the gains, and the model perplexity as a standard metric for predictive accuracy where a lower value indicates higher confidence and stability. To ensure unbiased latency measurements, all timed evaluations were preceded by a full forward pass to account for memory allocation and CPU cache-warming overheads.

    \begin{figure}[htbp]
     \centering
     \begin{subfigure}[b]{0.48\textwidth}
         \centering
         \includegraphics[width=\textwidth]{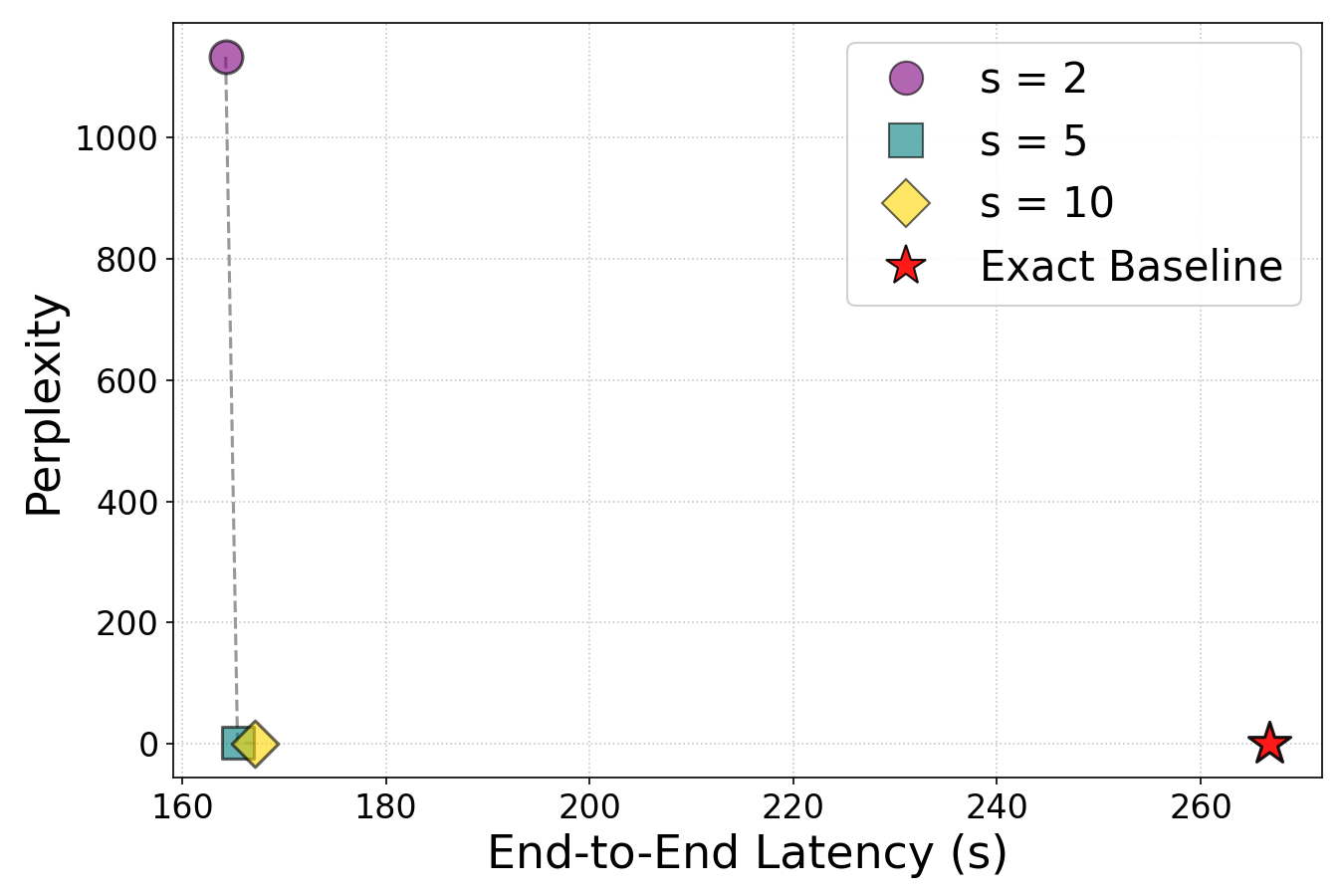}
         \caption{Results for OPT-125M}
     \end{subfigure}
     \hfill 
     \begin{subfigure}[b]{0.48\textwidth}
         \centering
         \includegraphics[width=\textwidth]{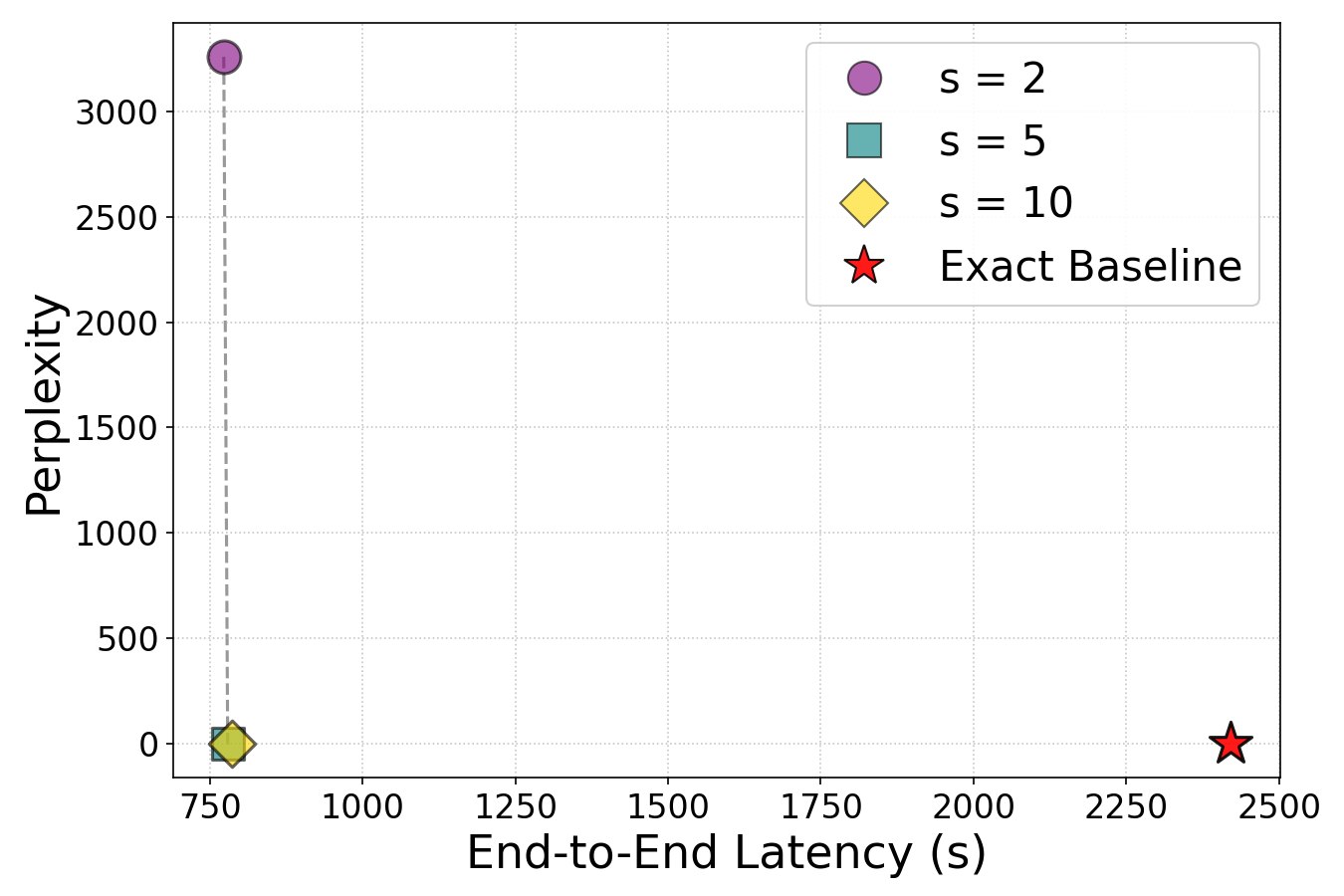}
         \caption{Results for OPT-1.3B}
   \end{subfigure}
     
     \caption{Comparison of the first-order approximations and the baseline across different model scales.}
     \label{fig:llm_test}
\end{figure}

\begin{figure}[htbp]
     \centering
     \begin{subfigure}[b]{0.48\textwidth}
         \centering
         \includegraphics[width=\textwidth]{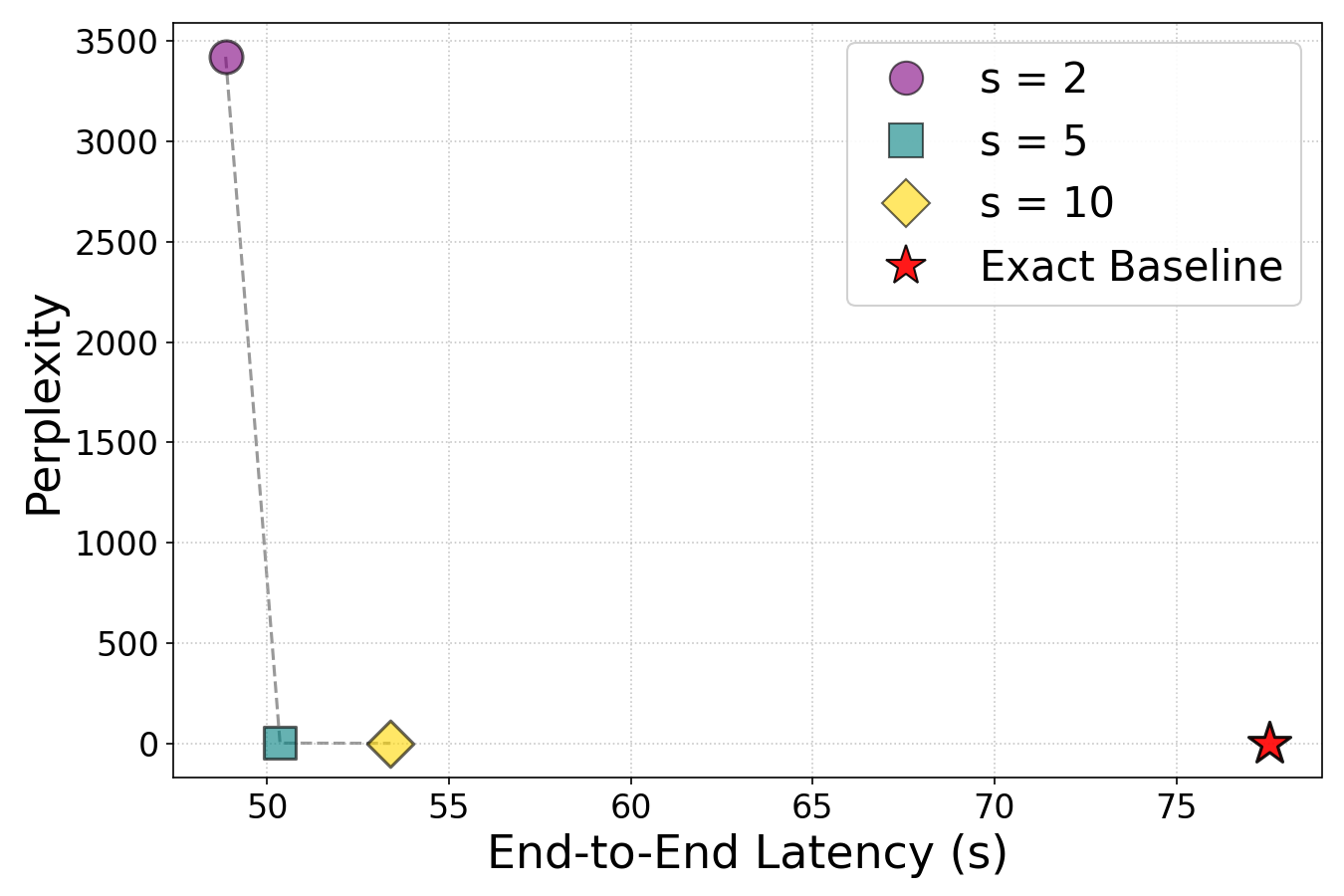}
         \caption{Results for OPT-125M}
     \end{subfigure}
     \hfill 
     \begin{subfigure}[b]{0.48\textwidth}
         \centering
         \includegraphics[width=\textwidth]{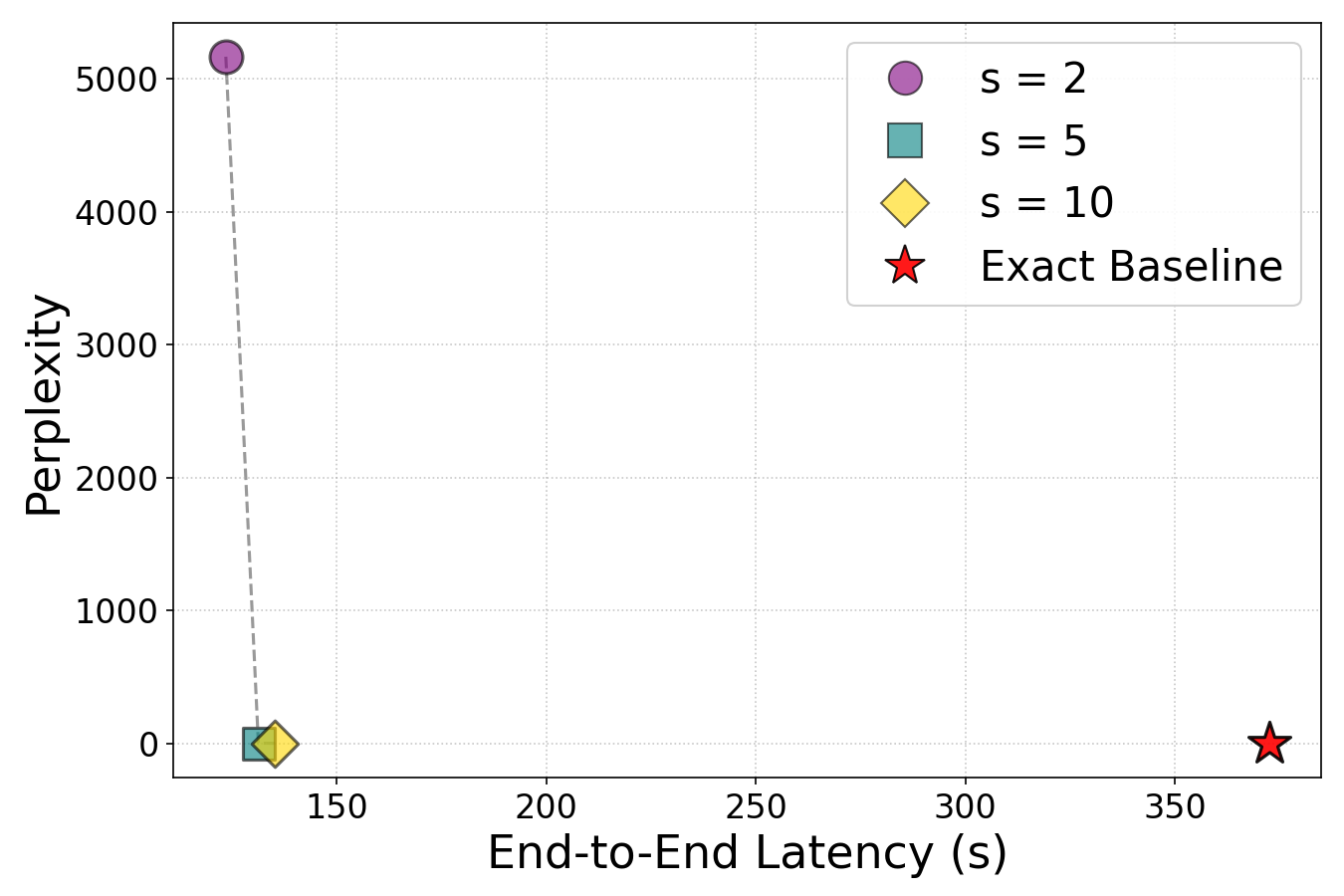}
         \caption{Results for OPT-1.3B}
   \end{subfigure}
     
     \caption{Comparison of the first-order approximations and the baseline across different model scales with multithreading using 14 CPU cores.}
     \label{fig:llm_test_multithread}
\end{figure}

\begin{table}[ht]
    \centering
    \renewcommand{\arraystretch}{1.5} 
    \begin{tabular}{l c c c c c}
        \toprule
        \textbf{Model} & \textbf{Mat} & \boldmath{$(T_B, P_B)$} & \boldmath{$(T_{s=2}, P_{s=2})$} & \boldmath{$(T_{s=5}, P_{s=5})$} & \boldmath{$(T_{s=10}, P_{s=10})$} \\
        \midrule
        OPT-125M & \makecell{(44183,3072)\\(44183,768)\\(3072,768)\\(768,768)} & \makecell{266.73552 s\\1.0658} & \makecell{164.27854 s\\1132.8163} & \makecell{165.44328 s\\1.5877} & \makecell{167.12624 s\\1.0596} \\
        \midrule
        OPT-125M$^*$ & \makecell{(44183,3072)\\(44183,768)\\(3072,768)\\(768,768)} & \makecell{77.55779 s\\1.0658} & \makecell{48.86126 s\\3421.8369} & \makecell{50.35627 s\\2.0283} & \makecell{53.39471 s\\1.0617} \\
        \midrule
        OPT-1.3B & \makecell{(28815,8192)\\(28815,2048)\\(8192,2048)\\(2048,2048)} & \makecell{2419.79320 s \\1.0509 } & \makecell{772.90297 s \\ 3260.6182} & \makecell{779.30732 s \\ 1.4873 } & \makecell{786.97144 s \\1.0419 } \\
        \midrule
        OPT-1.3B$^*$ & \makecell{(28815,8192)\\(28815,2048)\\(8192,2048)\\(2048,2048)} & \makecell{372.55008 s \\1.0509 } & \makecell{123.47397 s \\ 5164.7749} & \makecell{131.24341 s \\ 1.4869 } & \makecell{135.16379 s \\1.0423 } \\
        \bottomrule
    \end{tabular}
    \caption{Results for lower perplexity baselines in the end-to-end LLM application. $Mat$: Dimensions of matrices for input activation and static weights. $T_B$: Latency in baseline (in seconds). $P_B$: Perplexity in baseline. $T_{s=r}$: Latency (in seconds) with the SVD-based approximation algorithm when number of utilized components $k=r \log_2 n$, and $P_{s=r}$ denotes the corresponding perplexity. ${}^*$Multithreading with 14 CPU cores was allowed.}
    \label{tab:llm_test}
\end{table}

A fraction of the overall latency may represent other background processes that is agnostic to the matrix multiplication effort, but a equally significant fraction of the overall latency is indeed based on the multiplication, as indicated by the gains over the baseline model with a naive exact multiplication. As detailed in \cref{tab:llm_test} and shown in \cref{fig:llm_test} and \cref{fig:llm_test_multithread}, the first-order approximation establishes a clear Pareto front between computational speed and model accuracy:
\begin{itemize}
    \item \textit{Extreme Compression ($s=2$):} The method yields the fastest execution time, aggressively outperforming the exact baseline. However, capturing too few components from the LLM's singular value spectrum results in a measurable degradation in perplexity.
    \item \textit{The Optimal Sweet Spot ($s=5$):} At this operating point, the first-order approximation captures the necessary dominant components while recovering the discarded subspace via the residue matrices ($\DA$ and $\DB$) used in the first-order approach. Here, the model recovers near-lossless accuracy by converging to the baseline perplexity while still maintaining a large wall-clock speed-up.
    \item \textit{Diminishing Returns ($s=10$):} Increasing the components further maintains the baseline perplexity but increases the computational overhead, marginally eroding the latency gains.
    \item Occasionally, the approximated model may achieve a slightly lower (and better) perplexity than the exact baseline. This phenomenon occurs as truncation of the smallest singular values inherently filters out high-frequency noise within the matrices, acting as a form of low-rank regularization that can marginally improve the model's generalization on specific inputs.
\end{itemize}

Notably, the relative speed-up achieved by the proposed approximation increases significantly in the OPT-1.3B model compared to the OPT-125M model. With activation matrices of sizes $\{(44183,3072),(44183,768),(3072,768),(768,768)\}$ and 125 million parameters as in a small LLM model we see end-to-end gains in efficiency by a factor of $\approx 1.59X$. This end-to-end gain increases to a factor of $\approx 3.07X$ for 1.3 billion parameters in the OPT-1.3B LLM as shown. This empirical result aligns with the theoretical arithmetic complexity: as the inner matrix dimensions scale up, the $\mathcal{O}(n^2 \log n)$ approximation diverges further from the $\mathcal{O}(n^3)$ exact baseline as expected. As the number of parameters increases to say 1 trillion, the end-to-end gains in the LLM can be more than 6x depending on the tasks and systems used. We also present similar results and gains with text of higher complexity drawn from this manuscript in \cref{tab:llm_test_varied_text}, showing that the gains are relatively agnostic to the nature of text and corresponding baseline perplexity values. These results confirm that the proposed first-order approximation can reliably accelerate massive linear algebraic workloads in such learning architectures without compromising the integrity of the model. The corresponding reduction in power used due to these large reductions in computational effort and latency, are thus very significant as well.

\begin{table}[ht]
    \centering
    \renewcommand{\arraystretch}{1.5} 
\begin{tabular}{|c|c|c|cccccc|}
\hline
\textbf{Model} & \textbf{Mat} &  \boldmath{$(T_B,P_B)$ }& \multicolumn{6}{c|}{\boldmath{$(T_s,P_s)$} }\\
\hline
 &  &   & s=5 & s=10 & s=20 & s=30 & s=40 & s=50 \\
 \hline
OPT-125M & \makecell{53250,3072)\\(53250,768)\\(3072,768)\\(768,768)} & \makecell{125.17 s\\36.475} &  \makecell{62.13 s\\9866.6680} & \makecell{65.91 s\\1938.7532} & \makecell{69.59 s\\206.0167} & \makecell{73.10 s\\75.9855} & \makecell{82.34 s\\37.117} &\makecell{90.37 s\\36.2323}\\
\hline
OPT-1.3B & \makecell{(44375,8192)\\(44375,2048)\\(8192,2048)\\(2048,2048)} &\makecell{603.68 s\\23.3731}&\makecell{224.53 s\\10375.4541}&\makecell{229.05 s\\8423.0947}&\makecell{241.09 s\\233.1174}&\makecell{256.49 s\\77.0377}&\makecell{274.58 s\\33.098}&\makecell{292.99 s\\24.3073}\\
\hline
\end{tabular}
\caption{Results for higher perplexity baselines using text from a part of this manuscript in the end-to-end LLM application under a multithreaded (14-core) execution setting. $Mat$: Dimensions of matrices for input activation and static weights. $T_B$: Latency in baseline (in seconds). $P_B$: Perplexity in baseline. $T_{s=r}$: Latency (in seconds) with the SVD-based approximation algorithm when number of utilized components $k=r \log_2 n$, and $P_{s=r}$ denotes the corresponding perplexity.}
\label{tab:llm_test_varied_text}
\end{table}

\section{Conclusion}
A general approach to approximate multiplication that is suited for large dense matrices in practice, was presented using truncated decompositions and a first-order multiplication of the components. A few relevant decompositions of matrices suited for this objective were introduced. Efficient a priori evaluations for determining the suitability of a decomposition for the given matrices, as well as its relative errors in the multiplication, using the truncation errors in the decompositions was suggested. Theoretical results for estimation of relative error in the multiplication were presented in the analysis section. The presented numerical results demonstrate the utility of this approach even for reasonably low tolerances in relative error $\sim$ 1\%. We began by showing the gains in arithmetic efficiency of the proposed method over the naive exact multiplication for a given tolerance in relative error. Then we demonstrated gains in the run-times compared to a naive exact multiplication, and also the increase in speed and efficiency of end-to-end Large-Language-Model (LLM) operations, thus highlighting the large gains in practice.

\vspace{3mm}

\appendix

\section{Additional methods}

\subsection{Randomized low-rank multiplication}\label{sec:low_rank_algorithm}
The low-rank approximation applied to matrix multiplication $C=AB$, used to generate results presented in \cref{sec:results_truncated_SVD} is described here. The method considers the required product as a sum of outer products of columns from matrix $A$ and corresponding rows from matrix $B$ i.e. $C = \sum_{k = 1}^{n} A^{(k)} B_{(k)}$. Columns of $A$ numbered by positive integers $k$, and the corresponding rows of $B$ are randomly selected, totalling $c$ columns and rows in number \cite{matrixmultiply2006Kannan}. The $k^{th}$ column of $A$ indicated by a superscript (or a row of $B$ indicated by its subscript) is selected using rejection sampling with a probability $p_k$, based on its 2-norm.

\begin{equation*}
p_k=\frac{\norm{A^{(k)}}_2\norm{B_{(k)}}_2}{\sum_n \norm{A^{(k)}}_2\norm{B_{(k)}}_2}
\end{equation*}

Thus sampled outer-products are then summed as correspondingly up-weighted matrices as described in step 12 of $\cref{algo:matrix_mult}$, where the smaller outer-products tend to be ignored:

\begin{algorithm}
\caption{Randomized Low Rank Approximation of Matrix Multiplication}\label{algo:matrix_mult}
\begin{algorithmic}[1]
    \State \textbf{Input:} Columns and rows $A^{(k)}$, $B_{(k)}$ for each $k$, probability mass function $p_k$, number of samples $c$
    \State \textbf{Output:} Approximation matrix $M$ and relative error $E$
    \vspace{0.5em}
    
    \State \textbf{Initialization:}
    \State Initialize $M \gets 0$ (Zero matrix)
    \State Initialize $t \gets 0$ (Counter for accepted samples)
    
    \While{$t < c$}
        \State Sample index $k$ uniformly from $\{1, 2, \hdots, n\}$
        \State Generate uniform random number $U \in [0, 1]$
        \If{$U \cdot \max(p_k) < p_k$}
            \State Accept $k$
            \State Increment $t \gets t + 1$
            \State Update $M \gets M + \frac{A^{(k)} B_{(k)}}{c p_k}$
        \EndIf
    \EndWhile

    \State \textbf{Error Evaluation:}
    \State Compute relative error $E = \frac{\|M - AB\|_F}{\|AB\|_F}$
    
    \State \Return $M$, $E$
\end{algorithmic}
\end{algorithm}

\subsection{Front constants of additional distributions for estimating norm of the product of residue matrices}
Numerical trials were conducted for different distributions that may be useful in practical settings to estimate $\|\Delta A.\Delta B\|_F = c \|\Delta A\|_F\|\Delta B\|_F$. 25 trials each were preformed across 10 varying values of $n$ from 50 to 5000. Random matrices with entries drawn from the specified distributions (Uniform, Rademacher, Gaussian, Log-Normal, Student's $t_\nu$) were generated using rng(42) and MATLAB R2025b. Nearly identical values were confirmed for all $n$, for every distribution. The values of $c$ reported below in \cref{tab:front_constants} are for n=5000. Note that the largest possible norm of the product is bound by the Cauchy-Schwarz inequality $\norm{\Delta A.\Delta B}_F \leq \norm{\Delta A}_F \norm{\Delta B}_F$ and thus $c$ can be as large as 1. Hence it is advisable to use values of $c$ not much less than 1 even when the distribution of entries is nominally known, for a conservative estimate in practical algorithms.
        \begin{table}
        \centering
        \begin{tabular}{|c|c|c|}
        \hline
        \textbf{Distribution} & value of c & $\sigma$ at n=5000\\ \hline
        Uniform U(0,1)    &  0.750067 &  0.000050  \\ \hline
        Rademacher (±1)    & $ \frac{0.999997}{\sqrt{n}}$  & 0.000002  \\ \hline
        Normal (0,1)    & $ \frac{1.000019}{\sqrt{n}}$ & 0.000001     \\ \hline
        Log-Normal[exp(N(0,1))]    & $ 0.368169 \approx \frac{1}{e}$ & 0.000289     \\ \hline
        $t_v$(v=3)    & $ \frac{1.000222}{\sqrt{n}}$ & 0.000025    \\ \hline
        \end{tabular}
        \caption{Expected values of $\frac{\|\Delta A.\Delta B\|_F}{\|\Delta A\|_F\|\Delta B\|_F}$ estimated numerically. For $X \sim$Log-Normal distribution, using $\mathbb{E}[X]=e^{\frac{1}{2}}$, and $\mathbb{E}[X^2]=e^2$, the exact theoretical value for c is $\frac{1}{e}$.} \label{tab:front_constants}      
        \end{table}

\subsection{Singular values of example matrices} \label{sec:plots_SVD}

Distributions of the singular values for a few example matrices, including ones in \cref{tab:svd1}, are presented in \cref{fig:example_SVD_components}.
 \begin{figure}[ht]
    \centering
    \begin{minipage}[b]{0.3\textwidth}
        \centering
        \includegraphics[width=\textwidth]{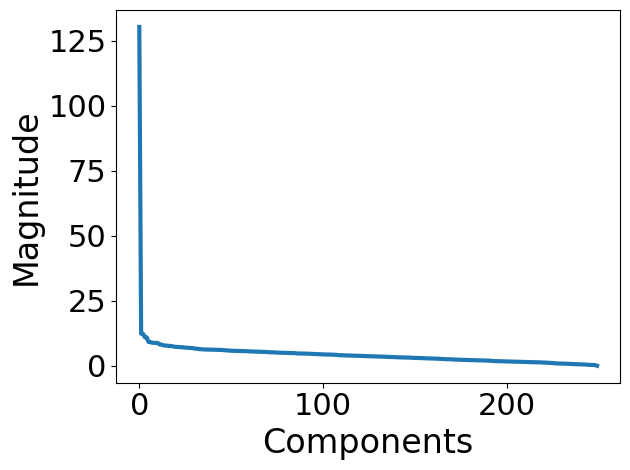}
        \caption*{Toeplitz}
    \end{minipage}%
    \begin{minipage}[b]{0.3\textwidth}
        \centering
        \includegraphics[width=\textwidth]{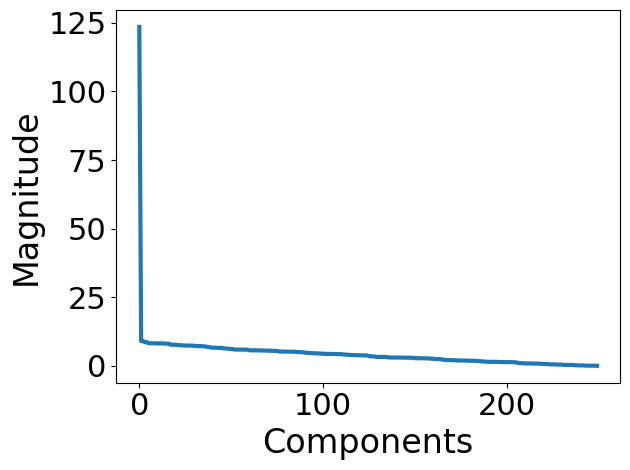}
        \caption*{Block Toeplitz}
    \end{minipage}%
    \begin{minipage}[b]{0.3\textwidth}
        \centering
        \includegraphics[width=\textwidth]{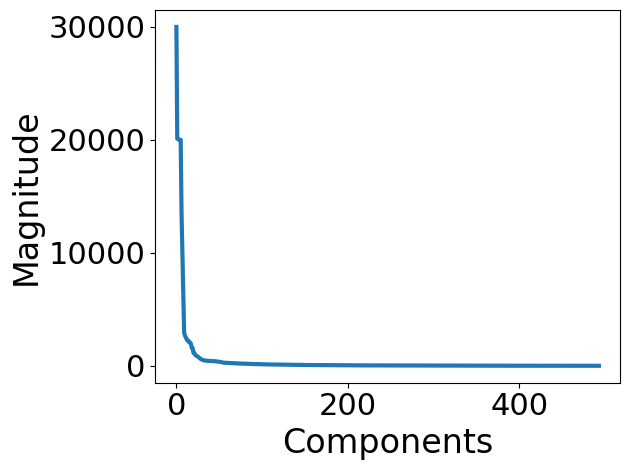}
        \caption*{Bus494}
    \end{minipage}
    
    \vskip 0.005cm 
    
    \begin{minipage}[b]{0.3\textwidth}
        \centering
        \includegraphics[width=\textwidth]{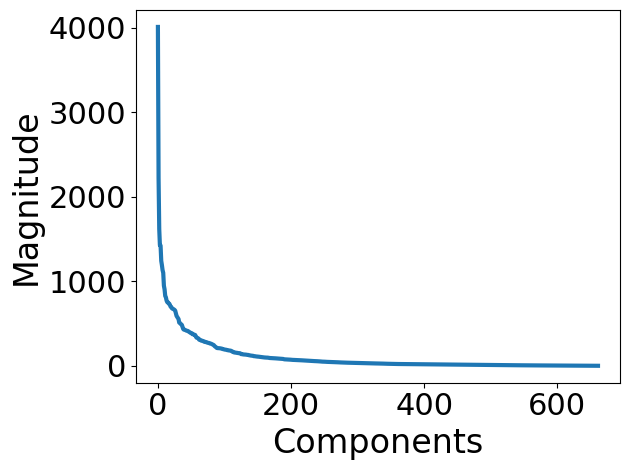}
        \caption*{Bus662}
    \end{minipage}%
    \begin{minipage}[b]{0.3\textwidth}
        \centering
        \includegraphics[width=\textwidth]{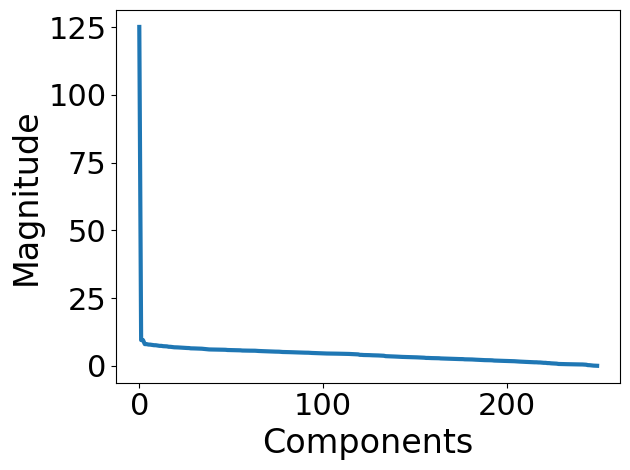}
        \caption*{Hankel}
    \end{minipage}%
    \begin{minipage}[b]{0.3\textwidth}
        \centering
        \includegraphics[width=\textwidth]{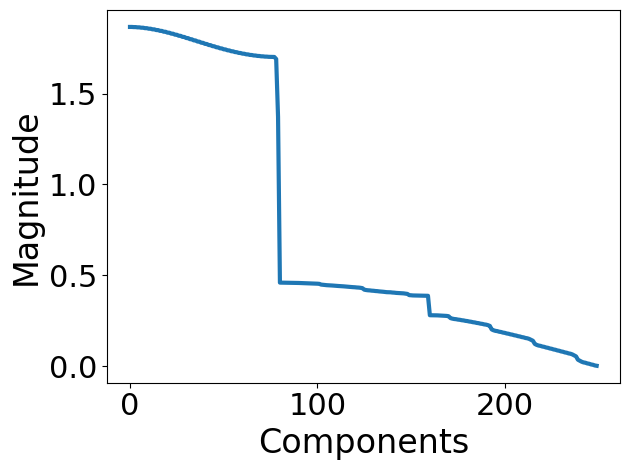}
        \caption*{Kappa matrix}
    \end{minipage}

    \vskip 0.005cm 
    
    \begin{minipage}[b]{0.3\textwidth}
        \centering
        \includegraphics[width=\textwidth]{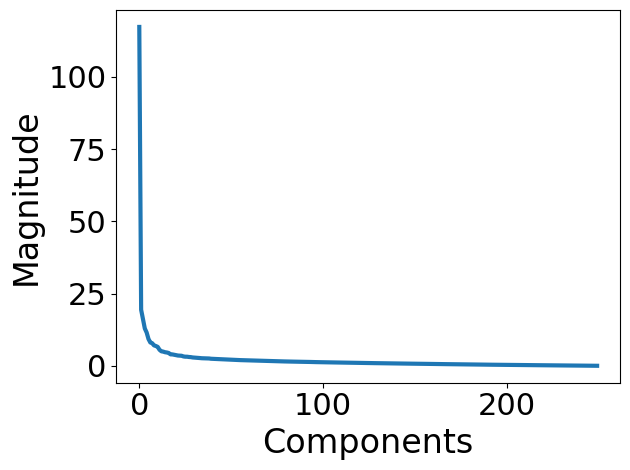}
        \caption*{Dog image}
    \end{minipage}%
    \begin{minipage}[b]{0.3\textwidth}
        \centering
        \includegraphics[width=\textwidth]{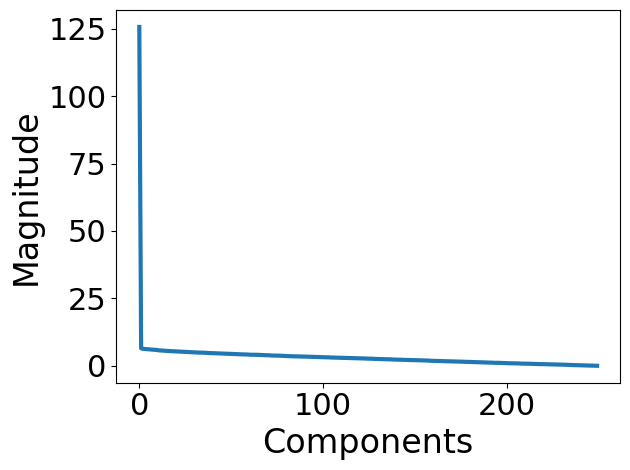}
        \caption*{Symmetric Random}
    \end{minipage}%
    \begin{minipage}[b]{0.3\textwidth}
        \centering
        \includegraphics[width=\textwidth]{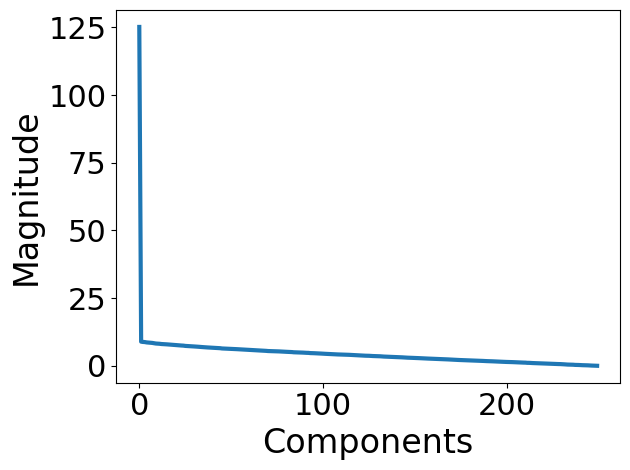}
        \caption*{General Random}
    \end{minipage}%

    \vskip 0.005cm 

    \begin{minipage}[b]{0.3\textwidth}
        \centering
        \includegraphics[width=\textwidth]{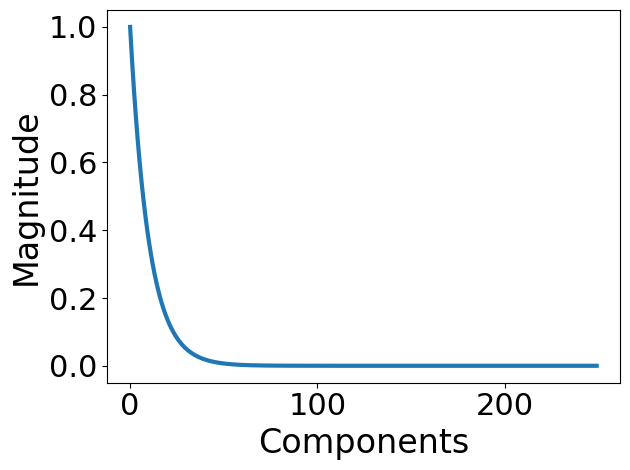}
        \caption*{Type-1}
    \end{minipage}%
    \begin{minipage}[b]{0.3\textwidth}
        \centering
        \includegraphics[width=\textwidth]{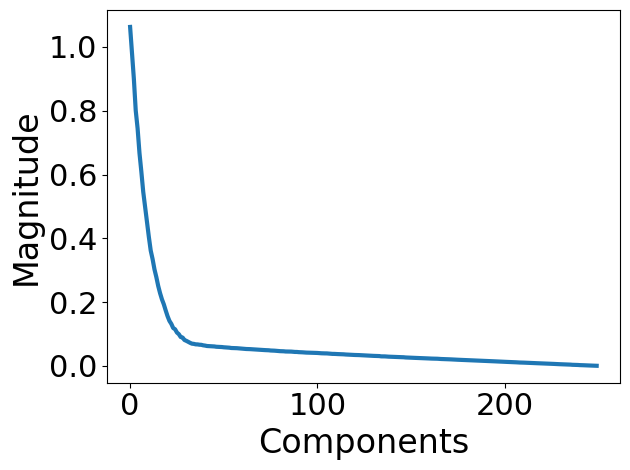}
        \caption*{Type-2}
    \end{minipage}%
    \begin{minipage}[b]{0.3\textwidth}
        \centering
        \includegraphics[width=\textwidth]{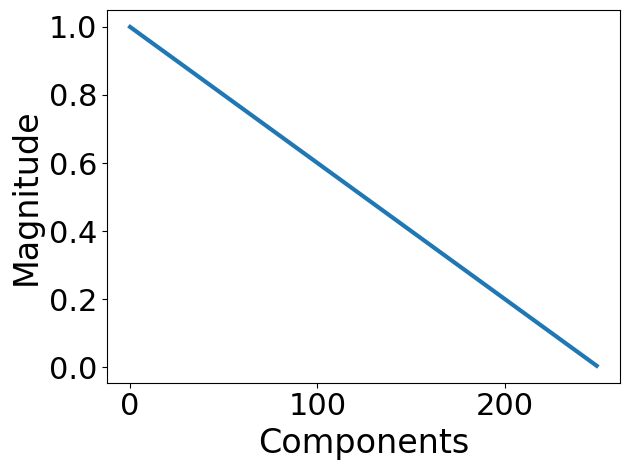}
        \caption*{Type-3}
    \end{minipage}

    \caption{Distribution of singular values for different types of matrices. The required set of entries for the Toeplitz, block Toeplitz, and Hankel matrices was generated randomly from a $\mathcal{U}(0,1)$ distribution. The Kappa matrix K was generated using the relation $K_{ij}=K_{ji}=e^{(-0.5|i-j|)}\times sin(i+1)$ for the entries.}\label{fig:example_SVD_components}
\end{figure}

\subsection{Circulant components of example matrices}\label{sec:plots_CD}
Distributions of the magnitudes of circulant components for a few example matrices, including ones in \cref{tab:cd1}, are presented in \cref{fig:example_CD_components}. The magnitude of a circulant component i.e. the matrix $R_k$, is given by the 2-norm of its first column $\|R_k^{(0)}\|_2$.

 \begin{figure}[ht]
    \centering
    \begin{minipage}[b]{0.3\textwidth}
        \centering
        \includegraphics[width=\textwidth]{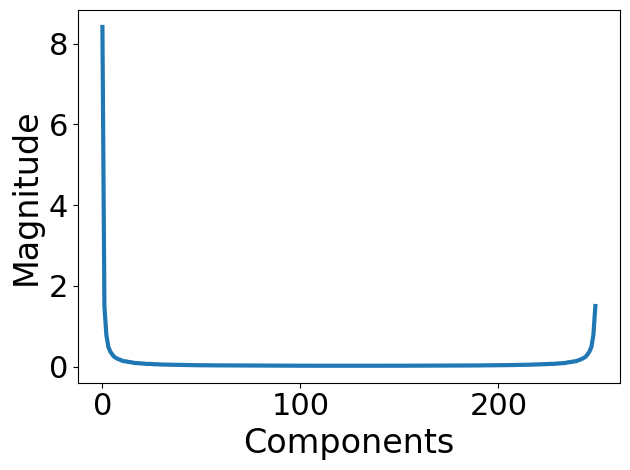}
        \caption*{Toeplitz}
    \end{minipage}%
    \begin{minipage}[b]{0.3\textwidth}
        \centering
        \includegraphics[width=\textwidth]{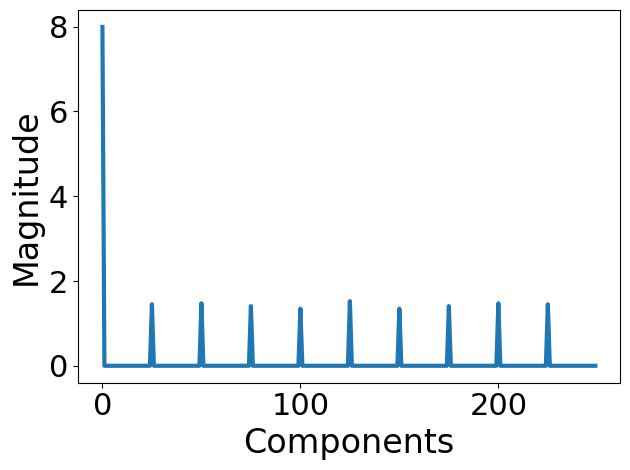}
        \caption*{Block Toeplitz}
    \end{minipage}%
    \begin{minipage}[b]{0.3\textwidth}
        \centering
        \includegraphics[width=\textwidth]{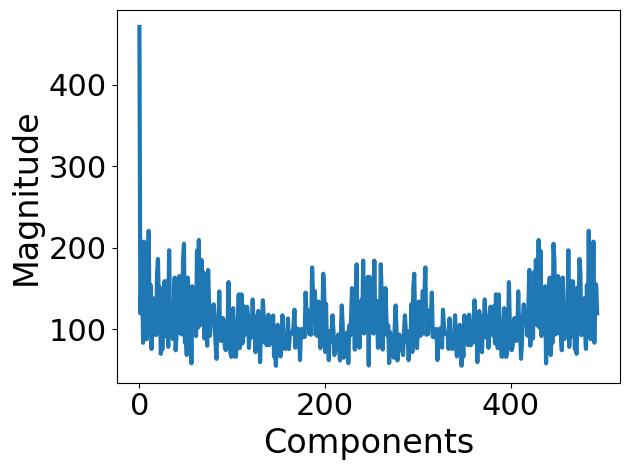}
        \caption*{Bus494}
    \end{minipage}
    
    \vskip 0.005cm 
    
    \begin{minipage}[b]{0.3\textwidth}
        \centering
        \includegraphics[width=\textwidth]{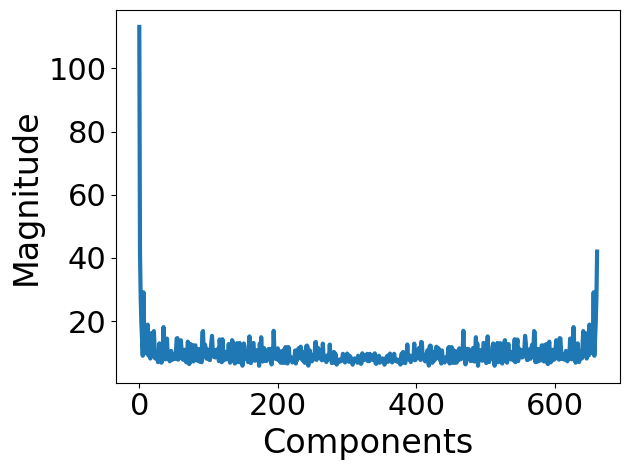}
        \caption*{Bus662}
    \end{minipage}%
    \begin{minipage}[b]{0.3\textwidth}
        \centering
        \includegraphics[width=\textwidth]{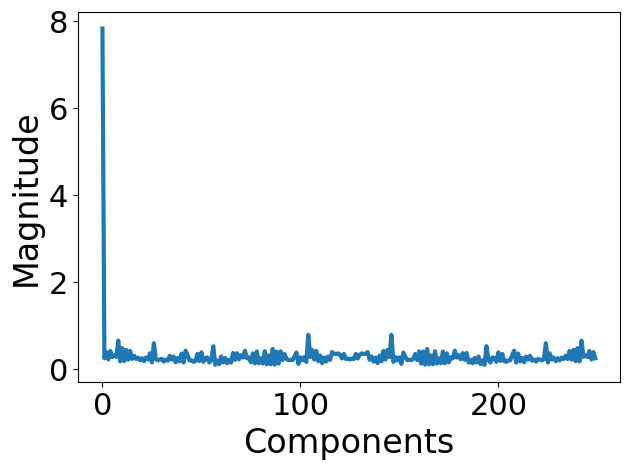}
        \caption*{Hankel}
    \end{minipage}%
    \begin{minipage}[b]{0.3\textwidth}
        \centering
        \includegraphics[width=\textwidth]{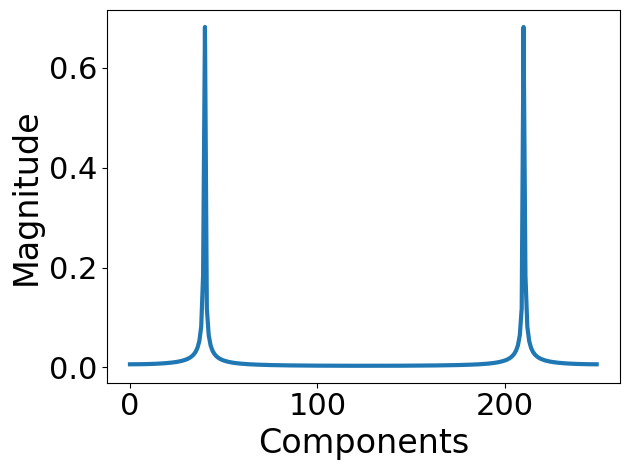}
        \caption*{Kappa matrix}
    \end{minipage}

    \vskip 0.005cm 
    
    \begin{minipage}[b]{0.3\textwidth}
        \centering
        \includegraphics[width=\textwidth]{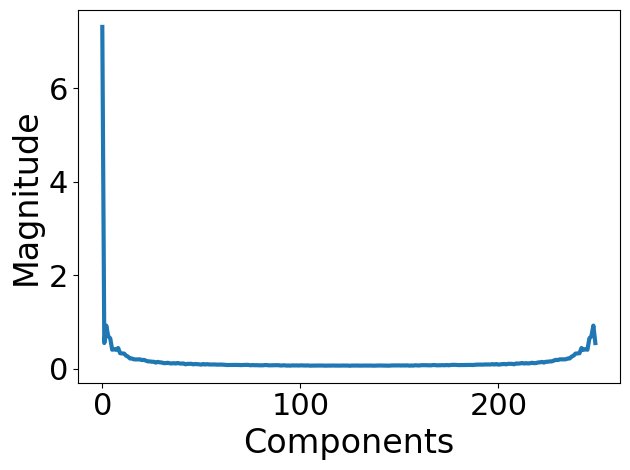}
        \caption*{Dog image}
    \end{minipage}%
    \begin{minipage}[b]{0.3\textwidth}
        \centering
        \includegraphics[width=\textwidth]{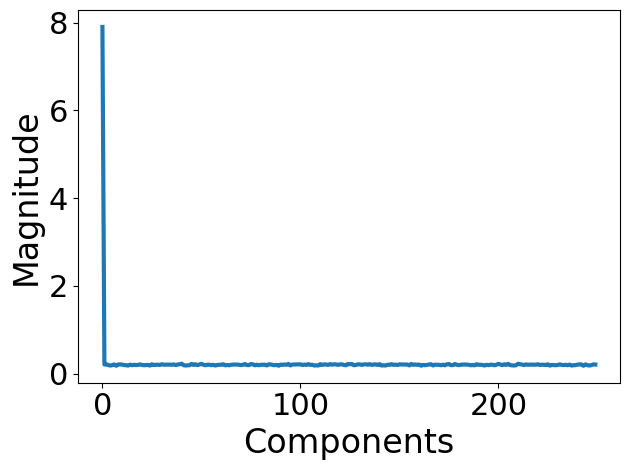}
        \caption*{Symmetric Random}
    \end{minipage}%
    \begin{minipage}[b]{0.3\textwidth}
        \centering
        \includegraphics[width=\textwidth]{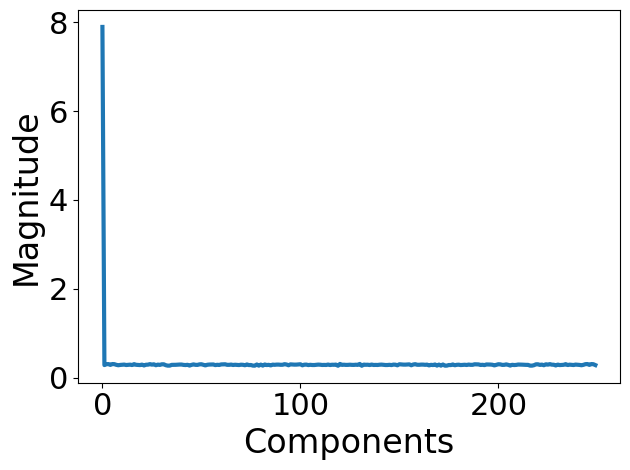}
        \caption*{General Random}
    \end{minipage}%
    
    \vskip 0.005cm 

    \begin{minipage}[b]{0.3\textwidth}
        \centering
        \includegraphics[width=\textwidth]{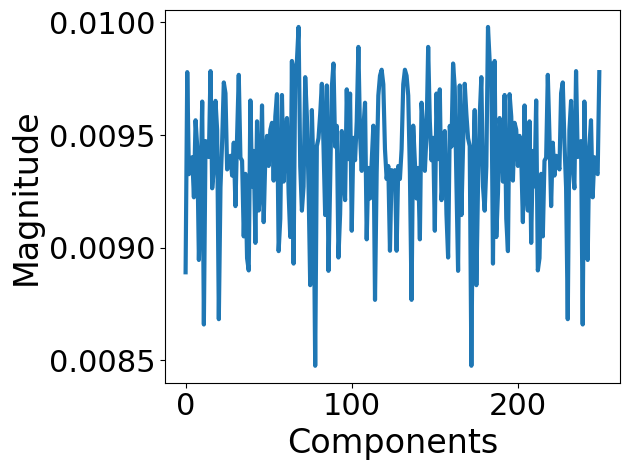}
        \caption*{Type-1}
    \end{minipage}%
    \begin{minipage}[b]{0.3\textwidth}
        \centering
        \includegraphics[width=\textwidth]{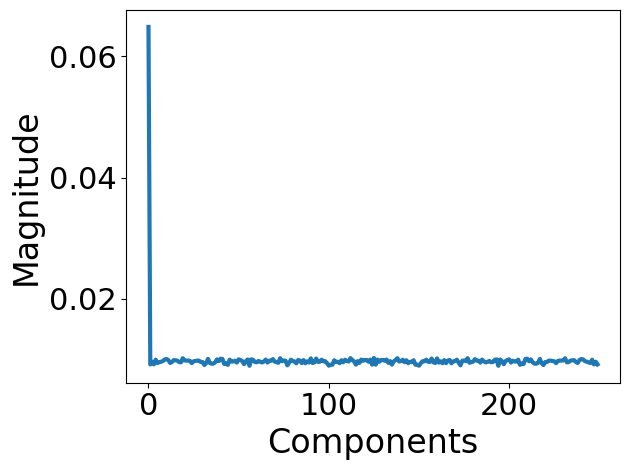}
        \caption*{Type-2}
    \end{minipage}
    \begin{minipage}[b]{0.3\textwidth}
        \centering
        \includegraphics[width=\textwidth]{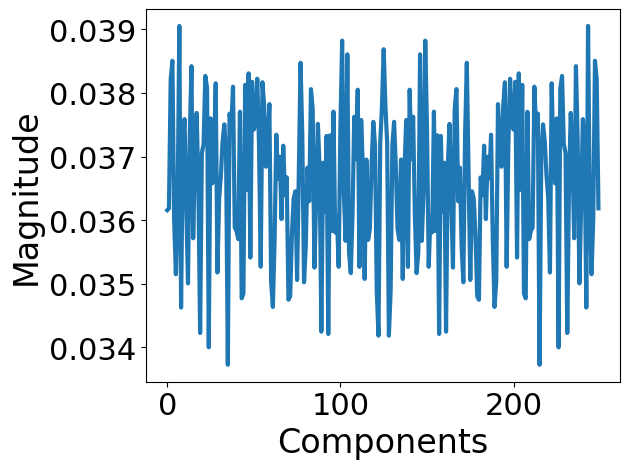}
        \caption*{Type-3}
    \end{minipage}

    \caption{Distribution of circulant components for different types of matrices. The required set of entries for the Toeplitz, block Toeplitz, and Hankel matrices was generated randomly from a $\mathcal{U}(0,1)$ distribution. The Kappa matrix K was generated using the relation $K_{ij}=K_{ji}=e^{(-0.5|i-j|)}\times sin(i+1)$ for the entries.}\label{fig:example_CD_components}
\end{figure}
    
\subsection{Cycles of a square matrix and its decomposition}\label{appendix:example_finding_cycles}
As an example,
\begin{align*}
A &=
\begin{bmatrix}
\color{green}{a_{00}} & a_{01} & \color{brown}{a_{02}} & \color{blue}{a_{03}} \\
\color{blue}{a_{10}} & \color{green}{a_{11}} & a_{12} & \color{brown}{a_{13}} \\
\color{brown}{a_{20}} & \color{blue}{a_{21}} & \color{green}{a_{22}} & a_{23} \\
a_{30} & \color{brown}{a_{31}} & \color{blue}{a_{32}} & \color{green}{a_{33}}
\end{bmatrix}
\\
&=\begin{bmatrix}
\color{green}{a_{00}} & 0 & 0& 0\\
0 & \color{green}{a_{11}} & 0 & 0 \\
0& 0 & \color{green}{a_{22}} & 0 \\
0& 0 & 0& \color{green}{a_{33}}
\end{bmatrix} + \begin{bmatrix}
0 & 0 & 0 & \color{blue}{a_{03}} \\
\color{blue}{a_{10}} & 0 & 0 & 0 \\
0 & \color{blue}{a_{21}} & 0& 0 \\
0 & 0 & \color{blue}{a_{32}} & 0
\end{bmatrix}+\begin{bmatrix}
0 & 0 & \color{brown}{a_{02}} & 0 \\
0 & 0 & 0 & \color{brown}{a_{13}} \\
\color{brown}{a_{20}} & 0 & 0& 0 \\
0& \color{brown}{a_{31}} & 0 & 0
\end{bmatrix}+\begin{bmatrix}
0 & a_{01} & 0& 0\\
0& 0 & a_{12} & 0 \\
0 & 0 & 0& a_{23} \\
a_{30} & 0 & 0& 0
\end{bmatrix}\\
&=\begin{bmatrix}
\color{green}{a_{00}} & 0 & 0& 0\\
0 & \color{green}{a_{11}} & 0 & 0 \\
0& 0 & \color{green}{a_{22}} & 0 \\
0& 0 & 0& \color{green}{a_{33}}
\end{bmatrix} C^{0}+\begin{bmatrix}
\color{blue}{a_{03}} & 0 & 0& 0\\
0 & \color{blue}{a_{10}} & 0 & 0 \\
0& 0 & \color{blue}{a_{21}} & 0 \\
0& 0 & 0& \color{blue}{a_{32}}
\end{bmatrix} C^{1}+\begin{bmatrix}
\color{brown}{a_{02}} & 0 & 0& 0\\
0 & \color{brown}{a_{13}} & 0 & 0 \\
0& 0 & \color{brown}{a_{20}} & 0 \\
0& 0 & 0& \color{brown}{a_{31}}
\end{bmatrix} C^{2}+\begin{bmatrix}
\color{black}{a_{01}} & 0 & 0& 0\\
0 & \color{black}{a_{12}} & 0 & 0 \\
0& 0 & \color{black}{a_{23}} & 0 \\
0& 0 & 0& \color{black}{a_{30}}
\end{bmatrix} C^{3}\\
&={\color{green}\Lambda_0}C^{0}+{\color{blue}\Lambda_1}C^{1}+{\color{brown}\Lambda_2}C^{2}+\Lambda_3C^{3}\\
\end{align*}
And hence \begin{align}
    \underline{A}&=\begin{bmatrix}
\color{green}{a_{00}} & \color{blue}{a_{03}} & \color{brown}{a_{02}}& a_{01}\\
\color{green}{a_{11}} & \color{blue}{a_{10}} & \color{brown}{a_{13}} & a_{12} \\
\color{green}{a_{22}}& \color{blue}{a_{21}} & \color{brown}{a_{20}} & a_{23} \\
\color{green}{a_{33}}& \color{blue}{a_{32}} & \color{brown}{a_{31}}& a_{30}
\end{bmatrix}\label{get_At1}
\end{align}

For $n \in \mathbb{N}$, $C^{0}=\mathbb{I}_{n}$, an $n\times n $ identity matrix has ones (non-zero entries) at the indices \{(0,0),(1,1),\dots,(n-1,n-1)\}. To obtain $C^{k}$, we need only the indices where the entries are non-zero, and the position of the non-zero entry in the $r^{th}$ row of $C^{k}$ is given by $(r-k) \mod n$. Similarly, we can obtain the diagonal matrix $\Lambda_{k}$, having its $r^{th}$ diagonal entry as the entry of $A$ indexed by [$r$, $(r-k) \mod n$], in $\sim n$ operations. Thus we construct $\underline{A}$ in $\sim n^2$ operations. By replacing steps 2 and 5 of \cref{direct_At_from_A} with $A_{t}=A$ and $col=(i-j)$ mod $n$ respectively, we can re-order $A$ to get $\underline{A}$ as in \cref{get_At1}.

As suggested in \cref{rightdecomp}, we may also need the decomposition $A=\sum_{i=0}^{n-1}C^{i}\Lambda_i$. In this case, \cref{direct_At_from_A} can be used as given to reorder $A$ as in \cref{get_At}. This following example can be referred for such a decomposition.

\begin{align*}
A &=
\begin{bmatrix}
\color{green}{a_{00}} & a_{01} & \color{brown}{a_{02}} & \color{blue}{a_{03}} \\
\color{blue}{a_{10}} & \color{green}{a_{11}} & a_{12} & \color{brown}{a_{13}} \\
\color{brown}{a_{20}} & \color{blue}{a_{21}} & \color{green}{a_{22}} & a_{23} \\
a_{30} & \color{brown}{a_{31}} & \color{blue}{a_{32}} & \color{green}{a_{33}}
\end{bmatrix}
\\
&=\begin{bmatrix}
\color{green}{a_{00}} & 0 & 0& 0\\
0 & \color{green}{a_{11}} & 0 & 0 \\
0& 0 & \color{green}{a_{22}} & 0 \\
0& 0 & 0& \color{green}{a_{33}}
\end{bmatrix} + \begin{bmatrix}
0 & 0 & 0 & \color{blue}{a_{03}} \\
\color{blue}{a_{10}} & 0 & 0 & 0 \\
0 & \color{blue}{a_{21}} & 0& 0 \\
0 & 0 & \color{blue}{a_{32}} & 0
\end{bmatrix}+\begin{bmatrix}
0 & 0 & \color{brown}{a_{02}} & 0 \\
0 & 0 & 0 & \color{brown}{a_{13}} \\
\color{brown}{a_{20}} & 0 & 0& 0 \\
0& \color{brown}{a_{31}} & 0 & 0
\end{bmatrix}+\begin{bmatrix}
0 & a_{01} & 0& 0\\
0& 0 & a_{12} & 0 \\
0 & 0 & 0& a_{23} \\
a_{30} & 0 & 0& 0
\end{bmatrix}\\
&=C^{0}\begin{bmatrix}
\color{green}{a_{00}} & 0 & 0& 0\\
0 & \color{green}{a_{11}} & 0 & 0 \\
0& 0 & \color{green}{a_{22}} & 0 \\
0& 0 & 0& \color{green}{a_{33}}
\end{bmatrix}+C^{1}\begin{bmatrix}
\color{blue}{a_{10}} & 0 & 0& 0\\
0 & \color{blue}{a_{21}} & 0 & 0 \\
0& 0 & \color{blue}{a_{32}} & 0 \\
0& 0 & 0& \color{blue}{a_{03}}
\end{bmatrix}+C^{2}\begin{bmatrix}
\color{brown}{a_{20}} & 0 & 0& 0\\
0 & \color{brown}{a_{31}} & 0 & 0 \\
0& 0 & \color{brown}{a_{02}} & 0 \\
0& 0 & 0& \color{brown}{a_{13}}
\end{bmatrix}+C^{3}\begin{bmatrix}
\color{black}{a_{30}} & 0 & 0& 0\\
0 & \color{black}{a_{01}} & 0 & 0 \\
0& 0 & \color{black}{a_{12}} & 0 \\
0& 0 & 0& \color{black}{a_{23}}
\end{bmatrix} \\
&=C^{0}{\color{green}\Lambda_0}+C^{1}{\color{blue}\Lambda_1}+C^{2}{\color{brown}\Lambda_2}+C^{3}\Lambda_3\\
\end{align*}

And hence \begin{align}
    \underline{A}&=\begin{bmatrix}
\color{green}{a_{00}} & \color{blue}{a_{10}} & \color{brown}{a_{20}}& a_{30}\\
\color{green}{a_{11}} & \color{blue}{a_{21}} & \color{brown}{a_{31}} & a_{01} \\
\color{green}{a_{22}}& \color{blue}{a_{32}} & \color{brown}{a_{02}} & a_{12} \\
\color{green}{a_{33}}& \color{blue}{a_{03}} & \color{brown}{a_{13}}& a_{23}
\end{bmatrix}\label{get_At}
\end{align}

\begin{algorithm}
\caption{Construct $\underline{A}$ from $A$ as in \cref{get_At}}\label{direct_At_from_A}
\begin{algorithmic}[1]
\State Initialize $\underline{A} \gets \text{zero matrix of size } n \times n$
\State $A_t \gets A^T$
\For{$j = 0$ to $n - 1$}
    \For{$i = 0$ to $n - 1$}
        \State $col \gets (i + j) \bmod n $
        \State $\underline{A}[i , j ] \gets A_t[i , col]$
    \EndFor
\EndFor
\end{algorithmic}
\end{algorithm}

\section{Evaluation of the variance of the norms of matrices with Haar singular vectors}
Let $S = \norm{AB}_F^2$.  Let $\alpha_1 = \sum_{i} D_1(i)^2$ and $\alpha_2 = \sum_{i} D_2(i)^2$ and $\beta_1 = \sum_{i} D_1(i)^4$ and $\beta_2 = \sum_{i} D_2(i)^4$.

\begin{lemma}
$\mathbb{E}[S^2] = \frac{1}{n(n+1)} (\beta_1 \beta_2 + \alpha_1^2 \alpha_2^2).$
\end{lemma}
\begin{proof}
    We have 
    \begin{align*}
        S^2 &= \sum_{i,j,k,l} D_1(i)^2 D_2(j)^2 D_1(k)^2 D_2(l)^2|Q_{i,j}|^2 |Q_{k,l}|^2 .
    \end{align*}
    The expectation is given by :
     \begin{align*}
        \mathbb{E}[S^2] &= \sum_{i,j,k,l}  D_1(i)^2 D_2(j)^2 D_1(k)^2 D_2(l)^2 \mathbb{E}[|Q_{i,j}|^2 |Q_{k,l}|^2], \\
        &= \sum_{(i,j) = (k,l)} D_1(i)^4 D_2(j)^4\mathbb{E}[|Q_{i,j}|^4] + \sum_{(i,j) \neq (k,l)}  D_1(i)^2 D_2(j)^2 D_1(k)^2 D_2(l)^2 \mathbb{E}[|Q_{i,j}|^2 |Q_{k,l}|^2] 
    \end{align*}

    Now using the fact that for Haar distributed matrices we have $\mathbb{E}[|Q_{i,j}|^4]  = \frac{2}{n(n+1)}$ and $\mathbb{E}[|Q_{i,j}|^2 |Q_{k,l}|^2]  = \frac{1}{n(n+1)}$. By substituting the above we get :

    \begin{align*}
         \mathbb{E}[S^2] &=  \sum_{(i,j) = (k,l)} D_1(i)^4 D_2(j)^4 \frac{2}{n (n+1)} + \sum_{(i,j) \neq (k,l)}  D_1(i)^2 D_2(j)^2 D_1(k)^2 D_2(l)^2 \frac{1}{n(n+1)}
    \end{align*}

    Let $\alpha_1 = \sum_{i} D_1(i)^2$ and $\alpha_2 = \sum_{i} D_2(i)^2$ and $\beta_1 = \sum_{i} D_1(i)^4$ and $\beta_2 = \sum_{i} D_2(i)^4$.

    Thus we have :
    \begin{align*}
         \mathbb{E}[S^2] &= \frac{1}{n(n+1)} (\beta_1 \beta_2 + \alpha_1^2 \alpha_2^2).
    \end{align*}
\end{proof}

Further, we can compute the variance of $S$ as :

\begin{lemma}
$\text{Var}(S) = \mathbb{E}[S^2] - \mathbb{E}[S]^2 = \frac{1}{n(n+1)} \left( \beta_1 \beta_2 - \frac{\alpha_1^2 \alpha_2^2}{n} \right)$.    
\end{lemma}
\begin{proof}
    By using the previous arguments we have :
    \begin{align*}
        \text{Var}(S) &= \mathbb{E}[S^2] - \mathbb{E}[S]^2 \\
        &= \frac{\beta_1\beta_2 + \alpha_1^2 \alpha_2^2}{n(n+1)} - \frac{\alpha_1^2 \alpha_2^2}{n^2} \\
        &= \frac{1}{n(n+1)} \left( \beta_1 \beta_2 - \frac{\alpha_1^2 \alpha_2^2}{n} \right)
    \end{align*}
\end{proof}

The expected Frobenius norm of the product matrix can be obtained using the Taylor series expansion. 

\begin{align*}
    f(S - \mathbb{E}[S] + \mathbb{E}[S]) &\approx f(\mathbb{E}[S]) + (S - \mathbb{E}[s])f'(\mathbb{E}[S]) +\frac{1}{2} (S - \mathbb{E}[S])^2 f''(\mathbb{E}[S])  
\end{align*}
Now using the expectation and applying the above to the function of a square-root, we get :

\begin{align*}
    \mathbb{E}[\sqrt{S}] \approx \sqrt{\mathbb{E}[S]} - \frac{1}{8 \mathbb{E}[S]^{3/2}}\text{Var}(S).
\end{align*}

Now using the variance of $S$,
\begin{align*}
    \mathbb{E}[\sqrt{S}] =  \sqrt{\mathbb{E}[S]} - \frac{n^{3/2}}{8 (\alpha_1 \alpha_2)^{3/2}}\left(  \frac{1}{n(n+1)} \left( \beta_1 \beta_2 - \frac{\alpha_1^2 \alpha_2^2}{n} \right) \right) + \epsilon.
\end{align*}

Further using the fact that $\mathbb{E}[S] = \sqrt{\frac{\alpha_1 \alpha_2}{n}}$, we have :
\begin{align*}
    \mathbb{E}[\sqrt{S}] &=  \sqrt{\mathbb{E}[S]}\left(1 - \frac{n^{2}}{8 (\alpha_1 \alpha_2)^{2}}\left(  \frac{1}{n(n+1)} \left( \beta_1 \beta_2 - \frac{\alpha_1^2 \alpha_2^2}{n} \right) \right) \right) + \epsilon. \\
    &= \sqrt{\mathbb{E}[S]}\left(1 - \frac{n}{8 (\alpha_1 \alpha_2)^{2}}\left(  \frac{1}{(n+1)} \left( \beta_1 \beta_2 - \frac{\alpha_1^2 \alpha_2^2}{n} \right) \right) \right) + \epsilon, \\
     &= \sqrt{\mathbb{E}[S]}\left(1 + \frac{1}{8(n+1)} - \frac{n\beta_1 \beta_2  }{8(n+1) (\alpha_1 \alpha_2)^{2}} \right) + \epsilon. 
\end{align*}

\section{Randomly generated matrices with a given distribution of singular values} \label{sec:matrix_generation}
The following algorithm was used to randomly generate matrices with a given distribution of singular values. The generated matrices have singular vectors that are  orthogonal matrices with a Haar distribution (uniform on $n$-sphere).

\begin{algorithm}[H]
\caption{Generate Matrices with Specific Singular Value Distributions} 
\begin{algorithmic}[1]
\Require List \texttt{S} with length \texttt{n}\Comment{List of singular values for a given distribution.}
\Ensure Matrix $Q_1 D Q_2^T$
\State $D \longleftarrow S$; \Comment{Create a diagonal matrix with $S$ as diagonal.}
\State Generate random matrices $M_1,M_2 \sim \mathcal{N}(0,1)^{\texttt{n} \times \texttt{n}}$
\State $Q_1 R_1 \longleftarrow M_1$, $Q_2R_2 \longleftarrow M_2$; \Comment{Compute QR decomposition.}
\State $Q_1 \longleftarrow Q_1.diag(sign(diag(R_1)))$; \Comment{Multiply a diagonal matrix with corresponding signs of $R$}
\State $Q_2 \longleftarrow Q_2.diag(sign(diag(R_2)))$ ; \Comment{Non-negative diagonal elements of $R_1$, $R_2$ for a Haar distribution of orthogonal matrices.}
\State Return: $Q_1 D Q_2^T$
\end{algorithmic}
\end{algorithm}

S[i]=$e^{\frac{-i}{10}}$, and S[i]=$\frac{n-i}{n}$ for $i=0,1,2,\dots,(n-1)$ were used for generating Type-1 and Type-3 matrices respectively. To generate Type-2 matrices, a Type-1 matrix $T$ is added with significant full-rank noise $\frac{0.5\|T\|_F}{\|U\|_F}U$, $U$ being a matrix with random entries from $\mathcal{U}$(0,1).

To create the query$(Q_i)$, key $(K_i)$, value $(V_i)$, and input matrices$(I_i)$, we used the GPT-Neo-1(3B) model with two input texts as follows :
\begin{itemize}
    \item ``In a world where distractions are endless and time seems to slip away faster than ever, choosing to live with intention, to act with kindness, and to remain grounded in gratitude becomes a radical act of self-awareness, a declaration that life is not merely a series of obligations or fleeting moments, but a profound journey filled with opportunities to grow, to connect deeply with others, and to contribute meaningfully to something greater than oneself, and though the path may be uncertain, filled with trials and unexpected detours, it is precisely in those uncertain spaces where courage is born, where resilience is forged, and where the quiet strength of the human spirit rises, reminding us that every small act of love, truth, and compassion truly matters."

    \item ``Success is not solely defined by wealth, recognition, or the accumulation of achievements, but rather by the quiet commitment to live with integrity, to treat others with respect, and to remain true to one’s values even when the world encourages compromise; it is found in the consistent effort to grow, to learn from mistakes, to listen deeply, and to act with empathy, even when it’s inconvenient or difficult, because real success is measured by the positive impact we have on others, the courage we show in the face of adversity, and the authenticity with which we move through life, knowing that our time is limited and our legacy is shaped not by what we gain, but by what we give, share, and stand for."
\end{itemize}
The programs used for generating the example matrices and their results determining the number of required components of a decomposition to achieve a given tolerance in relative error can be found in \cite{Codes}. This repository also contains the LLM matrices used. The Density-Functional-Theory matrices represent finite-element discretized DFT model of BCC Molybdenum supercell with a vacancy, after a projection into the desired eigenstates. These matrices are iterated in a self-consistent physics based approach to solve a non-linear eigenvalue problem representing energies of the material. A global DFT matrix, principal sub-matrices of which were used here in the experiments is also included in this repository.\\

\vspace{3mm}
\noindent \textbf{Author declarations}:

\vspace{3mm}
\noindent \textbf{Funding}: Murugesan Venkatapathi acknowledges the support of the Science and Engineering Research Board (SERB) grant CRG/2022/004178 in performing this research.

\vspace{3mm}
\noindent \textbf{Conflicts of interest}: The authors do not have any competing financial or non-financial interests to declare.

\bibliographystyle{siamplain}
\bibliography{reference.bib}
\end{document}